\documentclass[a4paper,12pt]{amsart} 
\usepackage{amssymb, amsmath, amsthm, color,graphicx, tikz}
\usepackage[colorlinks=true, breaklinks=true, linkcolor=black, citecolor=blue, urlcolor=red]{hyperref} 
\usepackage[english]{babel}

\numberwithin{equation}{section}
\newtheorem{theol}{Theorem}

\newtheorem{theorem}{Theorem}[section]

\newtheorem{lemma}[theorem]{Lemma}
\newtheorem{corollary}[theorem]{Corollary}
\newtheorem{proposition}[theorem]{Proposition}

\theoremstyle{definition} 
\newtheorem{definition}[theorem]{Definition}
\newtheorem{notation}[theorem]{Notation}
\newtheorem{remark}[theorem]{Remark}
\newtheorem{example}[theorem]{Example}

\newcommand{\act}{\curvearrowright}
\DeclareMathOperator{\ad}{ad}
\newcommand{\al}{\alpha}
\DeclareMathOperator{\Aut}{Aut}
\newcommand{\cC}{\mathcal C}

\newcommand{\fC}{\mathfrak C}

\DeclareMathOperator{\Coc}{Coc}
\newcommand{\cD}{\mathcal D}

\DeclareMathOperator{\End}{End}
\newcommand{\cF}{\mathcal F}
\DeclareMathOperator{\Fix}{Fix}
\newcommand{\fs}{\{0,1\}^{(\N)} }
\newcommand{\is}{\{0,1\}^{\N} }
\newcommand{\cAF}{\mathcal{AF}}
\DeclareMathOperator{\Gr}{Gr}
\newcommand{\ga}{\gamma}
\newcommand{\Ga}{\Gamma}
\DeclareMathOperator{\Hom}{Hom}
\DeclareMathOperator{\Homeo}{Homeo}
\newcommand{\cI}{\mathcal I}
\DeclareMathOperator{\id}{id}
\DeclareMathOperator{\Isom}{Isom}
\newcommand{\La}{\Lambda}
\DeclareMathOperator{\Leaf}{Leaf}
\newcommand{\N}{\mathbf{N}}
\newcommand{\cN}{\mathcal{N}}
\newcommand{\NCV}{N_{H(\fC)}(V)}
\DeclareMathOperator{\ob}{ob}
\newcommand{\ot}{\otimes}
\newcommand{\ov}{\overline}
\newcommand{\Q}{\mathbf{Q}}
\newcommand{\R}{\mathbf{R}}
\DeclareMathOperator{\Stab}{Stab}
\newcommand{\cSF}{\mathcal{SF}}
\newcommand{\SNQ}{\Stab_N(\Q_2)}
\DeclareMathOperator{\supp}{supp}
\newcommand{\fT}{\mathfrak T}
\newcommand{\ti}{\tilde}
\newcommand{\varep}{\varepsilon}
\newcommand{\Z}{\mathbf{Z}}

\begin{document}

\title[Classification of fraction groups I]{Classification of Thompson related groups arising from Jones technology I}
\thanks{
AB is supported by the Australian Research Council Grant DP200100067 and a University of New South Wales start-up grant.}
\author{Arnaud Brothier}
\address{Arnaud Brothier\\ School of Mathematics and Statistics, University of New South Wales, Sydney NSW 2052, Australia}
\email{arnaud.brothier@gmail.com\endgraf
\url{https://sites.google.com/site/arnaudbrothier/}}
\maketitle

\begin{abstract}
In the quest in constructing conformal field theories (CFT) Jones has discovered a beautiful and deep connection between CFT, Richard Thompson's groups and knot theory. 
This led to a powerful functorial framework for constructing actions of particular groups arising from categories such as Thompson's groups and braid groups. 
In particular, given a group and two of its endomorphisms one can construct a semidirect product where the largest Thompson's group $V$ is acting.
These semidirect products have remarkable diagrammatic descriptions which were previously used to provide new examples of groups having the Haagerup property. 
They naturally appear in certain field theories as being generated by local and global symmetries.
Moreover, these groups occur in a construction of Tanushevski and can be realised using Brin-Zappa-Szep's products with the technology of cloning systems of Witzel-Zaremsky.

We consider in this article the class of groups obtained in that way where one of the endomorphism is trivial leaving the case of two nontrivial endomorphisms to a second article. 
We provide an explicit description of all these groups as permutational restricted twisted wreath products where $V$ is the group acting and the twist depends on the endomorphism chosen.
We classify this class of groups up to isomorphisms and provide a thin description of their automorphism group thanks to an unexpected rigidity phenomena.
\end{abstract}

\hspace{10cm} \textit{\`A la m\'emoire de Vaughan Jones} 

\section*{Introduction}

The present paper studies a class of groups that are constructed via a recent framework due to Jones \cite{Jones17-Thompson}. 
This framework was discovered by accident in the land of quantum field theory as we will now explain, see the following survey for more details \cite{Brothier-19-survey}.
Conformal field theories (in short CFT) \`a la Haag-Kastler provide subfactors and conversely \textit{certain} subfactors provide CFT but the reconstruction is on a case by case basis and so far the most intriguing subfactors (with exotic representation theory uncaptured by groups and quantum groups) are not known to provide a CFT \cite{Longo-Rehren95,Evans_Kawahigashi_92_sf_cft,Jones-Morrison-Synder14,Bischoff17,Xu18-CFT}.
By using the planar algebra of a subfactor, Jones created a lattice model approximating the desired CFT.
This did not converge to a classical CFT but rather defined a discontinuous physical model particularly relevant at quantum phase transition where Richard Thompson's group $T$ plays the role of the spatial conformal group \cite{Jones16-Thompson,Jones18-Hamiltonian,Osborne-Stiegemann19, Stiegemann19-thesis,Brot-Stottmeister-M19,Brot-Stottmeister-Phys}.
By extracting the mathematical essence of this construction Jones found a wonderful machine for constructing actions of certain groups (e.g.~Thompson's groups and braid groups) called \textit{Jones' actions}.

Recall that Richard Thompson's group $F$ is the group of piecewise linear homeomorphisms of the unit interval having finitely many breakpoints all at dyadic rationals and having slopes powers of 2.
There are two other groups: Thompson's group $T$ containing $F$ and translations by dyadic rationals acting by homeomorphisms on the unit circle, and Thompson's group $V$ containing $T$ and discontinuous exchanges of subintervals of $[0,1]$ that act by homeomorphisms on Cantor space, see \cite{Cannon-Floyd-Parry96} for details. 
We will be focusing on Thompson's group $V$ in this paper but this study stays relevant for the smaller Thompson's groups $F$ and $T$.

Thompson's groups $F,T$ and $V$ are countable discrete groups that are notoriously difficult to understand. 
However, they admit simple descriptions in terms of fraction groups of categories, which is the description used by Jones' technology.
It has been known for a long time that a category having nice cancellation properties (a category with a left calculus of fractions) provides a groupoid by formally inverting its morphisms and in particular groups by fixing a common source to all morphisms: such groups are called \textit{fraction groups} or \textit{groups of fractions} \cite{GabrielZisman67}.
In particular, the fraction group of the category of binary forests (specialised at the object $1$) is (isomorphic to) Thompson's group $F$: elements of $F$ are described by (an equivalence class of) a pair of trees having the same number of leaves.
The larger groups $T$ and $V$ have similar descriptions but with trees having their leaves decorated by natural numbers corresponding to permutations, see \cite{Brown87, Cannon-Floyd-Parry96}. 
Jones made the crucial observation that, if we consider any functor starting from such a category (e.g.~the category of binary forests), then we obtain an action of the fraction group (e.g.~of $F$, $T$ or $V$).
This led to a powerful and practical framework for constructing group actions and put in evidence unsuspected connections between different fields of research. 
We present some of those connections.

\subsection*{Application of Jones' technology}
$\bullet$ Functors from the category of binary forests into the category of Conway tangles provide ways to construct knots and links via Thompson's group $F$. 
Jones proved that they can all be constructed in that way concluding that `` Thompson's group $F$ is as good as the braid group for producing knots and links.''
He discovered the so-called Jones subgroup $\Vec F\subset F$ or oriented Thompson's group which has remarkable properties such as being isomorphic to the ternary Brown-Thompson's group $F_3$ as proved by Golan and Sapir \cite{Jones17-Thompson, Golan-Sapir17}. 
Aiello proved the beautiful result that any \textit{oriented} links can be produced using $\Vec F$ and the construction of Jones \cite{Aiello20}.
This led to a profound connection between Thompson's group $F$ and knot theory, providing new invariants for knots and linking for the second time, after the Jones polynomials, subfactor theory with knot theory \cite{Jones_polynome_vna,Jones19-thomp-knot}.

$\bullet$ Using functors ending in the category of Hilbert spaces we obtain unitary representations and have access to matrix coefficients that are explicitly computable via an algorithm depending on the functor chosen \cite{Jones17-Thompson,Jones19Irred,ABC19}.
Jones and the author used this approach for constructing new families of unitary representations and matrix coefficients for Thompson's groups $F,T,$ and $V$ \cite{Brot-Jones18-2,Brot-Jones18-1}.
In particular, a pair $(a,b)$ of two bounded operators acting on a Hilbert space $H$ and satisfying the \textit{Pythagorean relation} $a^*a+b^*b=\id_H$ provides a unitary representation of the largest Thompson's group $V$. 
This produces many new explicit examples of positive definite maps and a deep connection between Thompson's groups $F,T,$ and $V$ and the Cuntz algebra complementing previous works of Nekrashevych \cite{Nekrashevych04}.
This also led to new effortless proofs establishing analytic properties of Thompson's groups: any intermediate group between the derived group of $F$ and $V$ is not a Kazhdan's group and $T$ has the Haagerup property.
Although, these are not new results it demonstrates how powerful is Jones' technology.
Recall that Reznikoff proved that Thompson's group $T$ was not a Kazhdan's group which was also a consequence of the combined work of Ghys-Sergiescu and Navas \cite{Reznikoff01,Ghys-Sergiescu87,Navas02}. 
Moreover, Farley proved the stronger result that $V$ has the Haagerup propery implying that all of its subgroup has this property and in particular are not Kazhdan's groups \cite{Farley03-H}.

$\bullet$ A functor $\Phi: \cC\to \Gr$ ending in the category of groups gives an action $G_\cC\act K$ of the fraction group $G_\cC$ associated to the source category $\cC$ on a limit group $K$.
The author made the key observation that the semidirect product $K\rtimes G_\cC$ is again a fraction group and provided an explicit description of the category inducing this group \cite{Brothier19WP}.
Jones framework can be then reapplied to this semidirect product $K\rtimes G_\cC$ for constructing unitary representations and computing matrix coefficients.
With this method the author provided a new large class of semidirect products with the Haagerup property. 
In particular, all wreath products $\oplus_{\Q_2}\Ga\rtimes V$, with $\Ga$ any discrete group having the Haagerup property and $V\act \Q_2$ the classical action of Thompson's group $V$ on the dyadic rationals, have the Haagerup property \cite{Brothier19WP}.
This provided, using a result of Cornulier, the first examples of finitely presented wreath products having the Haagerup property for a nontrivial reason (i.e.~the group acting is nonamenable and the base space is infinite) \cite{Cornulier06}.
This class of wreath products/fraction groups is contained in the class of groups that we are studying in this article.

$\bullet$ Groups constructed as above naturally appear in certain field theories.
Stottmeister and the author constructed physical models in the line of Jones' work but also using previous constructions appearing in loop quantum gravity.
In those physical models, keeping the notation from above, the physical space is approximated by $\Q_2$ with local symmetries being the group $\Ga$ and Thompson's group $T$ playing the role of spatial symmetries acting by local scale transformations and rotations on $\Q_2.$
Together they generate a group of the form $K\rtimes T$ admitting a fraction group description where $K \subset \prod_{\Q_2}\Ga$ plays the role of a discrete loop group \cite{Brot-Stottmeister-M19,Brot-Stottmeister-Phys}.

\subsection*{Construction of groups}
We now present a specific class of fractions groups for which Jones' technology can be efficiently applied. 
Groups studied in this article and its sequel belong to this class.
Consider a functor $\Phi:\cF\to \Gr$ from the category of forests to the category of groups.
Jones' technology provides a Jones' action $F\act K$ of Thompson's group $F$ on a group $K$.
Consider the semidrect product $G_0:=K\rtimes F$. 
The author made the trivial but key observation that $G_0$ is itself a fraction group.
Therefore, we can reapply Jones' technology to $G_0$ as explained in the third bullet point above.
As an initial study we consider the simplest class of functors $\Phi$: when $\Phi:\cF\to\Gr$ is \textit{covariant} and \textit{monoidal}. 
By a universal property of $\cF$ the class of covariant monoidal functors $\Phi:\cF\to\Gr$ is in bijection with all triples $(\Ga,\al_0,\al_1)$ where $\Ga$ is a group and $\al_0,\al_1\in\End(\Ga)$ are endomorphisms of $\Ga$.
The monoidal assumption implies that we can extend the Jones action $F\act K$ to the largest Thompson's group $V$. 
Therefore, we obtain a larger group $G:=K\rtimes V$. 
The group $G$ is still a fraction groups with remarkable diagrammatic descriptions.

We are studying in this article and its sequel the class of such groups $G=K\rtimes V$.
More precisely, we are studying in the first article the class of groups $G$ obtained from triples $(\Ga,\al_0,\varep_\Ga)$ where $\varep_\Ga:\Ga\to \Ga, g\mapsto e_\Ga$ is the trivial endomorphism sending all elements of $\Ga$ to the neutral element $e_\Ga.$
These groups can be described as restricted permutational wreath products $\oplus_{\Q_2}\Ga\rtimes V$ constructed from the action $V\act \Q_2$ and are the ones appearing in the third bullet point.
More precisely, if $\Ga$ has the Haagerup property as a discrete group and $\al_0$ is injective, then $G$ has the Haagerup property \cite{Brothier19WP}.
In the second paper we will be considering fraction groups obtained from triples of the form $(\Ga,\al_0,\al_1)$ with $\al_0,\al_1$ automorphisms producing a rather different class of groups: the group $K$ is a \textit{discrete loop group} $L\Ga$, see \cite{Brothier20-2}.
The groups obtained in the second article are the ones appearing in the physical models studied in \cite{Brot-Stottmeister-M19,Brot-Stottmeister-Phys}. 

\subsection*{Motivations}
The motivation of this work is multifold. 
%
As remarked in the last bullet point the groups $K\rtimes T$ constructed from certain covariant monoidal functors $\Phi:\cF\to\Gr$ appear in physical models. It is then important to obtain information about these groups as they describe symmetries. Moreover, knowing if the groups $G$ are pairwise isomorphic or not becomes an important question to compare the various models. 

Perhaps, the most natural motivation is to understand what kind of groups can be obtained from functors $\Phi:\cF\to \Gr.$ 
Can we simply describe these groups? Are they often isomorphic to each other? What kind of properties do they satisfy? What are their internal symmetries? Do they share some of the exceptional behavior among groups like Thompson's groups?
All of this questions are important.
Indeed, to each group in this class we can apply efficiently Jones' technology and construct explicit unitary representations for which we can often compute their matrix coefficients. 
This can lead to establish that large classes of groups have the Haagerup property or are not Kazhdan's groups for instance as explained above. One then wants to know to what kind of groups this results apply to justifying the initial questions of this paragraph.
Explicit descriptions of these groups can provide new information in group theory on permanence of certain properties. This was the case for the Haagerup property and wreath products as mentioned above.

Another motivation is to construct Thompson-like groups and to provide new tools and approach to other pre-existing works. Conversely, we are motivated by adapting and using other known techniques and results to our class of groups.
There are a number of generalisations of Thompson's groups $F,T$, and $V$.
Cloning systems of Witzel-Zaremsky (using Brin-Zappa-Szep's products) and Tanushevski's construction produce two large classes of Thompson-like groups \cite{Witzel-Zaremsky18,Zaremsky18-clone,Tanushevski16,Tanushevski17}. 
The author discovered that these two classes of groups have large intersections with the class of groups obtained from Jones' technology and \textit{covariant} functors $\Phi:\cF\to\Gr$. 
This is a wonderful coincidence as the technologies and questions studied are very different. Hence, each formalism, results proved, and techniques developed will be profitable to each community.
We have extensively compare these technologies and how they connect to each other in a long appendix in the sequel of this article. 
Hence, we will be brief here but invite the interested reader to consult \cite[Appendix]{Brothier20-2}.
Interestingly and fortunately, results proved by each research group are rather disjoint.
In particular, all the results presented in this present article, its sequel and the one concerning the Haagerup property are new \cite{Brothier20-2,Brothier19WP}. 
Therefore, they provide new results in cloning systems and in Tanushevski's framework.
Even the description of this groups as wreath products or semidirect products of discrete loop groups with Thompson's groups $F, T,$ and $V$ seem to be new which should clearly help understanding these groups.
Conversely, cloning system technology connects our work with a conjecture of Lehnert (later modified via a theorem of Bleak-Matucci-Neunh\'offer) concerning co-context-free groups \cite{Lehnert08-thesis, BleakMatucciNeunhoffer16,BZFGHM18}. 
Tanushevski and Witzel-Zaremsky prove that a large class of our fraction groups have exceptional finiteness properties.
Moreover, Tanushevski provides extensive descriptions of the normal subgroups of the $K\rtimes F$ appearing in this article.
Finally, using cloning systems Ishida proves that many of our fraction groups can be (left- or bi-) ordered \cite{Ishida17}.

\subsection*{Other generalisations of Thompson's groups $F,T$, and $V$.}
There are a number of other generalisations of Thompson's groups. 
We present here a few of them and compare them with the class of groups constructed using Jones' technology. 
This explains where sit our fraction groups compared with other classical constructions of Thompson-like groups.
Elements of Thompson's group $F$ can be diagrammatically described by pairs of (rooted, finite, ordered) binary trees with the same number of leaves.
If we add cyclic permutations or any permutations of the leaves we obtain Thompson's groups $T$ and $V$, respectively \cite{Cannon-Floyd-Parry96}. 
Note that an element of $V$ can be represented by two trees $t,s$ so that each leaf of $t$ is linked to a unique leaf of $s$ by a curve. Linking leaves of $t$ to $s$ corresponds to a permutation.
By replacing the (possibly intersecting) curves by braids we obtain the braided Thompson's group independently discovered by Brin and Dehornoy \cite{Brin07-BraidedThompson,Dehornoy06}.
Now, instead of having binary trees, we may consider $n$-ary trees obtaining new groups $F_n,T_n,V_n$ for $n\geq 2.$
Moreover, we may replace trees by forests with a fixed number $r$ of roots. 
This produces the so-called Higman-Thompson's groups $V_{n,r}$ and Brown-Thompson's groups $F_{n,r},T_{n,r}$ \cite{Higman74,Brown87}. 
A pair of trees can be associated to a pair of partitions of the unit interval and a transformation of $[0,1]$. 
It is then possible to define $d$-dimensional analogues of Thompson's group $V$ by considering partitions of $[0,1]^d$ and obtaining the Brin-Thompson's groups $dV$ for $d\geq 1$ \cite{Brin04-nV}.
The $n$-ary versions $F_n,T_n,V_n$ for $n\geq 2$ of Thompson's groups $F,T,$ and $V$ are acting on the boundary of the infinite $n$-ary tree $\mathcal T_n$ which is the Cantor space $\{0,\cdots,n-1\}^\N$ equal to all infinite words in the alphabet $\{0,\cdots,n-1\}.$
The action consists in changing finitely many prefix of infinite words. 
One can enlarge these groups by acting on the remaining infinite part of words. This can be done by fixing a \textit{self-similar} subgroup $H$ of the automorphism group of the infinite $n$-ary tree $\mathcal T_n$.
We then obtain a so-called Rover-Nekrashevych's group $V_n(H)$ which is a subgroup of the almost automorphism group of $\mathcal T_n$ \cite{Rover99, Nekrashevych05}.
Hughes introduced the notion of finite similarity structure for a compact ultrametic space and locally finitely determined group of local similarities, also known as finite similarity structure (FSS) groups \cite{Hughes09}.
This generalises the key example of $V_n$ acting on $\{0,\cdots,n-1\}^\N$ and the groups of Rover-Nekrashevych when the self-similar group $H$ is finite.
There are other important classes of Thompson-like groups that have remarkable diagrammatic descriptions but no longer using trees or forests: for instance the \textit{diagram groups} of Guba-Sapir and their extension as \textit{picture groups} by Farley \cite{Guba-Sapir97,Farley05}.

We now compare with the class $\mathcal{CMF}$ of groups $G=K\rtimes V$ constructed using Jones' technology via covariant monoidal functors $\Phi:\cF\to\Gr$ or, equivalently, triples $(\Ga,\al_0,\al_1)$ with $\Ga$ a group and $\al_0,\al_1\in\End(\Ga).$
It is the class of groups considered in this article and its sequel.
Elements of $G$ can be represented by a pair of binary trees, with permutations of their leaves and moreover decoration of the leaves by elements of $\Ga.$ 
One can describe the elements of the $T$ and $F$ versions of $G$ by considering exclusively cyclic permutations or no permutations at all, respectively.
We obtain that for these groups $n=2,r=1,d=1$ for the notations of above: we are considering pairs of \textit{binary} forests with $r=1$ root (i.e.~trees) and in dimension $d=1$. 
We don't have braids either. 
Moreover, the group $\Ga$ has no interaction with the tree structure of $\mathcal T_2$ nor with Cantor space $\{0,1\}^\N$. 
Hence, as far as the author knows there are no large intersection between the class $\mathcal{CMF}$ and all the class of groups presented except for the groups constructed by Witzel-Zaremsky and Tanushevski as explained earlier.
Moreover, we do not see any similarity in the construction producing the class $\mathcal{CMF}$ and the other classes of groups. 
Here is an argument tending to show that the class of groups considered are rather different. 
Farley (resp. Hugues) proved the equivalent statement that all the picture groups (resp. all FSS groups) admits a proper isometric action on a CAT(0) cubical complex implying the Haagerup property \cite{Farley05,Hughes09}.
Our class of groups contain all restricted permutational wreath products $\oplus_{\Q_2}\Ga\rtimes V$ of all groups $\Ga$ deduced from the action $V\act \Q_2$ of $V$ acting on the dyadic rational of the unit interval. In particular, if $\Ga$ does not have the Haagerup property, then this group is not a FSS group nor a picture group. 

\subsection*{Presentation of the main results.}
Consider a triple $(\Ga,\al_0,\varep_\Ga)$ where $\varep_\Ga$ is the trivial endomorphism. Since the second endomorphism is trivial we may write $\al=\al_0$ and consider the pair $(\Ga,\al)$ to express the data of the triple $(\Ga,\al,\varep_\Ga).$
To this pair $(\Ga,\al)$ is associated a unique covariant monoidal functor $\Phi:\cF\to\Gr$ satisfying that $\Phi(n)=\Ga^n$ and $\Phi(Y):\Ga\to \Ga^2, g\mapsto (\al(g),e_\Ga)$ for all $n\geq 1.$
This provides a limit group $K=\varinjlim_{t\in\fT} \Ga_t$ and a Jones' action $\pi_\Phi:V\act K$. 
We are considering the semidirect product or fraction group $G=K\rtimes V$.
A general argument on a larger class of fraction groups proves that $K\subset G$ is a characteristic subgroup. 
Using an inductive limit trick argument we show that $\al$ can always be assumed to be an automorphism, see Section \ref{sec:End}.
Starting with $\al$ an automorphism, we obtain that $G=K\rtimes V$ is isomorphic to a restricted permutational twisted wreath product $\Ga\wr_{\Q_2}V=\oplus_{\Q_2}\Ga\rtimes V$. 
This wreath product is obtained from the classical action $V\act \Q_2$ of $V$ on the dyadic rationals $\Q_2$ of the unit interval. 
The twist is described using slopes of elements of $V$ and the automorphism $\al$, see Section \ref{sec:description}.
This provides a second powerful description (after the tree-like diagrammatic description) of these fraction groups.
Using the wreath product description and the fact that $\oplus_{\Q_2}\Ga\subset \Ga\wr_{\Q_2}V$ is characteristic we prove the surprising fact that all isomorphisms $\theta:G\to \ti G$ between such fraction groups $G=\Ga\wr_{\Q_2}V$ and $\ti G=\ti\Ga\wr_{\Q_2}V$ are \textit{spatial} in the following sense.\\
{\bf Definition.}
An isomorphism $\theta:G\to\ti G$ is called \textit{spatial} if the following hold.
There exists an isomorphism $\kappa:K\to \ti K$, a map $c:V\to \ti K$ and a homeomorphism $\varphi$ of Cantor space which restricts to a bijection of $\Q_2$ satisfying that
$$\theta(av) = \kappa(a)\cdot c_v \cdot \ad_\varphi(v) \text{ for all } a\in K, v\in V$$
where $\ad_\varphi(v):=\varphi v \varphi^{-1}, \ v\in V$ and moreover
$$\supp(\kappa(a)) = \varphi(\supp(a)) \text{ for all } a\in K$$
where $\supp(a)$ denotes the support of $a\in K$.

The fact that all isomorphisms are spatial implies the following classfication theorem:

\begin{theol}[Theorem \ref{th:isomGalpha}]\label{theol:A}
Consider two groups with automorphisms $(\Ga,\al\in\Aut(\Ga))$ and $(\ti\Ga,\ti\al\in\Aut(\ti\Ga))$ and their associated fraction groups $G:= K\rtimes V$ and $\ti G:=\ti K\rtimes V$.

The groups $G$ and $\ti G$ are isomorphic if and only if there exists $\beta\in\Isom(\Ga,\ti\Ga)$ and $h\in\ti\Ga$ such that $\ti\al = \ad(h)\circ\beta \al\beta^{-1}.$
\end{theol}
In particular, we obtain that two nonisomorphic groups $\Ga,\ti\Ga$ provide two nonisomorphic fraction groups whatever are the choices of automorphisms $\al$ and $\ti\al$.
Note that this theorem provides a classification of a particular class of twisted permutational restricted wreath products $\Ga\wr_{\Q_2}V$.

{\bf Comparison between general results on wreath products and our results.}
The above result is surprising as it cannot be expected for arbitrary classes of wreath products. 
It is known that the standard restricted wreath product $\Ga\wr\La:=\oplus_\La\Ga\rtimes \La$ remembers the group $\Ga$ (if $\Ga\wr\La\simeq \ti\Ga\wr \ti\La$, then $\Ga\simeq\ti\Ga$) and moreover the subgroup $\oplus_\La\Ga\subset \Ga\wr\La$ is characteristic (except for very few cases) but isomorphisms between two such groups are not necessarily spatial in the sense above of \cite{Neumann64}.
In Remark \ref{rem:AWP}, for each group $\Ga$ nontrivial, we provide a simple construction of a non-spatial automorphism of the standard restricted wreath product $G=\Ga^2\wr\Z$.
There exist some partial results regarding isomorphisms between particular permutational and twisted permutational wreath products extending Neumann's work \cite{Bodnarchuk94}.
Although, results similar to ours concerning properties of isomorphisms in non-standard wreath products typically hold for finite groups and with assumptions on $\Ga$ such as being indecomposable (i.e.~it is not the direct sum of two nontrivial subgroups) even when the action $\La\act X$ is transitive, see \cite{Gross88}.
It is remarkable that for our results no assumptions are required for the group $\Ga$.
In Remark \ref{rem:AWP}, we provide examples of permutational twisted wreath products $\Ga\wr_X B$ and $\ti\Ga\wr_{\ti X} B$ that are isomorphic with $B\act X, B\act \ti X$ transitive but such that $\Ga$ and $\ti\Ga$ are not isomorphic. 

The fact that isomorphisms behave well with the wreath product structure gives us hope to understand in detail all isomorphisms between such fraction groups.
We indeed succeeded in decomposing any automorphism of a fixed wreath product in our class into the product of four \textit{elementary} ones. 
Moreover, we could describe the structure of the automorphism group via an explicit semidirect product.
We restricted this study to \textit{untwisted} wreath products $G=K\rtimes V \simeq \oplus_{\Q_2}\Ga\rtimes V$ corresponding to pairs $(\Ga,\id_\Ga)$ where $\id_\Ga$ is the identity automorphism of $\Ga.$
The fraction group $G$ is then isomorphic to a permutational wreath product with a trivial twist justifying the terminology.

Before stating this main result, we introduce some notations:
$Z\Ga$ is the centre of $\Ga$; $N(G)/Z\Ga$ is the group of maps $f:\Q_2\to\Ga$ normalising $G$ mod out by constant maps valued in the centre of $\Ga$; and $\SNQ$ is the group of homeomorphisms of Cantor space normalising $V$ and stabilising a copy of $\Q_2$ inside Cantor space.
\begin{theol}[Theorem \ref{theo:isomB}]\label{theol:B}
Let $\Ga$ be a group and $G:=K\rtimes V\simeq \oplus_{\Q_2}\Ga\rtimes V$ be the fraction group associated to the pair $(\Ga,\id_\Ga)$.
There exists a surjective morphism from the semidirect product 
$$\left( Z\Ga\times N(G)/Z\Ga\right) \rtimes \left(\SNQ\times \Aut(\Ga)\right)$$
onto the automorphism group of $G$ whose kernel is the set $\{(\bar g,\ad(g)^{-1}):\ g\in\Ga\}$ where $\bar g\in N(G)/Z\Ga$ is the equivalence class of the constant map equal to $g$ and $\ad(g)\in\Aut(\Ga)$ is the inner automorphism of $\Ga$ associated to $g$.
\end{theol}

There exists general results regarding the structure of the automorphism group of a wreath product \cite{Houghton63,Hassanabi78}. Although, since in general automorphisms are not spatial those results are much coarser than ours showing again how special are the class of wreath products arising from Jones' technology. Results similar to ours are known to hold when additional assumptions are made on $\Ga$ such as being finite, nonabelian and indecomposable, see \cite{Gross88}.

We refer the reader to Section \ref{sec:semidirect} for the definition of the action 
$$ \left(\SNQ\times \Aut(\Ga)\right) \act \left( Z\Ga\times N(G)/Z\Ga\right)$$
involved in the semidirect product of Theorem B.
The actions of $N(G),\SNQ$ and $\Aut(\Ga)$ on $G$ are not surprising: the first acts by conjugation, the second acts by shifting indices on $K=\oplus_{\Q_2}\Ga$ and by conjugation on $V$ and the third acts diagonally on $K=\oplus_{\Q_2}\Ga$ leaving $V$ invariant. 
However, the action of $Z\Ga$ on $G$ was unexpected. It is built using the slopes of elements of $V$ and the dyadic valuation.
More precisely, the map 
$$\ell:V\to \prod_{\Q_2}\Ga, \ell_v(x)=\log_2(v'(v^{-1}x)),\ v\in V, x\in\Q_2$$ satisfies the cocycle identity $\ell_{vw} = \ell_v + \ell_w^v, v,w\in V$ implying that given any $\zeta\in Z\Ga$ the formula 
$$av\mapsto a\cdot \zeta^{\ell_v}\cdot v,\ a\in \prod_{\Q_2}\Ga, v\in V$$ 
defines an automorphism of the \textit{unrestricted} wreath product $\prod_{\Q_2}\Ga\rtimes V.$
However, $\ell_v,v\in V$ is not finitely supported in general implying that our automorphism of above does not restrict to an automorphism of the \textit{restricted} wreath product.
By perturbating $\ell_v$, we can define a finitely supported map by considering $p_v=\ell_v + \nu-\nu^v$ where $\nu:\Q_2\to\Z$ is the dyadic valuation.
This latter map does satisfy the cocycle identity of above, is finitely supported and nontrivial when $v$ is.
Therefore, the formula
$$E_\zeta:av\mapsto a\cdot \zeta^{p_v}\cdot v,\ a\in \oplus_{\Q_2}\Ga, v\in V$$ 
defines an automorphism of $G$ for any $\zeta\in Z\Ga$ and in fact a faithful action $E:Z\Ga\act G.$

{\bf Rubin's faithful classes.}
As pointed out by an anonymous referee of this article the fact that all isomorphisms are spatial together with the thin description of automorphisms of each fraction group $G$ suggest that we are in a situation similar to where Rubin's technology applies \cite{Rubin89}.
A class $C$ of pairs $(X,G)$ where $X$ is a topological space and $G$ is a group acting by homeomorphisms on $X$ is called \textit{faithful}, if given any isomorphism $\theta:G\to \ti G$, then there exists a homeomorphism $f:X\to \ti X$ satisfying that $\theta$ is implemented by conjugation by $f$.
Rubin provided a number of examples of faithful classes. 
In particular, the class equal to the single pair $(V,\fC)$ where $V$ is the largest Thompson's group and $\fC$ is Cantor space is faithful; an essential fact used in this article.
It would be very interesting to know if for our class (or a certain subclass) of fraction groups we can find for each group $G$ a suitable topological space $X_G$ on which $G$ acts by homeomorphisms so that the class of pairs $(X_G,G)$ is faithful in Rubin's sense.

\subsection*{Further comments on the results and proofs.}

Consider the hypothesis of Theorem \ref{theol:A} and an isomorphism $\theta:G\to\ti G.$
We start by proving that $\theta(K)=\ti K$ using that $K\subset G$ is the unique maximal normal subgroup of $G$ satisfying a certain decomposability property.
This implies the following decomposition 
$$\theta(av) = \kappa(a)\cdot c_v \cdot \phi(v), \ a\in K,v\in V$$ 
with $\kappa:K\to\ti K$ an isomorphism, $c:V\to \ti K, v\mapsto c_v$ a map satisfying a cocycle condition and $\phi$ an automorphism of $V$.
Using a reconstruction theorem of Rubin we obtain the crucial fact: $\phi(v) = \varphi v\varphi^{-1}, v\in V$ for a (unique) homeomorphism $\varphi$ of Cantor space \cite{Rubin96}.
We then prove the surprising and unexpected fact that $\supp(\kappa(a)) = \varphi(\supp(a)), a\in K$ where $\supp(a):=\{x\in\Q_2:\ a(x)\neq e\}$ is the support of $a$. 
Note that this forces $\varphi$ to stabilise the dyadic rationals inside Cantor space, i.e.~$\varphi\in\SNQ.$
Moreover, this implies that $\kappa:\oplus_{\Q_2}\Ga\to \oplus_{\Q_2}\ti\Ga$ is a direct product of isomorphisms: $\kappa=\prod_{x\in\Q_2}\kappa_x \in \prod_{x\in\Q_2}\Isom(\Ga,\ti\Ga)$.

To prove the relation between $\al$ and $\ti\al$ we use the formula
$$(vav^{-1})(x) = \al^n(a(v^{-1}x)), \ a\in \oplus_{\Q_2}\Ga, v\in V, x\in \Q_2$$ where $2^n$ is the slope of $v$ at the point $v^{-1}x$.
We prove and use the fact that the slope of $\varphi v \varphi^{-1}$ at $\varphi(x)$ is equal to the slope of $v$ at $x$ when $x\in\Q_2,vx=x, v\in V$ and $\varphi\in\SNQ.$
Note that this equality of slopes will no longer be true in general if we were not assuming that $\varphi$ stabilises the dyadic rationals $\Q_2$, see Remark \ref{rem:Shayo}.
We will see in the second article that the situation is more complex when one has to work with \textit{all} automorphisms of $V$ and not only those coming from $\SNQ.$

We now consider the automorphism group of $G=K\rtimes V$ that we study in the \textit{untwisted} case. 
Hence, $G$ is a permutational wreath product of the form $\oplus_{\Q_2}\Ga\rtimes V$ for some group $\Ga.$
Using the fact that $K\subset G$ is a characteristic subgroup and the result concerning the support mentionned earlier we obtain that any $\theta\in\Aut(G)$ can be decomposed as 
$$\theta(av)=\kappa(a)\cdot c_v\cdot \ad_\varphi(v), a\in K,v\in V$$ such that $\kappa(a)(\varphi(x)) = \kappa_x(a(x)), x\in \Q_2$ for some $\kappa_x\in\Aut(\Ga)$ and a fixed $\varphi\in\SNQ.$
Up to a composition with $av\mapsto a^\varphi\cdot \ad_\varphi(v)$ we can assume that $\varphi$ is trivial.
Moreover, up to a composition with elements of the normaliser $N(G)$ and automorphisms of $\Aut(\Ga)$ acting on the wreath product we can assume that $\kappa$ is the identity.
It remains an automorphism of the form: $av\mapsto a\cdot c_v\cdot v$ which forces $c_v(x)$ to be in the centre of $\Ga$ for all $v\in V, x\in\Q_2.$
We obtain a map $v\in V\mapsto c_v\in \oplus_{\Q_2}Z\Ga$ satisfying the cocycle identity $c_{vw}=c_v \cdot c_w^v, v,w\in V$.
To finish the proof of the theorem we classify all such maps but valued in the product (not the direct sum) of the $Z\Ga$ that is:
$$\{d:V\to \prod_{\Q_2}Z\Ga:\ d_{vw}=d_v\cdot d_w^v, \ \forall v,w\in V\}.$$
These maps form an abelian group for the pointwise product and moreover any cocycle $d$ can be decomposed as a product of two as follows: $$d_v(x)=\zeta^{\log_2(v'(v^{-1}x))} \cdot f(x) f(v^{-1}x), v\in V, x\in\Q_2$$ for a pair $(\zeta,f)\in Z\Ga\times \prod_{\Q_2}Z\Ga$ that is unique up to multiplying $f$ by a constant map.
Hence, this later group of cocycles is isomorphic to $Z\Ga\times \prod_{\Q_2}Z\Ga/Z\Ga.$
By considering $p_v:=\log_2(v'(v^{-1}\cdot))+\nu-\nu^v, v\in V$ rather than $\log_2(v'(v^{-1}\cdot))$ we decompose the automorphism $av\mapsto a\cdot c_v\cdot v$ into a product of $\ad(f)$ with $f\in N(G)\cap \prod_{\Q_2}Z\Ga$ and an automorphism: $$E_\zeta:av\mapsto a\cdot \zeta^{p_v}\cdot v, a\in K, v\in V$$
with $\zeta\in Z\Ga$.
This achieves the proof that every automorphism of $G$ is the product of four kinds of elementary automorphisms as previously described.

We obtain that $\Aut(G)$ is generated by the copies of the groups $Z\Ga, N(G), \SNQ$ and $\Aut(\Ga)$. 
To avoid confusions we write here $D(Z\Ga)\subset N(G)$ for the subgroup of constant maps from $\Q_2$ to $Z\Ga.$
It is rather easy to see that $Z\Ga,N(G)/D(Z\Ga), \SNQ,\Aut(\Ga)$ sit faithfully inside $\Aut(G).$ 
We want to understand how those subgroups interact with each other.
A straightforward check shows that $Z\Ga$ commutes with $N(G)$ and $\SNQ$ commutes with $\Aut(\Ga).$
Both groups $\SNQ$ and $\Aut(\Ga)$ normalise $N(G)$ and act in the expected way: 
$$[(\varphi,\beta)\cdot f](x):= \beta(f(\varphi^{-1}x)),$$
where
$ \varphi\in\SNQ, \beta\in\Aut(\Ga), f\in N(G), x\in \Q_2.$
These actions are clearly factorisable into actions on the quotient group $N(G)/D(Z\Ga).$
The group $\Aut(\Ga)$ normalises $Z\Ga$ acting as $\beta\cdot \zeta:=\beta(\zeta), \beta\in\Aut(\Ga),\zeta\in Z\Ga.$
However, $\SNQ$ does not normalise $Z\Ga$ but normalises the product of groups $Z\Ga\times N(G)/D(Z\Ga).$
The action is more complicated than expected but can be written down using slopes of elements of $V$.
For clarity of the presentation we choose to first define a semidirect product 
$$\left( Z\Ga\times N(G)/D(Z\Ga) \right) \rtimes \left(\SNQ\times \Aut(\Ga)\right)$$ 
without referring to $G$ and thus proving that the complicating formula describing the action of $\SNQ\times\Aut(\Ga)$ on $Z\Ga\times N(G)/D(Z\Ga)$ is indeed well-defined.
We end the proof of Theorem \ref{theol:B} by defining the group morphism from the semidirect product $\left( Z\Ga\times N(G)/D(Z\Ga) \right) \rtimes \left(\SNQ\times \Aut(\Ga)\right)$ onto $\Aut(G)$ and by computing its kernel which is an easy task.

\subsection*{Comparison between this present first article and the second one.}
In this first article, we consider the class of fraction groups $G(\Ga,\al_0,\al_1)$ built from a triple $(\Ga,\al_0,\al_1)$ where $\Ga$ is a group, $\al_0,\al_1$ are endomorphisms with one of them trivial say $\al_1.$
We start by reducing to the case where $\al_0$ is an automorphism using a direct limit process, i.e.~there exists a group $\La$ and an automorphism $\beta_0$ such that $G(\Ga,\al_0,\varep_\Ga)\simeq G(\La,\beta_0,\varep_\Lambda).$
We prove a rigidity phenomena on maps between two of such fraction groups showing that all isomorphisms are spatial. 
This led to Theorem \ref{theol:A}: a complete classification of the class of fraction groups and Theorem \ref{theol:B}: a thin description of the automorphism group of an untwisted fraction group (i.e.~fraction groups obtained from the identity automorphism $\al_0=\id_\Ga$ and the trivial endomorphism $\al_1=\varep_\Ga$).
The proof of the rigidity phenomena is rather easy and so is the proof of Theorem \ref{theol:A}. The main technical difficulty is to obtain Theorem \ref{theol:B};  in particular in finding all automorphisms and understanding the structure of the automorphism group.

In the second article we follow a similar scheme than in the first one. 
We consider all triples $(\Ga,\al_0,\al_1)$ and their associated fraction groups $G(\Ga,\al_0,\al_1)$ where $\al_0,\al_1$ are any endomorphisms. 
Contrary to the first article there exists some groups $G(\Ga,\al_0,\al_1)$ not isomorphic to any $G(\La,\beta_0,\beta_1)$ with $\beta_0,\beta_1$ automorphism of a group $\La.$
After describing the general structure of $G(\Ga,\al_0,\al_1)$ for any triple we quickly specialise to $\al_0,\al_1$ being automorphisms in order to pursue a deeper study.
We prove a similar rigidity phenomena showing that all isomorphisms between two such fraction groups are spatial (up to multiplying the isomorphism by a centrally valued morphism).
This is the most difficult proof of the second article. From it we obtain a partial classification of the class of fraction group and in particular show that the fraction group obtained from $(\Ga,\al_0,\al_1)$ remembers $\Ga$. 
We do not know how to obtain a complete classification. 
We then prove that the two classes of groups considered in the first and second articles are disjoint: if $\Ga$ is nontrivial and $\al_0,\al_1$ are automorphisms, then $G(\Ga,\al_0,\al_1)$ is never isomorphic to $G(\La,\beta_0,\varep_\La)$ for a group $\La.$ 
This is proved by computing certain centralizer subgroups. 
We further prove that there are no nice embedding between one group of the first class and one group of the second class.
We describe the automorphism group of an untwisted fraction group (i.e.~when $\al_0=\al_1=\id_\Ga$) by decomposing each automorphism into a product of elementary ones. 
It is slightly easier to find these elementary automorphisms but the proof showing that there products exhaust all automorphisms is still quite subtle. 

\section*{Acknowledgement}
We warmly thank the anonymous referees for all the constructive and judicious suggestions they made to us.

\section{Preliminaries}\label{sec:preliminary}
In this section we start by briefly presenting the general framework of Jones' actions introduced in \cite{Jones16-Thompson}.
The general philosophy is that a nice category provides a group (a fraction group) and a functor starting from this category provides an action of this group (a Jones' action). 

The construction of such fraction groups comes from ideas of Ore which appear in the work of Mal'tsev and were adapted to categories by Gabriel and Zisman \cite{Maltsev53,GabrielZisman67}. 
The construction of actions from functors is the novel part of Jones' technology. 
Although, it can be interpreted as a special case of KAN extension. 
We refer the reader to the short appendix of \cite{Brothier19WP} and to the survey \cite{Brothier-19-survey} for details and a discussion.

We specialise our study to covariant monoidal functors from the category of binary forests to the category of groups which provides actions of Thompson's group $V$ on a group.
We consider the semidirect product which again has a natural structure of fraction group that we describe.
All of these have been extensively explained in \cite[Section 2]{Brothier19WP} that we refer the reader to for additional details.

We end this preliminary section by recalling and proving elementary facts on  Thompson's group $V$ that we consider as a subgroup of the homeomorphism group of Cantor space. Moreover, we recall properties of its automorphism group.

\subsection{Jones general framework}\label{sec:general-fram}
Consider a small category $\cC$ with a chosen object $e\in\ob(\cC)$ and assume that $\cC$ admits a calculus of left fractions in $e$. 
We then consider the set of pairs $(f,g)$ of morphisms of $\cC$ having domain $e$ and common codomain. 
We mod out this set of pairs by the equivalence relation generated by $(f,g)\sim (p\circ f, p\circ g)$ for any composable morphism $p$ and write $\dfrac{f}{g}$ for the equivalence class of $(f,g)$ calling it a fraction.
This quotient set admits a group structure given by the multiplication:
$$\dfrac{f}{g}\cdot \dfrac{f'}{g'} := \dfrac{p\circ f}{p'\circ g'} \text{ for any choice of $p,p'$ satisfying } p\circ g = p'\circ f'.$$
This forms a group $G_{\cC,e}=G_\cC$ that we call the fraction group of $(\cC,e)$ or simply the fraction group of $\cC$ if the context is clear.
Note that we are following here the original convention of Jones' construction. Hence, the fraction $\dfrac{f}{g}$ must be understood as the composition $f^{-1}\circ g$ where $f^{-1}$ is a formal inverse of $f$.

Jones made the fundamental observation that given a functor $\Phi:\cC\to\cD$ one can construct a group action $\pi_\Phi: G_\cC\act X_\Phi$ called the \textit{Jones' action}. 
The construction is again based on considering fractions.
To ease the presentation assume that morphism spaces of $\cD$ are sets and $\Phi$ is covariant. 
Consider the set of all pairs $(g,d)$ where $g$ is a morphism from $e$ to $a$ and $d$ is a morphism from $\Phi(e)$ to $\Phi(a)$ where $a$ runs over all objects of $\cC$.
Define the equivalence relation generated by $(g,d)\sim (f\circ g, \Phi(f)\circ d)$ on this set of pairs and write $\dfrac{g}{d}$ for the equivalence class of $(g,d)$ that we call a fraction.
We obtain a set of fractions $X_\Phi$ and the Jones action $\pi_\Phi:G_\cC\act X_\Phi$ defined by the formula:
$$\dfrac{f}{g}\cdot \dfrac{f'}{d} := \dfrac{p\circ f}{\Phi(p')\circ d} \text{ for any choice of $p,p'$ satisfying } p\circ g = p'\circ f'.$$
If the morphism spaces of $\cD$ are not sets, then we can adapt the construction by considering $\Hom_\cD(\Phi(e), \Phi(a))$ rather than $\Phi(a).$
A remarkable feature of this construction is that it does not require any assumptions on $\cD$ nor $\Phi.$ Hence, any functor starting form $\cC$ will provide a Jones' action. 
We will describe in details a collection of Jones' actions in Section \ref{sec:construction-FG} and in Proposition \ref{prop:twistWP}. 

If $\Phi$ is a covariant functor and $\cD$ is the category of Hilbert spaces with isometries for morphisms, then $X_\Phi$ is a pre-Hilbert space and the Jones action $\pi_\Phi$ can be extended into a unitary representation of $G_\cC$ on the completion of $X_\Phi.$
This provides a wonderful machine for constructing unitary representations and matrix coefficients, see \cite{Jones19Irred,Brot-Jones18-2,Brot-Jones18-1,Brothier19WP,ABC19}.

If $\Phi:\cC\to\Gr$ is a covariant functor where $\Gr$ is the category of groups, then $X_\Phi$ is a group and the Jones action $\pi_\Phi: G_\cC\act X_\Phi$ is an action by automorphisms on this group.
We can then consider the semidirect product $X_\Phi\rtimes G_\cC$ obtaining a group from the functor $\Phi$ (and the choice of a fixed object $e\in\cC$).

The author made the observation that $X_\Phi\rtimes G_\cC$ has a very natural description in terms of fraction group, see \cite{Brothier19WP}. 
There is a category $\cC_\Phi$ with the same objects as $\cC$ but with typically more morphisms. The fraction group of $(\cC_\Phi,e)$ is then isomorphic to $X_\Phi\rtimes G_\cC$.

We will describe and study those semidirect products when the initial category $\cC$ is a certain category of forests and $\Phi$ is covariant and monoidal.

\subsection{The case of forests and groups}

\subsubsection{The category of forests}\label{sec:forest}
{\bf Trees and forests.}
An ordered rooted binary tree $t$ is a tree-graph with one root $*$ and finitely many vertices.
We imagine it as a graph drawn in the plane with the root on the bottom and leaves on top.  Every vertex $v$ of $t$ that is not a leaf has two descendants $v_l,v_r$ that are vertices placed at the top left and top right of $v$, respectively. 
We say that the edge from $v$ to $v_l$ (resp. from $v$ to $v_r$) is a left edge (resp. a right edge).
A forest is a union of finitely many ordered rooted binary trees where the roots (and hence the trees) are ordered from left to right.
From now on we will call these objects trees and forests.

{\bf Category of forests.}
We form a (small) category $\cF$ whose set of objects is $\N:=\{0,1,2,3,\cdots\}$ and whose morphism space $\cF(n,m)$ from $n$ to $m$ is the set of all forests having $n$ roots and $m$ leaves for $1\leq n\leq m.$
By convention, $\cF(0,0)$ has one morphism corresponding to an empty diagram and $\cF(0,n)$ is empty for all $n\geq 1.$
The composition is done by stacking vertically forests $f,g$ on top of each other by lining the leaves of the bottom forest $g$ with the roots of the top forest $f$ obtaining $f\circ g.$
We may write $fg$ rather than $f\circ g$ when composing forests.
We equip this category with a monoidal structure $\otimes$ such that 
$n\otimes m = n+m$ for objects $n,m$ and $f\ot g$ is the horizontal concatenation of the two forests $f$ and $g$ where $f$ is on the left and $g$ on the right.
Note that the single element of $\cF(0,0)$ is the tensor unit.
We think of $\cF$ in the naive sense: it is a set equipped with two binary operations: a partially defined composition which corresponds to \textit{vertical} concatenations and a tensor product which corresponds to \textit{horizontal} concatenations.

{\bf Presentation.}
Denote by $I$ and $Y$ the tree with one leaf and the tree with two leaves, respectively. 
The category $\cF$ has the remarkable property that every morphism is the composition of tensor products of $I$ and $Y$.
Indeed, write $f_{j,n}= I^{\ot j-1}\ot Y\ot I^{\ot n-j}$ for the forest with $n$ roots, $n+1$ leaves such that the $j$-th tree is $Y$ and all the others are the trivial tree $I$ for $1\leq j\leq n.$
It is easy to see that every forest is the result of finite compositions of such $f_{j,n}$.
In fact, the category $\cF$ admits the presentation having $\{f_{j,n}:\ 1\leq j\leq n\}$ for set of generators and set of relations:
$$f_{q,n+1}\circ f_{j,n} = f_{j,n+1}\circ f_{q-1,n} \text{ for all } 1\leq j <q \leq n+1.$$

{\bf Partial order.}
Let $\fT$ be the set of all trees, which is the set of all morphisms of $\cF$ with source $1.$
Given a tree $t$ we equip its set of vertices with the usual tree distance $d$: $d(v,w)$ is the number of edges in the unique geodesic path between $v$ and $w$.
We equip $\fT$ with the following partial order:
$$t\leq s \text{ if and only if } s=f\circ t \text{ for some forest } f.$$
For each $n\geq 1$ we write $t_n$ for the tree with $2^n$ leaves all at distance $n$ from the root.
Observe that $(t_n:\ n\geq 1)$ is a cofinal sequence in $(\fT,\leq)$ meaning that for all $t\in \fT$ there exists $n\geq 1$ satisfying $t\leq t_n.$
We consider $t_\infty$ the \textit{infinite} rooted binary tree. 
For convenience, we will often identify elements of $\fT$ with finite rooted sub-trees of $t_\infty.$

\subsubsection{Cantor space}\label{sec:cantor}

{\bf Cantor space and the unit interval.}
We write $\fC$ for Cantor space that we define as being the set of all infinite sequences in $0$ and $1$, that is $\fC:=\{0,1\}^\N$, equipped with the product topology. 

We consider the map $$S:\fC\to [0,1], x=(x_n)_{n\in\N}\mapsto \sum_{n\in\N}\dfrac{x_n}{2^n}$$
which is surjective.
Recall that a dyadic rational is a number of the form $\dfrac{a}{2^b}$ with $a,b\in\Z$ equal to the ring $\Z[1/2].$
Any element of $[0,1]$ that is not a dyadic rational has a unique pre-image.
However, each dyadic rational $r\in(0,1)$ admits exactly two pre-images: $x,y\in\fC$ satisfying that there exists $N\geq 1$ such that $x_n=y_n$ if $n\leq N-1$, $x_N=1, y_N=0$ and $x_n=0,y_n=1$ for all $n\geq N+1$.
For example, $$\dfrac{5}{8} = \dfrac{1}{2} + \dfrac{1}{8} = S(101000\cdots) = S(100111\cdots) = \dfrac{1}{2} + \sum_{k=4}^\infty \dfrac{1}{2^k}.$$
In particular, $S$ realises a bijection from the set of finitely supported sequences $$\{x\in\fC:\ \exists N\geq 1,\ x_n=0,\ \forall n\geq N\}$$ onto the set of dyadic rationals contained in $[0,1)$. 

\begin{notation}We write $\Q_2$ for the set $\Z[1/2]\cap [0,1)$ that we identify with $\{ S(x):\ x\in\fC,\ \exists N\geq 1,\ x_n=0,\ \forall n\geq N\}$ and with the set of finitely supported sequences $\fs$ in $\fC$.\end{notation}

We write $\leq$ for the lexicographic order of $\fC$ and remark that $S(x)\leq S(y)$ if and only if $x\leq y$ for all $x,y\in\fC.$ 

{\bf Standard dyadic intervals.}
The topology of $\fC$ is generated by the following clopen sets (sets that are open and closed)  that are 
$$I:=\{m_I\cdot x:\ x\in \{0,1\}^\N\}$$ 
where $m_I\in\fs$ is a finite sequence of $0$ and $1$ that we call a word and where the symbol $\cdot$ is the concatenation.
We say that $I$ is a \textit{clopen interval} of $\fC$.
Observe that $S(I)=[\dfrac{a}{2^b},\dfrac{a+1}{2^b}]$ for certain $a,b\in\N$.
Moreover, if $|m_I|$ is the word length of the word $m_I$, then the Lebesgue measure of $S(I)$ is equal to $2^{-|m_I|}$ for any such $I$, i.e.~$|m_I|=b$ if we adopt the previous notations.
For technical reason we will consider the half-open interval  $\dot S(I):=[\dfrac{a}{2^b},\dfrac{a+1}{2^b})$ and call it a \textit{standard dyadic interval} (in short \textit{sdi}).
By abuse of terminology we may call $I$ a sdi rather than a clopen interval and may identify it with $\dot S(I)$ or with $\dot S(I) \cap \Q_2.$

{\bf Standard dyadic partitions.}
Consider a finite collection of clopen intervals $I_1,\cdots,I_n$ that are mutually disjoint and whose union is equal to $\fC$.
Up to reordering them we obtain that 
$$S(I_1)=[0,a_1], S(I_2) = [a_1,a_2],\cdots, S(I_n)=[a_{n-1},1]$$ with 
$$0<a_1<a_2<\cdots<a_{n-1}<1.$$
We say that the corresponding family of half-open intervals $[0,a_1), [a_1,a_2),\cdots, [a_{n-1},1)$ is a \textit{standard dyadic partition} of $[0,1)$ (in short \textit{sdp}).
The family is \textit{ordered} if $\sup(I_k)\leq \inf(I_{k+1})$ for all $1\leq k\leq n-1.$
If the context is clear we may suppress the word ordered.
The set of sdp admits a partial order: if $Q,P$ are sdp, then $Q\leq P$ if every sdi of $Q$ is a union of sdi of $P$.

{\bf Interpretation using trees.}
We now present how trees are useful for studying and representing Cantor space as a topological space.
Consider the rooted binary infinite tree $t_\infty$ with root $*$. 

{\bf From paths to sequences.}
Given an edge $e$ of $t_\infty$ we set $E(e)=0$ if $e$ is a left edge and $E(e)=1$ if $e$ is a right edge where $E$ stands for evaluation.
If $p$ is a path going from bottom to top (which is necessarily a geodesic then) it is the concatenation of some edges $p=e_1\cdot e_2\cdot e_3\cdots$. 
We extend $E$ on those paths as $E(p) = E(e_1)\cdot E(e_2)\cdot E(e_3)\cdots.$
If $p$ is finite, then $E(p)$ is a word in $0,1$ and if $p$ is infinite, then $E(p)\in\{0,1\}^\N.$
It is easy to see that $E$ realises a bijection from the set of infinite geodesic paths of $t_\infty$ with source $*$ and Cantor space.

{\bf From vertices to sdi.}
Given a vertex $\nu$ of $t_\infty$ there exists a unique geodesic path $p_\nu$ going from $*$ to $\nu$. 
We can then consider $E(p_\nu)\in\fs$ and the following subset of Cantor space:
$$I_\nu:=\{ E(p_\nu)\cdot x:\ x\in\is\}$$
that are all elements of $\fC$ with common prefix $E(p_\nu).$
We obtain bijections between vertices of $t_\infty$, finite sequences of 0,1 and sdi (i.e.~clopen intervals) of $\fC$.

{\bf From trees to sdp.}
Consider now a tree $t\in\fT$ that is a finite rooted tree that we identify with a rooted subtree of $t_\infty.$
To each leaf $\ell$ of $t$ corresponds a vertex of $t_\infty$ and thus a sdi $I^t_\ell$ of $\fC$. 
We obtain that $(I^t_\ell:\ \ell\in\Leaf(t))$ is a sdp of $\fC$ and in fact
$$t\in\fT\mapsto(I^t_\ell:\ \ell\in\Leaf(t))=:P_t$$
realises a bijection between the trees and the sdp of $\fC.$

{\bf Refinement and partial order.}
Consider now $t\in\fT$ and $f\in \Hom(\cF)$ a forest composable with $t$. 
Observe that the sdp $P_{f\circ t}$ associated to $f\circ t$ is a refinement of the sdp $P_t$: that is $P_t\leq P_{f\circ t}.$ 
Hence, $t\mapsto P_t$ defined above is an order-preserving bijection.

\subsubsection{Thompson's groups}

{\bf Thompson's group $V$.}
Consider two sdp $(I_1,\cdots,I_n)$ and $(J_1,\cdots,J_n)$ of $\fC$ with the same number of sdi. 
There exists a unique homeomorphism $v$ of $\fC$ satisfying that 
$$v(m_{I_k}\cdot x) = m_{J_k}\cdot x$$
for all $1\leq k\leq n$ and $x\in\is$ where $m_{I_k}$ is the unique word satisfying that 
$$I_k=\{ m_{I_k}\cdot x:\ x\in\is\}.$$
We write $V$ for the collection of all such maps $v$ which forms a group called Thompson's group $V$.
The element $v$ defines a bijection of $[0,1)$ as follows: consider the unique map realising an increasing affine bijection from $\dot S(I_k)$ to $\dot S(J_k)$ (recall that $\dot S(I_k)$ is equal to $S(I_k)$ minus its last point and $S:\fC\to [0,1]$ is the classical surjection).
This provides a piecewise linear bijection of $[0,1)$ having finitely many discontinuity points all at dyadic rational and having slopes powers of 2.
Equivalently, all such bijection of $[0,1)$ comes from an element of $V$.

{\bf Thompson's groups $F$ and $T$.}
Thompson's group $V$ contains two subgroups $F,T$ satisfying that $F\subset T$: the subgroup $T$ (resp. $F$) is the set of all transformation of $V$ which restricts to an homeomorphism of $\R/\Z$ (resp. of $[0,1)$) which is the set of all transformations $v$ as above sending an \textit{ordered} sdp $(I_1,\cdots,I_n)$ to another $(J_1,\cdots,J_n)$ in such a way that there exists $0\leq d\leq n-1$ satisfying that $v(I_k)=J_{k+d}$ (resp. $v(I_k)=J_k$) for all $1\leq k\leq n$ and where the index $k+d$ is considered modulo $n.$

{\bf Thompson's groups as fraction groups.}
The three Thompson's groups can be realised as fraction groups.
Thompson's group $F$ is isomorphic to the fraction group of $(\cF,1)$.
This comes from the fact that an element of $F$ is totally described by two ordered sdp with the same number of sdi.
The two sdp corresponds to a pair of trees with the same number of leaves.
To obtain Thompson's group $T$ one has to index the leaves of the trees in a cyclic way and thus we then consider the category of affine forests $\cAF$ with object the nonzero natural numbers and morphism space $\cAF(n,m)=\cF(n,m)\times \Z/m\Z, n\leq m.$
We obtain that $T$ is isomorphic to the fraction group of $(\cAF,1)$.
To obtain the larger Thompson's group $V$ one can have any indexing of the leaves of trees which corresponds in adding any permutation in the data of the trees giving the category $\cSF$ with same set of objects than before but with morphism space $\cSF(n,m)=\cF(n,m)\times S_m, n\leq m$ where $S_m$ is the group of permutations of $\{1,\cdots,m\}.$
Thompson's group $V$ is isomorphic to the fraction group of $(\cSF,1)$ and so $v\in V$ is equal to a fraction $\dfrac{\tau\circ t}{\sigma\circ s}$
with $t,s$ trees and $\tau,\sigma$ permutations playing the role of indexation of the leaves of $t$ and $s$, respectively.
Note that this fraction is equal to $\dfrac{\sigma^{-1}\tau \circ t}{s}=\dfrac{t}{\tau^{-1}\sigma\circ s}$ so we can always represent elements of $V$ with a fraction having at most one nontrivial permutation.

\subsubsection{Constructions of fraction groups}\label{sec:construction-FG}
The plan of this section is as follows.
First, we restrict the class of functors $\Phi:\cF\to\Gr$ considered that are in bijection with triples $(\Ga,\al_0,\al_1)$ with $\Ga$ a group and $\al_0,\al_1\in\End(\Ga).$
We fix one functor $\Phi$ and its associated triple $(\Ga,\al_0,\al_1).$
Second, we construct a limit group $\varinjlim_{t\in\fT}\Ga_t$ obtained by taking a limit on all trees $t\in\fT$. 
This limit group is the set $X_\Phi$ that we have described earlier but having, in this particular case, the extra-property of being a group.
Third, we provide three descriptions of the Jones action of Thompson's group $V$ on the limit group $\varinjlim_{t\in\fT}\Ga_t$. This is the Jones action $\pi_\Phi:V\act X_\Phi$ defined earlier.
Fourth, we describe the semidirect product $X_\Phi\rtimes V= \varinjlim_{t\in\fT}\Ga_t\rtimes V$ as a fraction group. 
To do this, we define a category $\cC_\Phi$ and establish that $X_\Phi\rtimes V$ is isomorpic to the fraction group of $(\cC_\Phi,1).$

{\bf Description of covariant monoidal functors.}
Consider $\Gr$ the category of groups that we equip with the classical monoidal structure: the direct sum $\oplus.$ 
Consider a (covariant) monoidal functor $\Phi:\cF\to \Gr$.
If $\Ga=\Phi(1)$, then $\Phi(n)=\Ga^n$ the $n$-th direct sum of $\Ga$ for all $n\geq 1$ since $\Phi$ is monoidal.
Consider $R:=\Phi(Y)$ (where $Y$ stands for the tree with two leaves) which is a group morphism from $\Ga$ to $\Ga\oplus\Ga$ and thus of the form 
$$R(g)=(\al_0(g),\al_1(g)),\ g\in \Ga$$ 
for some endomorphisms $\al_0,\al_1\in \End(\Ga)$.
Note that each morphism of $\cF$ is a composition of some $f_{j,n}=I^{\ot j-1}\ot Y\ot I^{\ot n-j}$ and that
$$\Phi(f_{j,n}) = \id_{\Ga^{j-1}} \oplus R \oplus \id_{\Ga^{n-j}}, \  n\geq 1, 1\leq j\leq n$$
since $\Phi$ is monoidal.
Therefore, $\Phi$ is completely described by $\Ga$ and the (ordered) pair $(\al_0,\al_1).$
Conversely, any choice of group $\Ga$ and pair of endomorphisms $(\al_0,\al_1)\in\End(\Ga)^2$ defines a covariant monoidal functor from $\cF$ to $\Gr.$
This latter fact comes from the presentation of the category $\cF$ given in Section \ref{sec:forest}.

{\bf Construction of the limit group $X_\Phi$.}
Assume we have chosen $\Ga,\al_0,\al_1$ and consider $R$ as above.
Let $\Phi$ be the associated monoidal functor.
For each tree $t\in\fT$ we consider 
$$\Ga_t=\{(g,t):\ g\in \Ga^n\}$$ 
a copy of the group $\Ga^n$ where $n=|\Leaf(t)|$ is the number of leaves of $t$.
We may identify $\Ga_t$ with the group $\{g:\Leaf(t)\to \Ga\}$ of all maps from $\Leaf(t)$ to $\Ga$.
Given a forest $f$ with $n$ roots we define 
$$\iota_{ft,t}:\Ga_t\to \Ga_{ft} \text{ such that } \iota_{ft,t}(g_1,\cdots,g_n):= \Phi(f)(g_1,\cdots,g_n).$$
For example, if $t=Y$ and $f=I\ot Y$, then $$\iota_{ft,t}(g_1,g_2) = (g_1, R(g_2)) = (g_1, \al_0(g_2), \al_1(g_2)).$$
This provides a directed system of groups $$(\Ga_t, \ \iota_{s,t}:\ s,t\in\fT, s\geq t)$$
indexed by the set of trees $\fT$.
Let $\varinjlim_{t\in\fT} \Ga_t$ be the directed limit which is still a group and which is in bijection with the quotient space:
$$\{ (g,t):\ t\in \fT,\ g\in \Ga^{ |\Leaf(t)|} \}/\sim$$
where $$(g,t)\sim (g',t') \text{ if and only if $\exists f,f'$ satisfying } ft=f't' \text{ and } \Phi(f)(g)=\Phi(f')(g').$$
This is the group $X_\Phi$ mentioned earlier in Section \ref{sec:general-fram}.

{\bf The Jones action of Thompson's group $V$ on $X_\Phi$.}

{\bf First description.}
We now describe the Jones action $\pi_\Phi: V\act \varinjlim_{t\in\fT} \Ga_t$.
If $v\in F$, then $v$ is described by a pair of trees $(t,s)$ and thus by a fraction $\dfrac{t}{s}.$
Consider $g\in \varinjlim_{t\in\fT}\Ga_t$.
Up to refining both $t$ and $s$ (by considering $(ft,fs)$ rather than $(t,s)$) we can assume that $g$ admits a representative $(h,s)\in \Ga_s$ and thus using the fraction notation we have $g=\dfrac{s}{h}.$
the Jones action is then $$\pi_\Phi\left(\dfrac{t}{s}\right)\left(\dfrac{s}{h}\right) = \dfrac{t}{h}.$$
Hence, $\pi_\Phi\left(\dfrac{t}{s}\right)$ restricted to $\Ga_s$ consists in the following trivial isomorphism:
$(h,s)\mapsto (h,t)$
from $\Ga_s$ to $\Ga_t.$

{\bf Second description.}
Here is another way to interpret this action.
Consider $g\in \varinjlim_{t\in\fT} \Ga_t$ and a representative $g$ in $\Ga_s$ for a tree $s$ that is large enough.
We may identify $g$ with a map $$\Leaf(s)\to \Ga, \ell\mapsto h_\ell.$$
We interpret $g$ as the tree $s$ where each of its leaf $\ell$ is decorated by $h_\ell$.
Consider $v\in V$ that we write as a fraction $\dfrac{\sigma\circ t}{s}$ where $t$ is a tree and $\sigma$ is a permutation.
Now $\sigma$ can be interpreted as a bijection $b:\Leaf(t)\to\Leaf(s)$ since $t$ and $s$ must have the same number of leaves.
We obtain that $\pi_\Phi(v)(\dfrac{s}{h})$ is equal to the tree $t$ such that any of its leaf $\lambda$ is decorated by the group element $h_{b(\lambda)}.$
Note that if $v\in F$, then $\sigma$ is trivial. 
This implies that we do not change the order of the component of $h$ (when reading from left to right the group element placed on top of the leaves of the tree $s$ or $t$).
If $v\in T$ we may change cyclically the order of the components and if $v\in V$ we may change using any permutation the order of the components.

{\bf Third description.}
Here is a spatial interpretation of Jones' action.
Consider $v\in V$ and $g\in \varinjlim_{t\in\fT}\Ga_t.$
There exist two sdp $I:=(I_1,\cdots,I_n)$ and $J:=(J_1,\cdots,J_n)$ of Cantor space such that $v(I_k)=J_k, 1\leq k\leq n$ and $v$ is adapted to the first sdp, i.e.~the restriction of $v$ to $I_k$ is affine.
There exist a tree $s$ and a representative of $g$ in $\Ga_s.$
Let $L:=(L_1,\cdots,L_m)$ be the sdp associated to $s$. 
The element $g$ is then the same data than the pair $(g_L,L)$ where $g_L:\fC\to\Ga$ is a map constant on each $L_k$ taking a certain value $g_k$ for $1\leq k\leq m.$
Let us assume that the partition $I$ is thinner than $L$. 
For simplicity, suppose for instance that $I=(L_1^0,L_1^1, L_2,\cdots,L_m)$ where $L_1^0,L_1^1$ are the first and second half of $L_1$, respectively.
We now represent $g$ as the pair $(g_I,I)$ with the map $g_I:\fC\to\Ga$ taking the values
$$g_I(x) =\begin{cases}
\al_0(g_1) \text{ if } x\in I_1=L_1^0\\
\al_1(g_1) \text{ if } x\in I_2=L_1^1\\
g_k \text{ if } x\in I_k=L_{k-1} \text{ with } 2\leq k\leq n
\end{cases}.$$
The element $\pi_v(g)$ is represented by the pair $(g_J,J)$ where:
$$g_J(x) =\begin{cases}
\al_0(g_1) \text{ if } x\in J_1\\
\al_1(g_1) \text{ if } x\in J_2\\
g_k \text{ if } x\in J_k \text{ with } 2\leq k\leq n
\end{cases}.$$

{\bf Description of the semidirect product $X_\Phi\rtimes V$ as a fraction group.}
We will now explain how we can interpret the semidirect product $\varinjlim_{t\in\fT} \Ga_t\rtimes V$ as a fraction group. 
We start by defining a category that admits a left calculus of fractions.
Consider the category $\cC_\Phi$ whose set of object is $\N$ and morphism spaces $\cC_\Phi(n,m)$ equal to $\cF(n,m)\times S_m\times \Ga^m$.
A morphism can be diagrammatically represented as a forest plus on top of it a permutation that we represent as a diagram with $n$ segments (if the forest has $n$ leaves) going from $(k,a)$ to $(\sigma(k),a+1), 1\leq k\leq n$ where the coordinates are taken in $\R^2$ and where $a$ stands for the altitude of the leaves in the diagram. On top of the permutation we place some elements of $\Ga$. Note that if $\Ga$ was trivial, then we will simply obtain a diagrammatic representation of the category giving the fraction group $V$. 
In picture we obtain that the morphism $(Y,(21), g_1,g_2)\in \cC_\Phi(n,m)$ is represented by the diagram:
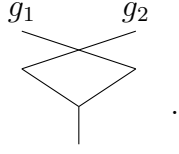
\newcommand{\Ysigmag}{
\begin{tikzpicture}[baseline = .4cm]
\draw (0,0)--(0,.5);
\draw (0,.5)--(-.75,1);
\draw (0,.5)--(.75,1);
\draw (-.75,1)--(.75,1.5);
\draw (.75,1)--(-.75,1.5);
\node at (-.75,1.75) {$g_1$};
\node at (.75,1.75) {$g_2$};
\end{tikzpicture}
}
$$\Ysigmag \ .$$
We interpret $(Y,(21), g_1,g_2)$ as the composition of three morphisms being $(g_1,g_2), (21)$ and $Y$ and thus identify $\Ga^m$ and $S_m$ as subsets of the endomorphism space $\cC_\Phi(m,m)$ and $\cF(n,m)$ as subsets of the morphism space $\cC_\Phi(n,m), n,m\geq 1.$
The composition of morphisms is explained by the following diagrams where we freely use the identifications just mentioned:
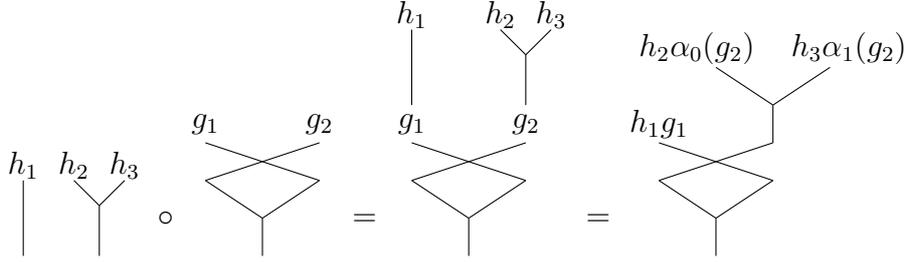
\newcommand{\hforest}{
\begin{tikzpicture}[baseline = .4cm]
\draw (0,0)--(0,1);
\draw (1,0)--(1,2/3);
\draw (1,2/3)--(2/3,1);
\draw (1,2/3)--(4/3,1);
\node at (0,1.2) {$h_1$};
\node at (2/3,1.2) {$h_2$};
\node at (4/3,1.2) {$h_3$};
\end{tikzpicture}
}
\newcommand{\compoone}{\begin{tikzpicture}[baseline = .4cm]
\draw (0,0)--(0,.5);
\draw (0,.5)--(-.75,1);
\draw (0,.5)--(.75,1);
\draw (-.75,1)--(.75,1.5);
\draw (.75,1)--(-.75,1.5);
\node at (-.75,1.75) {$g_1$};
\node at (.75,1.75) {$g_2$};
\draw (-.75,2)--(-.75,3);
\draw (.75,2)--(.75,2+2/3);
\draw (.75,2/3+2)--(2/3-.25,3);
\draw (.75,2/3+2)--(4/3-.25,3);
\node at (-.75,3.2) {$h_1$};
\node at (2/3-.25,3.2) {$h_2$};
\node at (4/3-.25,3.2) {$h_3$};
\end{tikzpicture}}
\newcommand{\compotwo}{\begin{tikzpicture}[baseline = .4cm]
\draw (0,0)--(0,.5);
\draw (0,.5)--(-.75,1);
\draw (0,.5)--(.75,1);
\draw (-.75,1)--(.75,1.5);
\draw (.75,1)--(-.75,1.5);
\node at (-.75,1.75) {$h_1g_1$};
\draw (.75,1.5)--(.75,2);
\draw (.75,2)--(0,2.5);
\draw (.75,2)--(1.5,2.5);
\node at (-.25,2.75) {$h_2\alpha_0(g_2)$};
\node at (1.75,2.75) {$h_3\alpha_1(g_2)$};
\end{tikzpicture}}
\newcommand{\compothree}{\begin{tikzpicture}[baseline = .4cm]
\draw (0,0)--(0,.25);
\draw (0,.25)--(-2,1.25);
\draw (0,.25)--(2, 1.25);
\draw (-1,.75)--(0,1.25);
\draw (-2,1.25)--(0,2);
\draw (0,1.25)--(2,2);
\draw (2,1.25)--(-2,2);
\node at (-2,2.2) {$h_1g_1$};
\node at (0,2.2) {$h_2\alpha_0(g_2)$};
\node at (2,2.2) {$h_3\alpha_1(g_2)$};
\end{tikzpicture}}
$$\hforest \circ \Ysigmag = \compoone = \compotwo$$
which is equal to
$$
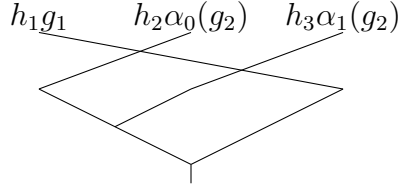
\compothree$$
Elements of the fraction group of this category (at the object 1) is the set of (equivalence classes of) pairs of decorated trees as defined above.
Note that the definition of the composition of the morphisms of $\cC_\Phi$ provides a number of skein-type relations such as the one of above (which consists in pulling a string around a vertex like the Yang-Baxter relation in a fusion category). We thus have a number of such skein relations that are automatically satisfied.

We are going to show that the fraction group of $(\cC_\Phi,1)$ is isomorphic to the semidirect product $X_\Phi\rtimes V= \varinjlim_{t\in\fT}\Ga_t\rtimes V$ by constructing an explicit isomorphism.
Consider an element $\ga$ of the fraction group of $(\cC_\Phi,1).$
It is the equivalence class of a pair of triples $( (t, \tau, g) , (s,\sigma,h))$ where $t,s$ are trees with $n$ leaves, $\tau,\sigma$ are permutations in $S_n$ and $g=(g_1,\cdots, g_n), h=(h_1,\cdots, h_n)$ are $n$-tuples of elements of $\Ga$, for a certain $n\geq 1.$
We may write it as a fraction $\dfrac{g \circ \tau\circ t}{ h \circ \sigma \circ s}$.
Note that $g$ and $\tau$ are automorphisms of $\cC_\Phi$ and thus we can multiply the numerator and denominator of the fraction by $(g\circ \tau)^{-1}.$ 
We obtain that $\ga=\dfrac{t}{\tau^{-1}\circ g^{-1}h\circ \sigma\circ s} = \dfrac{t}{h'\circ \sigma'\circ s}$ where $h'$ is obtained by permuting the entries of $g^{-1}h$ via the permuation $\tau^{-1}$ and $\sigma'=\tau^{-1}\sigma.$
Consider the map 
$$\ga= \dfrac{t}{h'\circ \sigma'\circ s} \mapsto \left( (h',t) , \dfrac{t}{\sigma' \circ s} \right) \in \Ga_t\times V.$$
This formula defines an isomorphism from the fraction group $G_{\cC_\Phi}$ onto the semidirect product $\left(\varinjlim_{t\in\fT}\Ga_t\right)\rtimes V$.


\subsection{Thompson's group: automorphisms and slopes}

\subsubsection{Slopes}
Consider $v\in V$ and note that for any $x\in [0,1)$ there exists a sdi $I$ on which $v$ is adapted to $I$ (is affine on this interval) with slope $2^n$ for a certain $n\in\Z.$
We write $v'(x)=2^n$ for the slope of $v$ at the point $x$.
Observe that $v(I)$ is also a sdi and we can find two words $m_I$ and $m_{v(I)}$ in 0,1 satisfying that 
$$I=\{ m_I\cdot x:\ x\in\is\} \text{ and } v(I)=\{m_{v(I)}\cdot x:\ x\in\is\}$$
where we view $I$ and $v(I)$ inside Cantor space $\fC$.
Recall that $|m_I|$ is the number of letters in the word $m_I$, i.e.~the word length of the word $m_I$.
One can see that $n= |m_I|-|m_{v(I)}|$. In this way we can alternatively define the slope of $v\in V$ at any element of $\fC.$
We obtain that $v':\fC\to \Z, x\mapsto v'(x)$ is continuous and equivalently there exists a sdp $(I_1,\cdots,I_n)$ such that $v'$ is constant on each $I_k, 1\leq k\leq n.$

In order to have slightly lighter notations we may remove parenthesis and write $vx$ and $vI$ rather than $v(x)$ and $v(I)$ for $v\in V, x\in\Q_2$ and $I$ any subset of $\Q_2$ or $\fC$.

It is easy to see that the slope of elements of $V$ satisfies the chain rule:
$$(vw)'(x) = v'(wx)\cdot w'(x) \text{ for all } v,w\in V, x\in [0,1) \text{ or in } \fC.$$
This will be often used. We will consider the map 
$$\ell:V\to \prod_{\Q_2}\Ga,\ v\mapsto \ell_v$$
such that 
$$\ell_v(x):= \log_2(v'(v^{-1}x)) \text{ for all } v\in V, x\in\Q_2$$
where $\log_2$ is the logarithm in base $2$.
Note that the chain rule implies that $\ell_{vw} = \ell_v + \ell_w^v$ where $\ell_w^v(x):= \ell_w(v^{-1}x), x\in \fC, v,w\in V.$

Consider $x\in\fC$ and the stabiliser subgroup
$$V_x:=\{v\in V:\ vx=x\}.$$
Denote by $V_x'$ the derived subgroup of $V_x$ (the subgroup generated by the commutators).
We provide a description of $V_x'$ and $V_x/V_x'$ using slopes. This is rather standard. A full proof can be found in \cite{Bleak-Lanoue10}. We provide a short proof for the convenience of the reader.
Recall that the group of germs of $V$ at a point $x\in\fC$ is the group $V_x$ quotiented by the subgroup $W_x$ of elements of $V_x$ that acts like the identity on a neighborhood of $x$.

\begin{lemma}\label{lem:slopeone}
Fix $x\in \Q_2$ and consider $V_x$ and its derived subgroup $V_x'$.
We have that $V_x'$ is the subgroup of $v\in V$ such that $v$ acts like the identity on a neighborhood of $x$ inside $\fC$ and equivalently  $$V_x'=\{v\in V:\ vx=x \text{ and } v'(x)=1\}.$$
The abelianisation $V_x/V_x'$ of $V_x$ is thus equal to the group of germs of $V$ at the point $x$.
Moreover, the morphism $V_x\to \Z, v\mapsto \ell_v(x)$ factorises into an isomorphism
$V_x/V_x'\to\Z.$
\end{lemma}

\begin{proof}
Fix $x\in \Q_2$ and consider the map 
$$L:V_x\to \Z, v\mapsto \ell_v(x):=\log_2(v'(x)).$$
This is a group morphism. 
Note that since $\Z$ is abelian the kernel of $L$ contains the derived subgroup $V_x'$, i.e.~$V_x'\subset \ker(L).$
For the reverse inclusion observe that
$\ker(L)=\cup_I \Fix_V(I)$ where $I$ runs over all sdi containing $x$ and $\Fix_V(I)$ is the subgroup of $v\in V$ satisfying $vy=y$ for all $y\in I$.
It is easy to see that $\Fix_V(I)$ is isomorphic to $V$ if $I$ is a proper sdi of $\fC$ and thus is simple and nonabelian since $V$ is.
Therefore, $\Fix_V(I)'=\Fix_V(I)$ implying that $\ker(L)'=\ker(L)$ and thus $\ker(L)=\ker(L)'\subset V_x'$.
We deduce that $\ker(L)=V_x'$.
To show that $L$ is surjective it is sufficient to observe that for any $x\in \Q_2$ there exists a sdi $I$ starting at $x$ and $v\in V$ adapted to $I$ so that $v(I)$ is the first half of $I$.
Hence, $vx=x$ and $v'(x)=1/2$. In particular, $L(v)=-1$ which implies that $L$ is surjective.
This finishes the proof of the lemma.
\end{proof}

\subsubsection{Automorphism group of $V$}
Let $\Aut(V)$ be the automorphism group of $V$.
Recall that $V$ is defined as a subgroup of the homeomorphism group $\Homeo(\fC)$ of Cantor space $\fC$.
Define $$\NCV:=\{\varphi\in\Homeo(\fC):\ \varphi V\varphi^{-1} = V\}$$
the normaliser subgroup of $V$ inside $\Homeo(\fC)$.
We have a map:
$$\ad:\NCV\to \Aut(V), \varphi\mapsto \ad_\varphi \text{ defined as } \ad_\varphi(v):=\varphi v\varphi^{-1},$$ 
for all $\varphi\in\NCV, v\in V.$
A classical argument using Rubin's theorem implies that $\ad$ is an isomorphism \cite{Rubin96}, see \cite[Section 3]{BCMNO19} for details and a proof of the faithfulness of $\ad$.

We will study elementary properties of elements of $\NCV$.

\begin{lemma}
If $\varphi\in \NCV$ and $I$ is a sdi, then $\varphi(I)$ is a finite union of sdi.
\end{lemma}

\begin{proof}
Write $\cI$ for the set of finite unions of sdi. 
To each $v\in V$ we associate $\Fix(v)$ the set of $x\in\fC$ so that $vx=x.$
It is easy to see that $v\in V\mapsto \Fix(v)$ is a surjection from $V$ onto $\cI$.
Moreover, if $\varphi\in\NCV$, then $\Fix(\ad_\varphi(v)) = \Fix(\varphi v\varphi^{-1}) = \varphi(\Fix(v)).$
This implies that if $I\in\cI$, then $\varphi(I)\in\cI$ proving the lemma.
\end{proof}

In this paper we will be working with a subgroup of $\NCV$.
Let $\SNQ$ be the set of $\varphi\in\NCV$ stabilising $\Q_2$ inside $\fC$ that is:
$$\SNQ:=\{\varphi\in\Homeo(\fC):\ \varphi V\varphi^{-1} = V \text{ and } \varphi(\Q_2)=\Q_2\}.$$
Recall that $\Q_2\subset\fC$ is the set of sequences $x=(x_n)_{n\in\N}\in\is$ satisfying that there exits $N\geq 1$ such that $x_n=0$ for all $n\geq N$.

\begin{remark}
In general, an element of $\NCV$ does not stabilise $\Q_2.$
Consider for instance $x=(x_n)_{n\in\N}\mapsto (\ov{x_n})_{n\in\N}$ where $\ov 0=1, \ov 1=0.$
It is an element of $\NCV$ sending all stationary sequences eventually equal to 0 (so $\Q_2$) to the the space of all stationary sequences eventually equal to 1 (that is the other copy of the dyadic rationals inside Cantor space).

There exist more exotic elements of $\NCV$ which do not stabilise the union of the copies of the dyadic rationals inside Cantor space, see Remark \ref{rem:Shayo}.
\end{remark}

We have the following useful fact for elements of $\SNQ$. 

\begin{proposition}\label{prop:slope}
If $x\in \Q_2, v\in V$ satisfying $vx=x$ and $\varphi\in\SNQ$, then
$$(\varphi v\varphi^{-1})'(\varphi(x)) = v'(x).$$
\end{proposition}

\begin{proof}
Consider $x\in \Q_2, v\in V_x$ and $\varphi\in\SNQ.$

We start by showing that if $v'(x)<1$, then $(\varphi v\varphi^{-1})'(\varphi(x))<1.$
Define $I,J$ some sdi containing $x$ and $\varphi(x)$ that are adapted to $v$ and $\varphi v\varphi^{-1}$, respectively.
Note that since $x\in\Q_2$ the sdi $I$ is necessarily of the form $[x,x+a)$ and the restriction of $v$ to this interval acts in the following way: $x+b\mapsto x + v'(x)b.$
We have a similar description of $J$ and the restriction of $\varphi v\varphi^{-1}$ to $J$ since $\varphi(x)\in\Q_2,\varphi v\varphi^{-1}\in V$ and $(\varphi v\varphi^{-1})(\varphi(x)) = \varphi(x).$
Up to reducing $J$ we can assume that $\varphi(I)$ contains $J$.
Assume that $v'(x)<1$ and observe that this condition is equivalent to have that $\lim_{n\to\infty} v^n(y)=x$ for all $y\in I$.
Consider $z\in J$ that we can write as $\varphi(y)$ for some $y\in I$.
By continuity of $\varphi$ we obtain that 
$$\lim_{n\to\infty} (\varphi v\varphi^{-1})^n (\varphi(y)) = \varphi(\lim_{n\to\infty}v^n(y))= \varphi(x)$$ implying that $(\varphi v\varphi^{-1})'(\varphi(x))<1.$

We finish the proof using the group of germs $V_x/V_x'$.
Since $x\in \Q_2$ we have by Lemma \ref{lem:slopeone} that $L_x:V_x\to \Z, v\mapsto \log_2(v'(x))$ factorises into an isomorphism $\ov L_x:V_x/V_x'\to \Z.$
Since $\varphi\in \SNQ$ we have that $\varphi(x)\in\Q_2$ and thus an isomorphism $\ov L_{\varphi(x)}:V_{\varphi(x)}/V_{\varphi(x)}'\to\Z$.
Consider the automorphism $\ad_\varphi$ of $V$ and observe that $\ad_\varphi(V_x)=V_{\varphi(x)}$.
Hence, $\ad_\varphi$ factorizes into an isomorphism $\ov\ad_\varphi:V_x/V_x'\to V_{\varphi(x)}/V_{\varphi(x)}$.
We obtain an automorphism 
$$f:=\ov L_{\varphi(x)} \circ \ov\ad_\varphi \circ {\ov L_x}^{-1}\in\Aut(\Z).$$
Therefore, $f=\pm \id_\Z$.
We proved that if $v'(x)<1$, then $(\varphi v\varphi^{-1})'(\varphi(x))<1$.
Therefore, if $n$ is a negative integer, then so is $f(n)$. 
This implies that $f$ is the identity automorphism and thus $v'(x)=(\varphi v\varphi^{-1})'(\varphi(x))$.
\end{proof}

\begin{remark}\label{rem:Shayo}
Recall that Cantor space contains two copies of the dyadic rationals that we denote by $\Q_2$ and $\Q_2^1$.
The last proposition works for $\varphi\in\SNQ$ but we can relax the assumptions and only require that $\varphi$ stabilises $\Q_2\cup \Q_2^1$.
Although, by adapting the proof of above we obtain that, if $vx=x, v'(x)\neq 1,x\in\Q_2$ and $\varphi(\Q_2)\neq \Q_2,\Q_2^1$, then $(\varphi v\varphi^{-1})'(\varphi(x))\neq v'(x)$, see \cite[Proposition 1.2]{Brothier20-2} for a precise statement and a proof.
Feyisayo Olukoya constructed such an exotic example of $\varphi\in\NCV$ satisfying that $\varphi(\Q_2)$ does not intersect $\Q_2\cup \Q_2^1$.
This example was communicated to us by Collin Bleak and we warmly thank both of them for sharing it.
The construction of this exotic homeomorphism uses transducers that were developed in \cite{GNS00} and further used to describe and classify automorphisms of $V$ in \cite{BCMNO19}. 

Note that if $vx\neq x$, then the last proposition does not work for trivial reasons.
Indeed, consider the element $\varphi \in V$ which permutes cyclically the intervals $[0,1/2]$, $[1/2,3/4]$ and $[3/4,1]$ and $v\in V$ that permutes $[0,1/2]$ and $[1/2,1]$.
Note that $v'(x)=1$ for all $x\in\Q_2$ but $\varphi v\varphi^{-1}([0,1/2]) = \varphi v [3/4,1] = \varphi [1/4,1/2] = [5/8,6/8]$.
Hence, $\varphi v\varphi^{-1}$ has slope $1/4$ on the interval $[0,1/2]$ while $v$ has slope $1$ everywhere.
\end{remark}

\section{General properties of fraction groups constructed from forests}\label{sec:characteristic}

In this section we consider \textit{any} monoidal covariant functor $\Phi$ from the category of forests $\cF$ to the category of groups $\Gr$ before restricting into a smaller class of functors in the coming sections.
As we have seen in the previous section $\Phi$ is totally described by the choice of a group $\Ga$ and a morphism 
$$\Ga\to\Ga\oplus\Ga, \ g\mapsto (\al_0(g),\al_1(g))$$
with $\al_0,\al_1 \in \End(\Ga)$.
To emphasize the choice of the pair $(\al_0,\al_1)$ we write $\Phi_\al$ for $\Phi$.
Denote by $K_\al := \varinjlim_{t\in\fT} \Ga_t$ the associated limit group, $\pi_\al: V\act K_\al$ the Jones action and $G_\al:= K_\al\rtimes V$ the associated semidirect product.
Finally write $\cC_\al$ for the larger category of forests with leaves decorated by elements of $\Ga$ and natural numbers satisfying that $G_\al$ is the fraction group of the category $\cC_\al$ at the object $1.$
The aim of this section is to prove the following:

\begin{theorem}\label{theo:KinG}
The normal subgroup $K_\al\lhd G_\al$ is characteristic of $G_\al$, i.e.~any automorphism of $G_\al$ restricts to an automorphism of $K_\al.$

Using obvious notations, consider an isomorphism $\theta:G_\al\to G_\beta$ between two fraction groups. We have that $\theta(K_\al)=K_\beta.$
\end{theorem}

To prove the theorem we are going to show that $K_\al$ is the unique maximal normal subgroup of $G_\al$ satisfying a certain decomposability property.

\begin{definition}
Consider a normal subgroup $N\lhd G$.
If $X\subset N$ is a subset, then we write $\cN_N(X)$ for the smallest normal subgroup of $N$ containing $X$ and call it the normaliser of $X$ inside $N$.
A normal subgroup $N\lhd G$ satisfies the \textit{decomposability property} if:
\begin{enumerate}
\item $N$ can be decomposed as a direct sum of two groups $N=A\oplus B$;
\item $N=\cN_G(A) = \cN_G(B)$.
\end{enumerate}
\end{definition}

Denote by $t_n, n\geq 1$ the tree with $2^n$ leaves all at distance $n$ from the root.
For example, $t_1=Y$ and $t_2=(Y\ot Y)\circ Y.$
Consider the permutations $\sigma=(21), \tau=(2134)$ (here we write $(a_1\cdots a_n)$ for the permutation $i\mapsto a_i$) and the elements $v_\sigma := \dfrac{\sigma t_1}{t_1}, v_\tau:=\dfrac{\tau t_2}{t_2}$. 
Note that in the first case we permute the two leaves of $t_1$ as in the second case we permute the two first leaves of $t_2$ and let invariant the two others.
The permutations $\sigma,\tau$ induce some permutation $\sigma_n,\tau_n$ on $2^n,n\geq 2$ elements that are 
$$\sigma_n(i) = i+2^{n-1} \mod 2^n$$ 
and 
$$\tau_n(i) = 
\begin{cases}
 i + 2^{n-2} \text{ if } 1\leq i\leq 2^{n-2} \\ 
 i - 2^{n-2}\text{ if } 2^{n-2}+1 \leq i \leq 2^{n-1} \\ 
 i \text{ if } 2^{n-1}+1 \leq i \leq 2^{n} \\ 
\end{cases}.
$$
More generally, for any permutation $\kappa\in S_{2^k},k\geq 1$ on $2^k$ elements and $n\geq k$ we can define a $2^{n-k}$-cable version $\kappa_n\in S_{2^n}$ of $\kappa$ that is:
$$\kappa_n(i + j 2^{n-k}) = i + \kappa(j) 2^{n-k} \text{ for } 1\leq i\leq 2^{n-k} \text{ and } 1\leq j\leq 2^{k}.$$
Identify the permutation $\kappa_n$ with the automorphism of the group $\Ga_{t_n}\simeq \Ga^{\oplus 2^n}$ that is permuting the coordinates.

\begin{definition}
Given any permutation $\kappa\in S_{2^k},k\geq 1$ and $n\geq k$ we consider the set 
$$X_{\kappa,n} = \{ g\kappa_n(g^{-1}) : \ g\in \Ga_{t_n}, \ \supp(g)\cap \kappa(\supp(g))=\emptyset \}$$
where $\supp(g)$ denotes the support of $g$, i.e.~the set of $1\leq i\leq 2^n$ for which the $i$-th component of $g$ is nontrivial.
\end{definition}

\begin{proposition}\label{prop:isom}
Consider a chain of nontrivial normal subgroups $L\lhd K\lhd G_\al$ and assume that $L$ is not contained inside $K_\al$. 
The following assertions are true.
\begin{enumerate}
\item For any permutation $\kappa\in S_{2^k}$ there exists $n_{\kappa,K}\geq 2$ such that if $n\geq n_{\kappa,K}$, then $X_{\kappa,n}$ is contained inside $K.$
\item Consider the following set $Y_n,n\geq 3$ of elements $y=(g, g^{-1} , e , e , g^{-1} , g , e , e )\in G_{t_n}$ for $g\in G_{ t_{n-3}}$ and where we identify $G_{t_n}$ with $G_{t_{n-3}}^8.$
There exists $n_{L,K}\geq 2$ such that if $n\geq n_{L,K}$, then $Y_n$ is contained inside $L.$
\item If $\ti K\lhd G_\al$ is a proper normal subgroup with the decomposability property, then $\ti K$ is contained inside $K_\al.$
\item The subgroup $K_\al\lhd G_\al$ is the unique maximal normal subgroup of $G_\al$ with the decomposability property.
In particular, it is a characteristic subgroup.
\end{enumerate}
\end{proposition}

\begin{proof}
Proof of (1). 
Since $K$ is nontrivial and not contained inside $K_\al$ there exists $a\in K_\al$ and a nontrivial element $v\in V$ such that $av\in K$.
Since $V$ is simple this implies that for any $w\in V$ there exists $b\in K_\al$ such that $bw\in \cN_{G_\al}(av)$ and thus $bw$ is in $K$ since it is a normal subgroup.
Consider a permutation $\kappa$ on $2^k$ elements with $k\geq 1$ and the associated element $v_\kappa\in V.$
For $n$ large enough there exists $a\in \Ga_{t_n}$ such that $av_\kappa$ belongs to $K.$
Given $x\in G_{t_n}$ we consider the element $xav_\kappa x^{-1}= xa\kappa_n (x)^{-1} v_\kappa$.
We obtain $\tilde a = xa\kappa_n (x)^{-1}$ with coordinate $\tilde a_i = x_i a_i x_{\kappa(i)}^{-1}.$
Choose $x$ such that $x_i=a_i^{-1}$ on a set of $i\in A$ such that $\kappa(A)\cap A=\emptyset$ and put $x_i=e$ outside of $A$.
We then obtain that $\tilde a_i=e$ for any $i\in A.$
Consider now $y\in G_{t_n}$ supported on such a set $A$ and observe that the commutator $[y,\tilde a v_\kappa] = [y,xav_\kappa x^{-1}]$ is equal to $x\kappa_n(x)^{-1}$ that is by definition in $K$.

Proof of (2).
A similar argument as above tells us that for any $v\in V$ there exists $a\in K_\al$ such that $av\in L.$
In particular, if $v=v_\sigma$ for the permutation $\sigma = (21)\in S_2$,  then there exists a large enough $n\geq 2$ such that $av_\sigma\in L$ and $a\in \Ga_{t_n}.$
By (1), we can choose a large enough $n$ such that $X_{\tau,n},X_{\kappa,n}\subset K$ where $\tau= (2134)\in S_4$ and $\kappa=(12346578)\in S_8.$
Write $a=(a_1,a_2,a_3,a_4)$ for an element of the group $G_{t_n}$ that we identify with the group $G_{t_{n-2}}^4$. 
Consider $x:=(a_1^{-1},a_1,e,e)$ that is in an element of $X_{\tau,n}$ and thus is in $K$.
Since $L$ is a normal subgroup of $K$ we have that $x av_\sigma x^{-1}$ is in $K$ that is $\tilde a v_\sigma$ with $\tilde a= xa\sigma_n(x)^{-1}.$
Note that $\tilde a$ is of the form $(e,\tilde a_2,\tilde a_3,\tilde a_4).$
Hence, we can assume that the first coordinate of $a\in G_{t_{n-2}}^4$ is trivial.
Now identify $G_{t_n}$ with $G_{t_{n-3}}^8$ and fix $g\in G_{t_{n-3}}$.
Define the element $x:= ( e , e , e , e , g^{-1} , g , e , e )$ and observe that it is in $X_{\kappa,n}$ and thus in $K$ by hypothesis on $n$.
Consider the commutator $[x,av_\sigma] = x a \sigma_n(x)^{-1} a^{-1}$ and observe that it is equal to 
$$( g , g^{-1} , e ,e , g^{-1}, g , e, e ).$$
This finishes the proof of (2).

Proof of (3).
Consider $\ti K\lhd G_\al$ such that $\ti K=A\oplus B$ and $\cN_{G_\al}(A)=\cN_{G_\al}(B)=\ti K.$
Assume that $\ti K$ is not contained inside $K_\al$. 
If $A$ or $B$ is contained inside $K_\al$, then so is its normaliser implying that $\ti K\subset K_\al$, a contradiction. 
Therefore, both $A$ and $B$ are not contained inside $K_\al.$
Since $A$ commutes with $B$ we obtain that $A$ is a normal subgroup of $\ti K$ and likewise for $B$.
By (2), applied to $A\lhd \ti K\lhd G_\al$ and $B\lhd \ti K\lhd G_\al$, there exists a large enough $n$ such that $Y_n\subset A$ and $Y_n\subset B$, a contradiction since $A\cap B=\{e\}.$

Proof of (4).
Consider $A_n\subset G_{t_n}$ (resp. $B_n$) the set of elements with support contained in the first (resp. the last) $2^{n-1}$ coordinates.
Note that $G_{t_n}=A_n\oplus B_n$ and that $\Phi_\al(f_n^{n+1})(A_n)\subset A_{n+1},  \Phi_\al(f_n^{n+1})(B_n)\subset B_{n+1}$ where $f_{n}^{n+1}$ is the forest with $2^n$ roots whose each tree is $Y$.
This implies that the set of fractions $A:=\{\dfrac{a}{t_n}: \ n\geq 1, a\in A_n\}$ forms a subgroup of $K_\al$ and that $K_\al=A\oplus B$.
Note that $v_\sigma A v_\sigma^{-1} =B$ where $\sigma=(21)\in S_2$, implying that the normaliser of $A$ inside $G_\al$ contains $K_\al$ and likewise for $B$.
Since $K_\al\lhd G_\al$ is a normal subgroup we obtain that $\cN_{G_\al}(A)=K_\al = \cN_{G_\al}(B).$
We obtain that $K_\al\lhd G_\al$ has the decomposability property.
By (3), any proper normal subgroup with this later property is contained inside $K_\al$ making it maximal.
The rest of the proposition is obvious. 
\end{proof}

Theorem \ref{theo:KinG} follows from the last point of the proposition.

\section{Description and classification of fraction groups}\label{sec:classification}

In this section we restrict the class of functors $\Phi$ considered.
We now assume that $\Phi:\cF\to \Gr$ is a covariant monoidal functor built from the data of a group $\Ga$ and a morphism of the following form:
$$\Phi(Y)(g):\Ga\to \Ga\oplus \Ga, \ g\mapsto (\al_0(g),e_\Ga).$$
We write $\al$ rather than $\al_0$ which can be any endomorphism of $\Ga.$
Recall that $e=e_\Ga$ is the neutral element of the group $\Ga$ (we may remove the subscript $\Ga$ if the context is clear).
The associated functors, limit group, Jones' action, fraction group and category are denoted as in the previous section by $\Phi_\al, K_\al, \pi_\al, G_\al, \cC_\al$, respectively.

We are going to fully classify the class of all such fraction groups $G_\al$ up to isomorphism.

\subsection{From an endomorphism to an automorphism}\label{sec:End}

In this subsection we reduce the study to $\al$ an automorphism. 
The idea is to build a group denoted by $\lim\Ga$ and extending $\al$ on $\lim\Ga$ as an automorphism $\lim\al\in\Aut(\lim\Ga)$. 
We will then show that the induced fraction groups $G_\al$ and $G_{\lim\al}$ are isomorphic justifying this reduction of study.

\begin{definition}
Consider a group $\Ga$ and any endomorphism $\al\in\End(\Ga)$.
We define the directed system of groups $(\Ga_n,n\geq 0)$ with the family of group morphisms $(\iota_n^{m}:\Ga_n\to\Ga_m, m\geq n)$ where 
$$\Ga_n:=\{(g,n):\ g\in\Ga\}$$ is a copy of $\Ga$ and 
$$\iota_n^{n+p}(g,n) := (\alpha^p(g),n+p) , \ n,p\geq 0, g\in\Ga.$$ 
We write $\lim\Ga:=\varinjlim_\alpha \Ga_n$ for the inductive limit of this directed system. 
Denote by $\sim$ the equivalence relation generated by $(g,n)\sim (\al(g),n+1),g\in\Ga,n\geq 0$ and write $[g,n]$ for the equivalence class of $(g,n)$ which is by definition an element of $\lim\Ga.$
We extend the endomorphism $\alpha$ of $\Ga$ into an endomorphism $\lim\al$ of $\lim\Ga$ as follows: 
$$(\lim\alpha)[g,n] := [\alpha(g),n] \text{ for all } n\geq 0, g\in \Ga.$$
We obtain two groups of fractions $G_\al=K_\al\rtimes V$ and $G_{\lim\al}=K_{\lim\al}\rtimes V$.
\end{definition}

\begin{lemma}
The map $\lim\al$ is a group automorphism of $\lim\Ga.$
\end{lemma}

\begin{proof}
Write $\beta_n:\Ga_n\to \Ga_{n}$ for all $n\geq 0$ defined as $\beta_n(g,n) = (\al(g),n).$
Consider $n,p\geq 0, g\in \Ga$ and note that 
\begin{align*}
\iota^{n+p}_n\circ \beta_n(g,n) & = \iota_n^{n+p}(\al(g),n) = (\al^{p+1}(g),n+p)\\
& = \beta_{n+p}(\al^p(g), n+p) = \beta_{n+p}\circ \iota_{n}^{n+p}(g,n).
\end{align*}
Therefore, the family $(\beta_n:\ n\geq 0)$ defines a map $\lim\al$ from $\lim\Ga$ to itself.
Moreover, $\lim\al$ is a group morphism because each $\beta_n,n\geq 0$ is.

Assume that $\lim\al[g,n]=e$ where $e$ is the neutral element of $\lim\Ga.$
Note that $[h,j]=e$ if and only if there exists $p$ large enough satisfying $\al^p(h)=e$ for $h\in\Ga,j\geq 0.$
Since $\lim\al[g,n]=[\al(g),n]$ we obtain that $\al^p(g)=e$ for $p$ large enough and thus $[g,n]=e$ implying that $\lim\al$ is injective.

Consider $[g,n]\in \lim\Ga$ with $n\geq 0$ that is equal to $[\al(g),n+1] = (\lim\al)[g,n+1]$ and thus belongs to the range of $\lim\al.$
This implies that $\lim\al$ is surjective and all together $\lim\al$ is an automorphism of the group $\lim\Ga.$
\end{proof}

\begin{proposition}\label{prop:isomaut}
Consider the morphism $\theta_0:\Ga\to\lim\Ga, g\mapsto [g,0]$. 
This induces the morphism
$$\theta_t:\Ga_t\to(\lim\Ga)_t, g=(g_\ell)_{\ell\in\Leaf(t)} \mapsto (\theta_0(g_\ell))_{\ell\in\Leaf(t)}$$
for any tree $t\in\fT.$
These maps are compatible with the two directed structures inducing a group isomorphism $$\theta:\varinjlim_{t\in\fT} \Ga_t\to \varinjlim_{t\in\fT}(\lim\Ga)_t.$$
This isomorphism is $V$-equivariant for the two Jones' actions and extends uniquely into a group isomorphism between the fraction groups $G_\al$ and $G_{\lim\al}.$
\end{proposition}

\begin{proof}
Let $\Phi:\cF\to\Gr$ and $\lim\Phi:\cF\to\Gr$ be the monoidal functors induced by $\al$ and $\lim\al$, respectively.

To prove that we have a directed system of maps $\theta_t$ it is sufficient to check that:
$$\theta_{ft}\circ \Phi(f) = (\lim\Phi)(f)\circ \theta_{t} \text{ for all } t\in\fT, f\in\Hom(\cF).$$
This later equality is a consequence of the following: $\lim\al\circ \theta_0 = \theta_0\circ \al.$
Therefore, $\theta$ is well-defined and is a group morphism as a limit of group morphisms.

Consider $(g,t)\in \varinjlim_{t\in\fT} \Ga_t$ such that $\theta_t(g,t)=e.$
We have that $g=(g_\ell)_{\ell\in\Leaf(t)}$ with $g_\ell\in\Ga.$
If $\theta_t(g,t)=e$, then $\theta_0(g_\ell)=e$ for any $\ell\in \Leaf(t)$. 
That is, there exists a power $N_\ell\geq 1$ such that $\al^{N_\ell}(g_\ell)=e$ by definition of the limit group $\lim\Ga.$
Put $N:=\max(N_\ell : \ell\in\Leaf(t))$ and consider a forest $f$ that is composable with $t$ and satisfying that each of its leaves is at distance $N$ from the root in its connected component.
We get that $\Phi(f)(g)$ has all of its components equal to $e$ or $\al^{N}(g_\ell)=e$ since $N\geq N_\ell$.
Since $(g,t)\sim (\Phi(f)(g),ft) =(e,ft)$ inside the limit group $\varinjlim_{t\in\fT}\Ga_t$ we obtain that $(g,t)$ is trivial and thus $\theta$ is injective.

Consider $g\in \varinjlim_{t\in\fT} (\lim\Ga)_t$.
We can assume, up to identifying classes with representatives, that $g=(g,t)\in (\lim\Ga)_t$ for a certain tree $t$ and thus can be written as $g=(g_\ell)_{\ell\in\Leaf(t)}$ with $g_\ell\in\lim\Ga.$
Since $\lim\Ga$ is the direct limit of the family of groups $(\Ga_n:\ n\geq 0)$ and $\Leaf(t)$ is finite we can assume that there exists a large enough $n\geq 0$ such that $g_\ell=[x_\ell,n]$ with $x_\ell\in\Ga$ for all $\ell\in\Leaf(t).$
Consider a forest $f$ composable with $t$ for which each of its leaves is at distance $n$ from the root in its connected component.
We obtain that $(g,t)\sim ( (\lim\Phi)(f)(g),ft)$ and note that every component of $(\lim\Phi)(f)(g)$ is either trivial or of the form $[\al^n(x_\ell),n]\in\lim\Ga$.
Observe that $[\al^n(x_\ell),n] = [x_\ell,0]$ inside $\lim\Ga$ and thus belongs to the range of $\theta_0$.
We obtain that $g$ is in the range of $\theta_{ft}$ implying that $\theta$ is onto.
\end{proof}

Here is a simple and explicit example which describes well the construction of above. For pedagogical reason we have chosen an example where the endomorphism $\al$ is injective.

\begin{example}
Consider the group $\Ga:=\Z$, a natural number $q\geq 2$ and the non-surjective endomorphism 
$$\al:\Z\to\Z, z\mapsto qz.$$ 
We are going to construct $(\lim \Ga, \lim\al)$ and construct an explicit isomorphism $j:\lim\Ga\to \Z[1/q]$ conjugating $\lim\al$ with $x\in\Z[1/q]\mapsto qx\in\Z[1/q].$
To do this we define a family of groups $\Ga_n$ and injective morphisms $j_n:\Ga_n\to \Z[1/q]$ and take the limit in $n$.
Note that $j$ and $j_n$ have no equivalent in the proof of above. We used them in order to exhibit a particularly nice description of the pair $(\lim\Ga,\lim\al)$.

For each $n\geq 0$ define $\Ga_n:=\{(z,n):\ z\in \Z\}$ which is a copy of the group $\Z$.
Consider the morphism 
$$j_n:\Ga_n\to \Z[1/q], \ (z,n)\mapsto \dfrac{z}{q^n} \text{ for all } n\geq 0.$$
Observe that $$j_{n+p}\circ \iota_n^{n+p}(z,n) = j_{n+p}( q^pz, n+p) = \dfrac{q^p z}{q^{n+p}} = \dfrac{z}{q^n} = j_n(z,n) \text{ for all } z\in \Z, n,p\geq 0.$$
Therefore, the family of morphisms $(j_n:\ n\geq 0)$ defines a morphism
$$j: \lim\Ga\to \Z[1/q],\ [z,n]\mapsto \dfrac{z}{q^n}.$$ 
Since all the $j_n,\ n\geq 0$ are injective so is $j$. 
It is not hard to see that $j$ is surjective implying that $j$ is a group isomorphism from $\lim \Ga$ to $\Z[1/q].$
Now, define the morphisms
$$\beta_n:\Ga_n\to\Ga_n,\ (z,n)\mapsto (qz,n), \ n\geq 0.$$
We have that 
\begin{align*}
\iota_n^{n+p}\circ \beta_n(z,n) & = \iota_n^{n+p}(qz,n) = (\dfrac{qz}{q^{n+p}}, n+p) = ( q \cdot \dfrac{z}{q^{n+p}}, n+p)\\
& = \beta_{n+p}(\dfrac{z}{q^{n+p}}, n+p ) =\beta_{n+p}\circ \iota_{n}^{n+p}(z,n)
\end{align*}
for all $z\in\Z,n,p\geq 0.$
Hence, we can define a limit endomorphism 
$$\lim\al: \lim\Ga\to\lim\Ga,\ [z,n]\mapsto [\al(z),n]=[qz,n].$$
Moreover, we observe that:
$$(j\circ \lim\al) [z,n] = j([qz,n]) = \dfrac{qz}{q^n} = q\cdot \dfrac{z}{q^n} = q\cdot j( [z,n]) \text{ for all } z\in\Z,n\geq 0.$$
All together, we deduce that $\lim\Ga$ is isomorphic to $\Z[1/q]$ and that $\lim\al$ is conjugated with the automorphism:
$$\Z[1/q]\to\Z[1/q], \ x\mapsto qx.$$ 
\end{example}

\subsection{Description of the fraction groups}\label{sec:description}
In this subsection we provide a description of $G_\al$ in terms of a twisted permutational restricted wreath product.

\begin{proposition}\label{prop:twistWP}
Consider a group $\Ga$, an \textit{automorphism} $\alpha\in\Aut(\Ga)$ and the associated fraction group $G_\al=K_\al\rtimes V$.
Let $\oplus_{\Q_2}\Ga$ be the direct sum of $\Ga$ over the dyadic rationals $\Q_2$ and define the action 
$$\beta:V\act \oplus_{\Q_2}\Ga, \
\beta_v(a)(x):= \al^{\log_2(v'(v^{-1}x)) } \left( a(v^{-1} x) \right) , \ v\in V, a\in \oplus_{\Q_2}\Ga, x\in \Q_2.$$
Write $\oplus_{\Q_2}\Ga\rtimes_\alpha V$ for the associated wreath product.
\begin{itemize}
\item The groups $G_\al$ and $\oplus_{\Q_2}\Ga\rtimes_\alpha V$ are isomorphic.

\item Here is an explicit isomorphism: For $t\in\fT$ and $\ell\in\Leaf(t)$ we write $N_\ell^t$ and $I_\ell^t:=[r_\ell^t, s_\ell^t)$ for the distance between the root of $t$ and its leaf $\ell$ and for the sdi associated to $\ell$, respectively.
The map $$\theta_t:\Ga_t\to \oplus_{\Q_2}\Ga$$
defined as
$$\theta_t(g):\Q_2\to \Ga,\ \theta_t(g)(x)=\begin{cases}
\al^{-N_\ell^t} (g_\ell) \text{ if } x=r_\ell^t, \ \ell\in\Leaf(t) \\
0 \text{ otherwise}
\end{cases}$$
for $g=(g_\ell:\ \ell\in \Leaf(t))\in \Ga_t$ and $x\in\Q_2$ is an injective morphism.
The family $(\theta_t:\ t\in\fT)$ defines a group isomorphism from $K_\al=\varinjlim_{t\in\fT}\Ga_t$ to $\oplus_{\Q_2}\Ga$ which is $V$-equivariant and thus extends uniquely into an isomorphism from $G_\al$ to $\oplus_{\Q_2}\Ga\rtimes_\al V.$
\end{itemize}
\end{proposition}

\begin{proof}
This proposition was already observed in \cite[Section 2.3]{Brothier19WP}.
We provide a proof for the convenience of the reader.
Let $(\Ga_t, \iota_{s,t}:\ s,t\in \fT, s\geq t)$ be the directed system of groups associated to the functor $\Phi_\al$ built from $\al.$
Consider a tree $t$ and $\theta_t$ as described above.
It is clearly valued in $\oplus_{\Q_2}\Ga$ since $|\supp(\theta_t(g))|\leq |\Leaf(t)|$ for all $g\in \Ga_t.$
It defines a group morphism from $\Ga_t$ to $\oplus_{\Q_2}\Ga$ and in fact an isomorphism from $\Ga_t$ to $\oplus_{L_t}\Ga$ where $L_t:=\{ r^t_\ell:\ \ell\in \Leaf(t)\}.$
Consider a forest $f$ that is composable with $t$.
For each leaf $\ell$ of $t$ is associated a root $R_\ell$ of $f$ and a geodesic path starting at $R_\ell$ going upward with only left edges ending at $S_\ell$ which is a leaf of $f.$
If $N^f_\ell$ is the length of this path we obtain that $\iota_{ft,t}(g)$ is supported on $\{R_\ell:\ \ell\in\Leaf(t)\}$ such that 
$$\iota_{ft,t}(g)(S_\ell)= \al^{N_\ell^f}(g_\ell), \ell\in \Leaf(t).$$
We deduce that $\theta_{ft}\circ\iota_{ft,t} = \theta_t$ and thus there exists a unique group morphism $\theta:\varinjlim_{t\in\fT}\Ga_t\to \oplus_{\Q_2}\Ga$ satisfying that $\theta$ restricts to $\theta_t$ on $\Ga_t$ for all $t\in\fT.$

It remains to show that $\theta$ is an isomorphism and is $V$-equivariant.
For each $t\in\fT$ the morphism $\theta$ restricts into an isomorphism from $\Ga_t$ onto $\oplus_{L_t}\Ga_t$.
In particular, $\theta$ is injective.
Moreover, since any dyadic rational of $\Q_2$ is the first point of a sdi we have that $\cup_{t\in\fT} L_t=\Q_2$ and thus $\cup_{t\in\fT} \oplus_{L_t}\Ga = \oplus_{\Q_2}\Ga$ implying that $\theta$ is surjective.
We have proven that $\theta$ is an isomorphism.

Let us show that $\theta$ is $V$-equivariant.
Consider $g\in K_\al$ and $v\in V$.
We can assume that $g=(g_\ell)_{\ell\in\Leaf(t)}$ is in $\Ga_t$ for some tree $t\in\fT$. 
Moreover, taking $t$ large enough we can assume that $v$ is adapted to the sdp $(I_\ell^t:\ \ell\in\Leaf(t))$ of $t$ sending this sdp to the sdp $(J_\ell^s:\ \ell\in\Leaf(s))$ (here we implicitly identify the leaves of $t$ and $s$).
Consider the Jones action $\pi:V\act K_\al$ and note that $\pi_v(g)$ has for representative $(g_\ell)_{\ell\in\Leaf(s)}$ but as an element of $\Ga_s$ inside $K_\al.$
We obtain that $\theta(\pi_v(g))$ is supported in $\{r_\ell^s:\ell\in\Leaf(s)\}$ satisfying 
$$\theta(\pi_v(g))(v r_\ell^t) = \theta(\pi_v(g))(r_\ell^s) = \al^{-N_\ell^s}(g_\ell) = \al^{N_\ell^t-N_\ell^s}(\theta(g)(r_\ell^t)).$$
We conclude by observing that the slope of $v$ restricted to $I_\ell^t$ is equal to $\dfrac{|I_\ell^s|}{|I_\ell^t|} = 2^{N_\ell^t-N_\ell^s}$ implying that $\beta_v(\theta(g))(vr_\ell^t)=\theta(\pi_v(g))(vr_\ell^t)$ for all $\ell$ and thus $\beta_v(\theta(g))=\theta(\pi_v(g)).$
\end{proof}

\begin{remark}
The isomorphism defined in the previous proposition is very convenient to work with. 
Although, it may not be the most obvious one.
Perhaps, the most natural way to consider $G_\al$ as a wreath product is to replace $N_\ell^t$ by the number $M_\ell^t$ of \textit{left} edges between the root and the leaf $\ell$ rather than the number of all edges.
This provides an isomorphism from $K_\al$ to $\oplus_{\Q_2}\Ga$.
However, to make it $V$-equivariant we then consider the action 
$$\ga_v(a)(x) = \al^{M_x^v}(a(v^{-1}x))$$
rather than $\beta_v$ where $M_x^v:=M_{v^{-1}I}-M_{I}$ for the choice of a sdi $I$ adapted to $v^{-1}$ and containing $x$ and where $M_I$ is equal to the number of \textit{left} edges to go from $[0,1)$ to $I$ inside the infinite rooted binary $t_\infty$ as described in \eqref{sec:forest}. 
Note that this formula does not depend on the choice of $I$.
\end{remark}

\subsection{Thin classification of fraction groups}

In this section we will decide when two fraction groups are isomorphic or not.
We consider some groups $\Ga,\ti\Ga$, automorphisms $\al\in\Aut(\Ga),\ti\al\in \Aut(\ti\Ga)$ and the associated fraction groups $G_\al=K_\al\rtimes V, G_{\ti\al}=K_{\ti\al}\rtimes V$ that we often write $G=K\rtimes V$ and $\ti G=\ti K\rtimes V$, respectively.
Moreover, we identify the fraction groups with the corresponding wreath products described in the previous subsection.
We start by constructing some elementary isomorphisms.

\begin{lemma}\label{lem:isomone}
Consider two isomorphic group $\Ga,\ti\Ga$ and $\beta\in\Isom(\Ga,\ti\Ga), \al\in\Aut(\Ga).$
Then the fraction groups $G=K\rtimes V$ and $\ti G = \ti K \rtimes V$ constructed from $\al$ and $\ti\al:= \beta \al \beta^{-1}$ are isomorphic.
\end{lemma}

\begin{proof}
Consider the isomorphism $\kappa$ from $\prod_{\Q_2}\Ga$ to $\prod_{\Q_2}\ti \Ga$ defined as $\kappa(a)(x) = \beta(a(x))$ for all $a\in \prod_{\Q_2}\Ga$ and $x\in\Q_2.$
Observe that $\supp(\kappa(a)) = \supp(a)$ for any maps $a$ implying that $\kappa$ sends finitely supported maps to finitely supported ones and thus $\kappa$ restricts to an isomorphism from $K$ to $\ti K$.
Let us check that $\kappa$ is $V$-equivariant. Denote by $\pi:V\act K$ and $\ti\pi:V\act \ti K$ the Jones actions.
We have that 
\begin{align*}
\kappa(\pi_v(a))(vx) & = \beta(\pi_v(a)(vx)) = \beta(\al^{\log_2(v'(x))}(a(x)))\\
& = (\beta\al \beta^{-1})^{\log_2(v'(x))}(\beta(a(x)) \\
& = {\ti \al}^{\log_2(v'(x))}(\beta(a(x))\\
& = \ti\pi_v(\kappa(a))(vx), \ a\in K, v\in V, x\in \Q_2.
\end{align*}
Therefore, the isomorphism $\kappa:K\to\ti K$ extends to an isomorphism $\theta:G\to\ti G$ such that $\theta(av)=\kappa(a) v$ for any $a\in K, v\in V.$
\end{proof}

The construction of the following isomorphism is less trivial than the last one and uses the dyadic valuation. 

\begin{notation}\label{not:valuation}
Let $\nu:\Q\to \Z$ be the dyadic valuation such that $\nu(0)=0$ and $\nu(\prod_{p: \text{ prime }} p^{n_p}) = n_2$ for any finitely supported maps $p\mapsto n_p\in\Z$.
\end{notation}

\begin{lemma}\label{lem:isomtwo}
Consider a group $\Ga$, an automorphism $\al\in\Aut(\Ga)$ and the associated fraction group that we denote here by $G=K\rtimes_\al V$.
Given $k\in \Ga$ we consider the automorphism $\ti\al:=\ad(k)\circ\al\in\Aut(\Ga)$ and the associated fraction group that we denote by $\ti G=K\rtimes_{\ti\al}V$.

We have that $G\simeq \ti G.$
\end{lemma}

\begin{proof}
The idea of the proof is to construct two isomorphisms $\ad(f)$ and $av\mapsto a\cdot c_v\cdot v$ of the \textit{unrestricted} wreath products. 
Then we show that that the combined isomorphism $av\mapsto \ad(f)(a\cdot c_v\cdot v)$ restricts into an isomorphism between the \textit{restricted} wreath products.

We start by constructing a family $(k_n,n\in\Z)$ of $\Ga$ satisfying $\ti\al^n=\ad(k_n)\al^n$ for all $n\in\Z.$
Such a family is defined by induction as follows:
$$\begin{cases} k_0=e\\ 
k_{n+1} = k_n\al^n(k) \text{ for } n\geq 0\\
k_{-(m+1)} = k_{-m} \al^{-(m+1)}(k^{-1}) \text{ for } m\geq 0\end{cases}.$$
We will be using the following equation:
\begin{equation}\label{eq:k}
k_n \al^n(k_m) = k_{n+m} \text{ for any } n,m\in\Z.\end{equation}

Consider the dyadic valuation $\nu:\Q\to\Z$ defined in \eqref{not:valuation}.

{\bf Claim:} For any $v\in V$ there exists a finite subset $F_v\subset \Q_2$ such that 
$$\log_2(v'(x)) = \nu(vx) - \nu(x) \text{ for any } x\in \Q_2\setminus F_v.$$

Fix $v\in V$ and note that there exists a sdp $P$ (finite by definition) such that $v$ restricted to each sdi of $P$ is the composition of an affine map with slope a power of $2$ and a translation by a dyadic rational.
Moreover, the property of the claim is closed under composition and taking inverses. 
Therefore, it is sufficient to check this property individually for the functions $f:x\mapsto 2 x$ and $\tau:x\mapsto x+\dfrac{1}{2^m}$ with $m\geq 1$ with domain $[0,1]$.
The map $x\mapsto \nu(2x)-\nu(x)$ is constant equal to $1$ on $(0,1]$ and $x\mapsto \nu(x+1/2^m) - \nu(x)$ has its support contained in $\{n/2^m: 0\leq n\leq 2^m\}$ if $x$ is restricted to $[0,1]$.
In both cases we obtain the equality $\log_2(f'(x)) = \nu(fx)-\nu(x)$ and $\log_2(\tau'(x))=\nu(\tau x)-\nu(x)$ for all but finitely many $x\in \Q_2.$

Consider $\ov K:=\prod_{\Q_2} \Ga$ the group of all maps from $\Q_2$ to $\Ga.$
Using the description of $G,\ti G$ as wreath products we can easily extend the Jones actions $\pi:V\act K$ and $\ti\pi:V\act K$ induced by $\al$ and $\ti\al$ respectively into actions of $V$ on $\ov K$.
We continue to denote by $\pi,\ti\pi:V\act \ov K$ these extensions and write $\ov K\rtimes_\al V$ and $\ov K\rtimes_{\ti\al} V$ for the corresponding semidirect products.

{\bf Claim:} We have an isomorphism $$\theta:\ov K\rtimes_\al V\to \ov K\rtimes_{\ti\al} V$$ defined as 
$$\theta(av) = \ad(f) (a\cdot c_v \cdot v) = f\cdot a \cdot c_v \cdot v \cdot f^{-1}, \ a\in \ov K , v\in V$$
where we define
$$f(x):=k_{\nu(x)}, \ c_v(vx) :=(k_{\log_2(v'(x))})^{-1} \text{ for all } x\in \Q_2, v\in V.$$

Since $\ad(f)$ is an automorphism of $\ov K\rtimes_{\ti\al} V$ it is sufficient to show that
$$\rho:av\mapsto a\cdot c_v \cdot v$$ defines an isomorphism from $\ov K\rtimes_\al V$ to $\ov K\rtimes_{\ti\al} V.$
The map $\rho$ is multiplicative if and only if 
$$\pi_v(b)c_{vw} = c_v \ti\pi_v(b c_w) \text{ for all } b\in \ov K, v,w\in W.$$

Fix $x\in\Q_2$ and write $n:=\log_2(v'(wx))$ and $m:=\log_2(w'(x)).$
We have that
$$[\pi_v(b)c_{vw}] (vwx) = \al^n(b(wx)) \cdot k_{\log_2((vw)'(x))}^{-1} =  \al^n(b(wx)) \cdot k_{n+m}^{-1}$$
using the chain rule $(vw)'(x) = v'(wx)\cdot w(x)$ which is valid for elements of $V$.
Now,
\begin{align*}
[c_v \ti\pi_v(b c_w)](vwx) & = k_n^{-1} \cdot {\ti\al}^n( b(wx) c_w(wx)) = k_n^{-1}\ad(k_n)\al^n\left( b(wx) \cdot k_m^{-1}\right)\\
& = \al^n(b(wx)) \cdot \al^n(k_m^{-1}) k_n^{-1}\\
& = \al^n(b(wx))\cdot k_{m+n}^{-1}
\end{align*}
by Equation \eqref{eq:k}.
This proves that $\rho$ is multiplicative.

It is then a group morphism.
A similar proof shows that the formula $av\mapsto a\cdot c_v^{-1}\cdot  v$ defines a group morphism from $\ov K\rtimes_{\ti\al} V$ to $\ov K\rtimes_\al V$ which is an inverse of $\rho$ implying that $\rho$ is an isomorphism.

{\bf Claim:} The isomorphism $\theta$ restricts to an isomorphism from $K\rtimes_\al V$ onto $K\rtimes_{\ti\al} V.$

Observe that 
$$\theta(av) = (faf^{-1})\cdot (f c_v \ti\pi_v(f^{-1}))\cdot v, a\in \ov K,v\in V$$ and that $\supp(faf^{-1})\subset \supp(a).$
Therefore, it is sufficient to check that $f c_v \ti\pi_v(f^{-1})$ is finitely supported for any $v\in V$.
Fix $v\in V,x\in \Q_2\setminus F_v$ and write $n:=\log_2(v'(x)).$
We have that:
\begin{align*}
(f c_v \ti\pi_v(f^{-1}))(vx) & = f(vx) c_v(vx) k_n \al^n( f(x)^{-1}) k_n^{-1}\\
& = k_{\nu(vx)} \al^n(k_{\nu(x)}^{-1}) k_n^{-1}\\
& = k_{\nu(vx)} [k_n\al^n(k_{\nu(x)}) ]^{-1} \\
& = k_{\nu(vx)} (k_{n+\nu(x)})^{-1}\\
& = e,
\end{align*}
by definition of $F_v$ and Equation \ref{eq:k}.
Therefore, the support of $(f c_v \ti\pi_v(f^{-1}))$ is contained in the finite set $F_v.$
This implies that $\theta$ maps $K\rtimes_\al V$ inside $K\rtimes_{\ti\al} V$. 
A similar proof shows that $av\mapsto \ad(f^{-1})( a [f^v c_v^{-1} (f^v)^{-1}] v)$ defines a group morphism from $K\rtimes_{\ti\al}V$ to $K\rtimes_\al V$ which is an inverse to $\theta$ implying that $\theta$ restricts into an isomorphism from $G$ to $\ti G$.
\end{proof}

We now prove that an isomorphism $\kappa:K\to \ti K$ that is spatial can be decomposed as a product of isomorphisms. Recall that $K=\oplus_{\Q_2}\Ga$ and $\ti K=\oplus_{\Q_2}\ti\Ga$ where $\Ga,\ti\Ga$ are groups.

\begin{lemma}\label{lem:decomposition}
Let $\kappa:K\to \ti K$ be an isomorphism satisfying that $\supp(\kappa(a))=\varphi(\supp(a))$ for a certain $\varphi \in\SNQ)$. 
Then there exists a unique family of isomorphisms $(\kappa_x:K\to \ti K, x\in \Q_2)$ satisfying that
$$\kappa(a)(\varphi(x)) = \kappa_x(a(x)) \text{ for all } a\in K , x\in\Q_2.$$
\end{lemma}
\begin{proof}
Consider $g\in \Ga,x\in \Q_2$ and $g_x\in K$ the element supported in $\{x\}$ taking the value $g$ at $x$.
We have that $\kappa(g_x)$ has support equal to $\{\varphi(x)\}$ if $g\neq e$ and thus there exists $\kappa_x(g)\in \ti\Ga$ satisfying that $\kappa(g_x) = [\kappa_x(g)]_{\varphi(x)}$, the map supported in the singleton $\{\varphi(x)\}$ taking the value $\kappa_x(g)$ at $\varphi(x).$
Using that $\kappa$ is a group morphism we obtain that 
\begin{align*}
[\kappa_x(gh)]_{\varphi(x)} & =\kappa((gh)_x) = \kappa(g_x \cdot h_x) = \kappa(g_x)\cdot \kappa(h_x)\\
& = [\kappa_x(g)]_{\varphi(x)}\cdot [\kappa_x(h)]_{\varphi(x)} = [\kappa_x(g)\cdot\kappa_x(h)]_{\varphi(x)}
\end{align*}
for $g,h\in \Ga$ implying that $\kappa_x$ is a group morphism.
Therefore, $(\kappa_x,\ x\in \Q_2)$ is a family of group morphisms implying that 
$$\prod_{x\in\Q_2}\kappa_x: K\to \ti K, a\mapsto (\varphi(x)\mapsto \kappa_x(a(x)))$$ is a group morphism.
This latter morphism coincide with $\kappa$ on  $\{g_x:\ g\in\Ga,x\in\Q_2\}$ but since this set is generating $K$ we obtain that $\kappa=\prod_{x\in\Q_2}\kappa_x$.
It is rather obvious to see that $a\mapsto (\varphi(x)\mapsto \kappa_x(a(x)))$ is an isomorphism if and only if each $\kappa_x$ is an isomorphism which finishes the proof.
\end{proof}
 
Before proving the main theorem of this section we prove the following surprising rigidity fact: any isomorphism between fraction groups of the class considered is \textit{spatial} in the sense described below.

\begin{proposition}\label{prop:supportone}
Consider two groups $\Ga,\ti\Ga$ with $\Ga$ nontrivial and the associated fraction groups $G:=K\rtimes V, \ti G:=\ti K\rtimes V$ where $K=\oplus_{\Q_2}\Ga, \ti K=\oplus_{\Q_2}\ti\Ga$.
Assume we have an isomorphism $\theta:G\to \ti G.$

Then $\theta$ restricts to an isomorphism $\kappa: K\to \ti K$.
Moreover, there exists a unique homeomorphism $\varphi$ of Cantor space, normalising $V$ and stabilising $\Q_2$ (i.e.~$\varphi\in\SNQ$) satisfying
$$\supp(\kappa(a)) =\varphi(\supp(a)) \text{ for all } a\in K.$$
In particular, there exists a unique family of isomorphisms $(\kappa_x:\Ga\to \ti\Ga,\ x\in \Q_2)$ satisfying that 
$$\kappa(a)(\varphi(x)) = \kappa_x(a(x)) \text{ for all } a\in K, x\in \Q_2.$$
\end{proposition}

\begin{proof}
Consider $\theta:G\to\ti G$ as above. 
Theorem \ref{theo:KinG} implies that $\theta$ restricts to an isomorphism $\kappa$ from $K$ to $\ti K$. 
Therefore, 
$$\theta(av)=\kappa(a) \cdot c_v \cdot \phi_v,a\in K,v\in V$$ where 
$$\kappa\in\Isom(K,\ti K),\phi\in\Aut(V) \text{ and } c:V\to K, v\mapsto c_v.$$
By Rubin's theorem we have that $\phi=\ad_\varphi$ for a unique $\varphi\in\NCV.$
Define
$$W_z:=\{ v\in V:\ vz = z \text{ and } v'(z)=1\}$$
for any $z\in\fC$. 
Note that $v\in W_z$ if and only if there exists a sdi $I$ containing $z$ on which $v$ acts like the identity.
This implies that $\ad_\varphi:v\mapsto \varphi v \varphi^{-1}$ restricts to an isomorphism from $W_z$ to $W_{\varphi(z)}$, i.e.~$\phi(W_z)=W_{\varphi(z)}.$
Since $\Ga$ is nontrivial there exists $g\in\Ga, g\neq e.$
For $x\in\Q_2$ we write $g_x\in K$ for the map supported at $\{x\}$ taking the value $g$.
Observe that if $v\in W_x$, then 
\begin{align*}
\theta(vg_xv^{-1}) & = \theta( [\al^{\log_2(v'(x))}(g)]_{vx} ) =\kappa(g_x)\\
& = \theta(v) \theta(g_x) \theta(v)^{-1}\\
& = c_v\cdot \phi_v \kappa(g_x) \phi_v^{-1} \cdot c_v^{-1}\\
& = \ad(c_v) \ti\pi_{\phi_v}(\kappa(g_x)).
\end{align*}
We deduce that $\supp(\kappa(g_x)) = \phi_v(\supp(\kappa(g_x))$. 
Therefore, the group $\phi(W_x)=W_{\varphi(x)}$ is a subgroup of the stabiliser subgroup
$$\Stab_V(\supp(\kappa(g_x))):=\{ w\in V:\ w\cdot \supp(\kappa(g_x)) = \supp(\kappa(g_x))\}.$$
This implies that $\supp(\kappa(g_x))$ is equal to the singleton $\{\varphi(x)\}.$
Indeed, assume that there exists $s\in \supp(\kappa(g_x)),s\neq\varphi(x)$.
Since $\supp(\kappa(g_x))$ is a finite subset of $\Q_2$ we can find a sdi $I$ such that $\varphi(x)\notin I$ and $I\cap \supp(\kappa(g_x))= \{s\}.$
Let $I_0$ and $I_1$ be the first and second half of $I$ and consider $v\in V$ permuting $I_0$ with $I_1$ and fixing all other elements of $\fC$.
In particular, $v\in W_{\varphi(x)}$ and thus $v$ stabilises $\supp(\kappa(g_x)).$
However, by definition of $v$ we have that $vs\notin \supp(\kappa(g_x))$, a contradiction.
Since $\kappa$ is injective and $g\neq e$ we have that $\supp(\kappa(g_x))$ has at least one point and this point must be $\varphi(x).$
In particular, $\varphi$ stabilises $\Q_2.$
By observing that any $a\in K$ is a finite product of some $g_x$ as above we obtain that $\supp(\kappa(a))=\varphi(\supp(a)).$

The second statement of the proposition is given by Lemma \ref{lem:decomposition}.
\end{proof}

We are now able to prove the main theorem of this section which shows that the only isomorphic pairs of fraction groups come from the last two lemmata.

\begin{theorem}\label{th:isomGalpha}
Consider two groups with automorphisms $(\Ga,\al\in\Aut(\Ga))$ and $(\ti\Ga,\ti\al\in\Aut(\ti\Ga))$ and their associated fraction groups $G:= K\rtimes V$ and $\ti G:=\ti K\rtimes V$, respectively.

The groups $G$ and $\ti G$ are isomorphic if and only if there exists $\beta\in\Isom(\Ga,\ti\Ga)$ and $h\in\ti\Ga$ such that $\ti\al = \ad(h)\circ\beta \al\beta^{-1}.$
\end{theorem}

\begin{proof}
The statement is trivially true if $\Ga$ and $\ti\Ga$ are the trivial groups. We now assume that $\Ga$ is nontrivial.

Consider an isomorphism $\theta:G\to \ti G$ with $G,\ti G$ as above.
By Proposition \ref{prop:isom}.4 we have that $\theta(K)=\ti K$.
Therefore, we have:
$$\theta(av)=\kappa(a) \cdot c_v \cdot \phi_v \text{ for all } a\in K, v\in V$$ 
for some $\kappa\in\Isom(K,\ti K)$, $\phi=\ad_\varphi\in\Aut(V)$, $c:V\to\ti K$ where $\varphi\in \NCV$.
Moreover, Proposition \ref{prop:supportone} implies that $\varphi$ stabilises $\Q_2$ and that $\kappa$ can be written as a product of isomorphisms: there exists a family $(\kappa_x:\ x\in \Q_2)$ such that 
$$\kappa(a)(\varphi(x)) = \kappa_x(a(x)) \text{ for all } x\in\Q_2, a\in K.$$
In particular, $\Ga$ is isomorphic to $\ti\Ga$ (via any of the $\kappa_x,x\in\Q_2$).

Let us show now that $\ti\al=\ad(h)\circ \beta \al \beta^{-1}$ for suitable $\beta\in\Isom(\Ga,\ti\Ga),h\in\ti\Ga.$
Fix $x\in\Q_2$ and $v\in V$ such that $vx=x$ and $v'(x)=2.$
Note that $\phi_v(\varphi(x))=\varphi(x)$ and $\phi_v'(\varphi(x))=2$ in virtue of Proposition \ref{prop:slope}.
If $g\in\Ga$ and $g_x$ is the map supported in $\{x\}$ taking the value $g$ we obtain that 
$$\theta(vg_xv^{-1}) = \kappa( [\al(g)]_x )= c_v \cdot \phi_v\cdot  \kappa(g_x) \cdot \phi_v^{-1} \cdot c_v^{-1}= \ad(c_v)( \ti\pi_{\phi_v}(\kappa(g_x))).$$
Note that $\kappa(g_x)$ is supported in $\{\varphi(x)\}$ and thus $\ti\pi_{\phi_v}(\kappa(g_x))$ is also supported in $\{\varphi(x)\}$ since $\phi_v(\varphi(x))=\varphi(x).$
Moreover, $\phi_v'(\varphi(x))=2$ implying that $\ti\pi_{\phi_v}(a)(\varphi(x)) = \ti\al(a(\varphi(x)))$ for all $a\in \ti K.$
If we evaluate the equality of above at $\varphi(x)$ we obtain:
$$\kappa_x(\al(g)) = \ad(c_v(\varphi(x))( \ti\al(\kappa_x(g))).$$
Since $g$ was arbitrary we deduce the equality:
$$\kappa_x \circ \al = \ad(c_v(\varphi(x)) \circ \ti\al \circ\kappa_x.$$
In particular, 
$$\ti\al= \ad(h)\circ \beta \al \beta^{-1}$$
where $\beta=\kappa_x$ and $h=c_v(\varphi(x))^{-1}.$

The converse is given by Lemmata \ref{lem:isomone} and \ref{lem:isomtwo}.
\end{proof}

\begin{remark}\label{rem:AWP}
Note that it is possible to follow Neumann's original proof for restricted wreath products in order to obtain that $\Ga$ is isomorphic to $\ti\Ga$ \cite{Neumann64}. However, the proof would be rather indirect and would provide a less precise statement regarding the relation between the automorphisms $\al$ and $\ti\al$. 
Our proof takes advantage of the highly transitive action of $V\act \Q_2$.

The fact that all isomorphisms between two groups in our class are spatial is very surprising and does not hold in general even for restricted standard wreath products. 
Moreover, one cannot expect to generalise Neumann's theorem for the class of all restricted permutational and twisted wreath products constructed from transitive actions. 
We illustrate these remarks with specific examples.

We start by constructing non-spatial automorphisms.
Consider a nontrivial group $\Ga.$
Let $G:=\Ga^2\wr \Z=\oplus_{n\in\Z}\Ga^2\rtimes \Z$ be the standard restricted wreath product of $\Ga^2=\Ga\oplus\Ga$ with $\Z$.
An element of $\oplus_{n\in\Z}\Ga^2$ is a finitely supported map $f:\Z\to \Ga^2, n\mapsto (f(n)_0, f(n)_1)$. 
Now we re-index them as maps from $\frac{1}{2}\Z$ to $\Ga$ using the isomorphism:
$$j: \oplus_{n\in\Z}\Ga^2\to \oplus_{n\in \frac{1}{2}\Z} \Ga$$
defined as
$$j(f)(n) = f(n)_0 \text{ and } j(f)(n+1/2) = f(n)_1 \text{ for all } n\in \Z, f\in\oplus_{n\in\Z}\Ga^2.$$
Consider the automorphism $\kappa\in \Aut(\oplus_{n\in \frac{1}{2}\Z} \Ga)$ which consists in shifting indices by $1/2.$
Observe that the conjugated automorphism $\kappa':= j^{-1}\circ \kappa\circ j$ of $\oplus_\Z \Ga^2$ is $\Z$-equivariant and thus extends to an automorphism $\theta$ of $\Ga^2\wr\Z$. 
However, it is not spatial. 
Indeed, fix $g\in\Ga$ nontrivial, $n\in\Z$ and define $f\in\oplus_{n\in\Z}\Ga$ supported at $n\in\Z$ so that $f(n)=(g,g).$
We have that $\kappa'(f)$ has support $\{n,n+1\}$ since $f(n)=(e_\Ga,g)$ and $f(n+1)=(g,e_\Ga).$
If $\theta$ was spatial, then the support of $f$ and $\kappa'(f)$ should have the same cardinal which is not the case here.

We now present isomorphic twisted permutational restricted wreath products that are constructed from non-isomorphic groups $\Ga$ and $\ti\Ga$.
Let $\Ga$ be any nontrivial group and consider $B:=B_0\oplus B_1$ the direct sum of two nontrivial groups.
Consider the standard restricted wreath product $G:=\Ga\wr B=K\rtimes B$.
Define now $\ti\Ga:= \Ga^{B_1}$, $\ti K:= \oplus_{B_0} (\Ga^{B_1})=\oplus_{B_0}\ti\Ga$ and the following action of $B$:
$$B\act \ti K ,\ [(b_0,b_1)\cdot f](a_0):= \beta_{b_1}(f(b_0^{-1}a_0)), (b_0,b_1)\in B, a_0\in B_0, f\in\ti K$$
such that 
$$\beta:B_1\act \Ga^{B_1}$$
is the left-shift action.
The semidirect product $\ti G:=\ti K\rtimes B$ is a twisted permutational restricted wreath product.
Note that the map:
$$\kappa:K\to \ti K, \kappa(f)(b_0,b_1) = f(b_0)(b_1),\ f\in K, b_0\in B_0, b_1\in B_1.$$
is a $B$-equivariant isomorphism inducing an isomorphism from $G$ to $\ti G$.
Although, $\Ga$ is not isomorphic to $\ti\Ga:=\Ga^{B_1}$ in general and the two actions $B\act B_0$ and $B\act B$ are transitive
\end{remark}

Considering endomorphisms rather than automorphisms we obtain the following classification result.

\begin{corollary}
Consider some groups $\Ga,\ti\Ga$ and endomorphisms $\al\in\End(\Ga),\ti\al\in\End(\ti\Ga).$
Let $G=K\rtimes V$ and $\ti G = \ti K\rtimes V$ be the associated fraction groups.
Let $\lim\Ga$ be the directed limit of groups constructed in Section \ref{sec:End} with automorphisms $\lim\al$ and similarly consider $(\lim{\ti \Ga},\lim{\ti \al}).$

The groups $G$ and $\ti G$ are isomorphic if and only if there exists an isomorphism $\beta\in\Isom(\lim\Ga,\lim{\ti\Ga})$ and $h\in\lim{\ti\Ga}$ such that 
$$\lim\al = \ad(h)\circ \beta\circ \lim{\ti\al}\circ \beta^{-1}.$$
\end{corollary}

Note that the isomorphism $\beta$ of above can be arbitrary. 
This suggests that there are no simple ways to express the connections between $(\Ga,\al)$ and $(\ti\Ga,\ti\al)$ without passing through the limits $(\lim\Ga,\lim\al)$ and $(\lim\ti\Ga,\lim\ti\al).$

\section{Description of the automorphism group of a fraction group}\label{sec:AutG}

In all this section we consider fraction groups that are isomorphic to \textit{untwisted} restricted permutational wreath products. 
These are the fraction groups obtained from the choice of a group $\Ga$ and the morphism $\Ga\to\Ga\oplus\Ga, g\mapsto (g,e_\Ga)$, that is when the endomorphism $\al$ is the identity automorphism $\id_\Ga$.
We consider a fixed group $\Ga$ that we assume nontrivial, the trivial case being not interesting for our study.
Put $K:=\oplus_{\Q_2}\Ga$ the group of finitely supported maps from $\Q_2$ to $\Ga$ and write $G:=K\rtimes V$ for the restricted permutational wreath product associated to the action $V\act \Q_2$ which is the fraction group we want to study.
The aim of this section is to provide a clear description of the automorphism group $\Aut(G)$.

\subsection{The four groups acting on $G$}
We start by fixing some notation and defining certain groups.
Recall that $\Homeo(\fC)$ is the group of homeomorphisms of Cantor space $\fC$ and that Thompson's group $V$ is identified with a subgroup of it.
Let 
$$\NCV:=\{\varphi\in\Homeo(\fC):\ \varphi V \varphi^{-1}=V\}$$ be the normaliser subgroup of $V$ inside $\Homeo(\fC)$ and put $\SNQ$ the stabiliser subgroup of $\NCV$ for the subset $\Q_2\subset \fC.$
Therefore, 
$$\SNQ:=\{ \varphi\in \Homeo(\fC):\ \varphi V\varphi^{-1} =V \text{ and } \varphi(\Q_2)=\Q_2\}.$$
As explained earlier $\NCV\to \Aut(V), \varphi\mapsto \ad_\varphi$ realises an isomorphism. 
The fact that this morphism is surjective is a consequence of Rubin's theorem and the injectivity follows from an easy argument, see \cite{Rubin96} and \cite[Section 3]{BCMNO19} for details.
Moreover, note that $\SNQ$ is a proper subgroup of $\Aut(V)$. 
Indeed, if $\sigma$ is the nontrivial permutation of $\{0,1\}$ we can induce an homeomorphism of $\fC$. This element is in $\NCV$ but does not stabilise $\Q_2$.
We refer the reader to \cite{BCMNO19} for a deep description of $\Aut(V)$.

Consider the group 
$$\ov K:=\prod_{\Q_2}\Ga$$ of all maps from $\Q_2$ to $\Ga$ and the unrestricted permutational wreath product $$\ov G:=\prod_{\Q_2}\Ga\rtimes V.$$
Identify $G:=\oplus_{\Q_2}\Ga\rtimes V$ with the corresponding subgroup of $\ov G$ and consider the elements of $\ov K$ inside $\ov G$ which normalise $G$.
We write
$$N_{\ov K}(G):=\{f\in \prod_{\Q_2}\Ga:\ f G f^{-1} =G\}$$ 
for the normaliser subgroup.
As usual $\Aut(\Ga)$ is the automorphism group of $\Ga$ and $Z\Ga$ the centre of $\Ga.$

We are going to show that $\Aut(G)$ is generated by some copy of the following groups:
$$\SNQ, \Aut(\Ga), N_{\ov K}(G) \text{ and } Z\Ga.$$
They will all act faithfully on $G$ except $N_{\ov K}(G)$ that we will mod out by the normal subgroup of constant maps from $\Q_2$ to $Z\Ga$ that we simply denote by $Z\Ga$.

\subsubsection{Action of the automorphism group of the input group}

The whole automorphism group $\Aut(\Ga)$ acts on $G$ in the expected diagonal way:
$$\beta\cdot av := \ov \beta(a) v, \ \beta\in \Aut(\Ga), a\in K, v\in V$$
where $$\ov\beta(a):\Q_2\to \Ga, x\mapsto \beta(a(x)).$$

It is easy to check that this formula defines a faithful action by automorphism $\ov\beta:\Ga\act G$.

\subsubsection{Action of certain automorphisms of Thompson's group $V$} 

The stabiliser $\SNQ$ acts spatially on $G$ as follows:
$$ \varphi \cdot (av) := a^\varphi \cdot \ad_\varphi(v), \ \varphi \in \SNQ, a\in \oplus_{\Q_2}\Ga, v\in V$$
where $$a^\varphi:\Q_2\to \Ga, x\mapsto a(\varphi^{-1} x) \text{ and } \ad_\varphi(v)= \varphi v \varphi^{-1}.$$
This is well-defined since $\supp(a^\varphi)=\varphi(\supp(a))$ and thus if $a$ is finitely supported so is $a^\varphi.$
This action is clearly faithful since the action of $\NCV$ on $V$ is known to be faithful, see \cite[Section 3]{BCMNO19}.

\subsubsection{Adjoint action of $\Ga$-valued functions}
Given $f\in \ov K$ we can act on $\ov G$ by conjugation:
$$f\cdot (av) = \ad(f)(av) = favf^{-1} = (fa(f^v)^{-1})\cdot v, \ f\in \ov K, a\in \ov K, v\in V.$$ 
This restricts to an action by automorphisms:
$$\ad: N_{\ov K}(G)\to \Aut(G).$$

Let us compute the kernel of $\ad$.
Assume that $\ad(f)=\id_G$ for a certain $f\in N_{\ov K}(G).$
We then obtain that 
$$v = \ad(f)(v) = f(f^v)^{-1} v$$ implying that $f=f^v$ for all $v\in V$.
Since $V\act \Q_2$ is transitive we obtain that $f$ is constant.
Consider $a\in K, x\in \Q_2$ and observe that 
$$a(x)=\ad(f)(a)(x) = f(x)a(x)f^{-1}(x)$$ 
implying that $f(x)$ is central in $\Ga$ and thus $f$ is constant and valued in $Z\Ga.$
Conversely, if $\zeta\in Z\Ga$ and $f(x)=\zeta$ for all $x\in \Q_2,$ we have that 
$$\ad(f)(av) = fa(f^v)^{-1} v = \zeta \zeta^{-1} a v = av$$ for all $a\in K, v\in V$.
We have proven that the kernel of the action $\ad:N_{\ov K}(G)\act G$ is equal to the group of all constant maps from $\Q_2$ to $Z\Ga$ that we simply denote by $Z\Ga$.

Note that the group $N_{\ov K}(G)$ is in general strictly larger than $K$ as shown in the following remark. It is clear that constant maps are in the normalisers but there are also some less trivial ones.

\begin{remark}
We provide an example of an element of the normaliser subgroup that is not in $K$ nor constant. 
Assume $\Ga=\Z_2.$
Consider the following set: $$X:=\{ \dfrac{4k+1}{2^n} :\ n\in \Z, k\in\Z \}$$ and write $Y:=X\cap [0,1)$.
We put $f:=\chi_Y$ the characteristic function of $Y$ interpreted as an element of $\ov K=\prod_{\Q_2} \Z_2.$
Observe that $X=2X$ and that the symmetric difference $(X+\dfrac{1}{2})\Delta X$ is locally finite in the sense that its intersection with any interval $(x,y), x,y\in \R$ is finite.
This implies that for any $v\in V$ we have that $v\cdot \chi_Y = \chi_Y \mod \oplus_{\Q_2}\Z_2.$
Hence, $$\ad(f)(av) = fa(f^v)^{-1} v = \chi_{Y\Delta v Y} \cdot a v, \ a \in K, v\in V,$$
which is an element of $K\rtimes V$ since $\chi_{Y\Delta v Y}$ is finitely supported. However, $f=\chi_Y$ is not in $K$ since $Y$ is an infinite set nor is constant.
\end{remark}

\subsubsection{Exotic automorphisms: actions of the centre of the input group}
There is an unexpected action of $Z\Ga$ on $G$ that we now describe.
We construct those exotic automorphisms in a similar way than in Lemma \ref{lem:isomtwo} by using the dyadic valuation and the slopes of elements of $V$.

\begin{proposition}\label{prop:zetacocycle}
For any $v\in V$ we define the map $p_v:=\log_2(v')^v - \nu +\nu^v$ that is
$$p_v(x) = \log_2(v'(v^{-1}x)) - \nu(x) + \nu(v^{-1}x) ,\ x\in\Q_2.$$
For any $\zeta\in Z\Ga$  we put 
$$c(\zeta): V\to \prod_{\Q_2}Z\Ga, \ c(\zeta)_v= \zeta^{p_v}, \ c(\zeta)_v(x)= \zeta^{p_v(x)}, v\in V, x\in\Q_2.$$
The map $c(\zeta)$ is valued in $\oplus_{\Q_2}Z\Ga$ and satisfies the following cocycle identity:
$$c(\zeta)_{vw} = c(\zeta)_v \cdot c(\zeta)_w^v \text{ for all } v,w\in V.$$

This defines an injective group morphism:
$$E:Z\Ga\to \Aut(G),\ E_\zeta(av) := a \cdot c(\zeta)_v \cdot v \text{ for all } a \in K, v\in V, \zeta\in Z\Ga.$$
\end{proposition}

\begin{proof}
The first claim of Lemma \ref{lem:isomtwo} implies that $p_v:\Q_2\to\Z$ is finitely supported and so  $c(\zeta)_v$ is for all $v\in V,\zeta\in Z\Ga.$

Let us show that $E$ defines a group morphism.
Consider $a,b\in K, v,w\in V, \zeta\in Z\Ga$.
We have that 
\begin{align*}
E_\zeta(av) \cdot E_\zeta(bw) & = a \cdot \zeta^{p_v} \cdot v \cdot b \cdot \zeta^{p_w} \cdot w = a\cdot \zeta^{p_v}\cdot  b^v \cdot \zeta^{p_w^v} \cdot vw = \zeta^{p_v + p_w^v} \cdot ab^v \cdot vw \\
E_\zeta(av\cdot bw) & = E_\zeta( a b^v \cdot vw) = ab^v\cdot  \zeta^{p_{vw}}\cdot  vw. 
\end{align*}
Observe now that
\begin{align*}
[p_v + p_w^v](x) & = \log_2(v'(v^{-1}x)) - \nu(x) + \nu(v^{-1}x) + \log_2(w'(w^{-1}v^{-1}x)) - \nu(v^{-1}x) + \nu(w^{-1}v^{-1}x)\\
& = \log_2(v'(v^{-1}x)) +  \log_2(w'(w^{-1}v^{-1}x)) - \nu(x) + \nu(w^{-1}v^{-1}x)\\
p_{vw}(x) & = \log_2((vw)'((vw)^{-1}x) - \nu(x) + \nu((vw)^{-1}x) \\
& = \log_2(v'(v^{-1}x)) + \log_2(w'(w^{-1} v^{-1}x) - \nu(x) + \nu(w^{-1} v^{-1} x),
\end{align*}
for $x\in \Q_2.$
We proved that $$p_v+p_w^v = p_{vw}.$$
All together this implies that for any $\zeta\in Z\Ga$ the map $E_\zeta$ is an endomorphism of the group $G$.
It is easy to see that $E_\zeta$ is bijective with inverse map $E_{\zeta^{-1}}$ implying that $E_\zeta$ is an automorphism of $G$ for $\zeta\in Z\Ga.$
It is rather obvious that $E_\zeta\circ E_\eta = E_{\zeta\eta}$ for $\zeta,\eta\in Z\Ga$ implying that $E$ defines a group morphism from $Z\Ga$ to $\Aut(G).$

To finish the proof it is sufficient to prove that $E$ is faithful.
Assume that $E_\zeta=\id_G$ for a certain $\zeta\in Z\Ga.$
This is equivalent to having
$$\zeta^{p_v(x)}=e \text{ for all } v\in V, x\in\Q_2.$$
Consider $v\in V$ such that $v0=0$ and $v'(0)=2.$
We obtain that $\zeta^{p_v(0)}= \zeta$ and thus $\zeta=e$. 
Therefore, the kernel of $E$ is trivial.
\end{proof}

We now fix some notations concerning the actions of those four groups on $G$ and summarise the observation of above in the following proposition.
\begin{proposition}
Consider the direct product $\SNQ\times \Aut(\Ga)$ and define the map 
$$A: \SNQ\times \Aut(\Ga)\to \Aut(G), \ A_{\varphi,\beta}(av):= \ov \beta(a)^\varphi \ad_\varphi(v)$$
for $\varphi\in  \SNQ, \beta\in\Aut(\Ga), a\in K, v\in V$ such that 
$$\ov\beta(a)(x):=\beta(a(x)) \text{ and } a^\varphi(x):=a(\varphi^{-1}(x)), x\in\Q_2.$$
The map $A$ is an injective group morphism.
Define the map $$\ad:N_{\ov K}(G)\to \Aut(G), \ad(f)(av):= favf^{-1} = (faf^{-1}) f(f^v)^{-1} v$$
for $f\in N_{\ov K}(G), a\in K, v\in V.$
The map $\ad$ is a group morphism and its kernel is $Z\Ga\lhd N_{\ov K}(G).$
We continue to write $$\ad:N_{\ov K}(G)/Z\Ga\to \Aut(G)$$ for
the factorised injective group morphism.

For any $v\in V$ we put $$p_v:\Q_2\to \Z, \ p_v(x):= \log_2(v'(v^{-1}x)) -\nu(x) + \nu(v^{-1}x), \ x\in\Q_2$$
where $\nu$ is the dyadic valuation and define the map
$$E:Z\Ga\to \Aut(G),\ E_\zeta(av) = a \cdot \zeta^{p_v} \cdot v, \ a\in K , v\in V, \zeta\in Z\Ga$$
which is an injective group morphism.
\end{proposition}
We have nothing to prove except that the actions of $\SNQ$ and $\Aut(\Ga)$ on $G$ mutually commute which is an easy computation.

\subsection{Semidirect product}\label{sec:semidirect}

We will later prove that any automorphism of $G$ can be decomposed uniquely as a product of four elements of the groups 
$$\SNQ, \Aut(\Ga),N_{\ov K}(G)/Z\Ga \text{ and } Z\Ga.$$
In this section we describe how these four groups sit together inside $\Aut(G)$.
They come in two direct products: $\SNQ\times \Aut(\Ga)$ and $Z\Ga\times N_{\ov K}(G)/Z\Ga$ with the first direct product acting on the second in a semidirect product fashion.
Before defining the action we will need a technical lemma.

\begin{lemma}\label{lem:formula}
Consider $x\in \Q_2, \phi,\varphi\in\SNQ$ and $v,w\in V$.
We have the following formula:
\begin{enumerate}
\item If $vx=wx$, then $$\log_2((\varphi^{-1} v \varphi)'(\varphi^{-1}x)) - \log_2(v'(x)) = \log_2((\varphi^{-1} w \varphi)'(\varphi^{-1}x)) - \log_2(w'(x));$$
\item Given $x\in\Q_2$ and $v\in V$ satisfying $v0=x$ we put $$\ga_\varphi(x):=  \log_2((\varphi^{-1} v\varphi)'(\varphi^{-1}0)) - \log_2(v'(0)).$$
This formula does not depend on the choice of $v$ and defines a map $\ga_\varphi:\Q_2\to \Z$;\\
\item For any $x\in \Q_2$ and $v\in V$ we have
$$\ga_\varphi(vx)-\ga_\varphi(x) = \log_2((\varphi^{-1} v \varphi)'(\varphi^{-1}x)) - \log_2(v'(x));$$\\
\item We have the equality $\ga_{\phi\varphi} = \ga_\phi + \ga_\varphi^\phi - \ga_\varphi(\phi^{-1}(0))$.
\end{enumerate}
\end{lemma}

Note that here the symbol $0$ denotes the usual zero of the dyadic rationals $\Q_2.$

\begin{proof}
Consider $x\in \Q_2, \phi,\varphi\in\SNQ$ and $v,w\in V$.

Proof of (1).

Assume that $vx=wx$ and put $u:=v^{-1}w$.
Since $ux=x$ we can apply Proposition \ref{prop:slope} obtaining that 
$$(\varphi^{-1} u \varphi)'(\varphi^{-1}x) = u'(x).$$
Observe that 
\begin{align*}
\dfrac{(\varphi^{-1} vu \varphi)'(\varphi^{-1}x)}{(vu)'(x)} & = \dfrac{([\varphi^{-1} v \varphi] \circ [\varphi^{-1} u \varphi])'(\varphi^{-1}x)}{v'(ux)\cdot u'(x)} \\
& = \dfrac{[\varphi^{-1} v \varphi]'(\varphi^{-1}x) \cdot [\varphi^{-1} u \varphi]'(\varphi^{-1}x)}{v'(x)\cdot u'(x)} \\
& = \dfrac{[\varphi^{-1} v \varphi]'(\varphi^{-1}x)}{v'(x)}. \\
\end{align*}
Taking the logarithm we obtain (1).

Proof of (2).

We need to show that if $v0=w0$, then 
$$\log_2((\varphi^{-1} v \varphi)'(\varphi^{-1}0)) - \log_2(v'(0)) = \log_2((\varphi^{-1} w \varphi)'(\varphi^{-1}0)) - \log_2(w'(0)).$$
This is a direct consequence of (1) applied to $x=0.$

Proof of (3).

Consider $x\in \Q_2$ such that $w0=x$ and let $v$ be in $V$.
Observe that 
\begin{align*}
\ga_\varphi(vx)-\ga_\varphi(x) = & \ga_\varphi(vw0) - \ga_\varphi(w0)\\
= & \log_2((\varphi^{-1} vw\varphi)'(\varphi^{-1}0)) - \log_2((vw)'(0))\\
& -  \log_2((\varphi^{-1} w\varphi)'(\varphi^{-1}0)) + \log_2(w'(0))\\
= &  \log_2((\varphi^{-1} v\varphi)'(\varphi^{-1}w0)) + \log_2((\varphi^{-1} w\varphi)'(\varphi^{-1}0))- \log_2((vw)'(0))\\
& -  \log_2((\varphi^{-1} w\varphi)'(\varphi^{-1}0)) + \log_2(w'(0))\\
= &  \log_2((\varphi^{-1} v\varphi)'(\varphi^{-1}w0)) - \log_2(v'(w0)) - \log_2(w'(0)) + \log_2(w'(0))\\
= & \log_2((\varphi^{-1} v\varphi)'(\varphi^{-1}w0)) - \log_2(v'(w0))\\
= & \log_2((\varphi^{-1} v\varphi)'(\varphi^{-1}x)) - \log_2(v'(x)).\\
\end{align*}
Proof of (4).

Consider $y\in \Q_2$ and $v\in V$ such that $v0=y.$

Observe that
\begin{align*}
\ga_{\phi\varphi}(y)  = & \ga_{\phi\varphi}(v0) = \log_2((\phi\varphi)^{-1} v (\phi\varphi))'(\varphi^{-1}\phi^{-1}0) - \log_2(v'(0))\\
= & \log_2( [ \varphi^{-1} ( \phi^{-1} v \phi ) \varphi ]'( \varphi^{-1} [ \phi^{-1}  0] ) - \log_2( ( \phi^{-1} v \phi)'(\phi^{-1} 0) ) \\
& +  \log_2( ( \phi^{-1} v \phi)'(\phi^{-1} 0) )  - \log_2(v'(0))\\
= & \log_2( [ \varphi^{-1} ( \phi^{-1} v \phi ) \varphi ]'( \varphi^{-1} [ \phi^{-1}  0 ] ) - \log_2( ( \phi^{-1} v \phi)'(\phi^{-1} 0) ) + \ga_\phi(v0)\\
= & \ga_\varphi( (\phi^{-1} v \phi) (\phi^{-1} 0) ) - \ga_\varphi(\phi^{-1} 0) + \ga_\phi(v0) \text{ using (3) }\\
= & \ga_\varphi( (\phi^{-1} v 0) ) - \ga_\varphi(\phi^{-1} 0) + \ga_\phi(v0)\\
= & [\ga_\varphi^\phi + \ga_\phi](y) - \ga_\varphi(\phi^{-1}0).
\end{align*}

\end{proof}

Let $\nu:\Q\to \Z$ be the dyadic valuation, see Notation \ref{not:valuation}.
For any $\varphi\in\Stab_N(\Q_2)$ define the map $\mu_\varphi:\Q_2\to \Z$ as follows:
\begin{equation}\label{def:muvarphi}\mu_\varphi = \ga_\varphi -\nu^\varphi + \nu.\end{equation}
Note that if $x\in \Q_2$ and $v\in V$ satisfies $v0=x$, then
$$\mu_\varphi(x) = \log_2((\varphi^{-1} v\varphi)'(\varphi^{-1}0)) - \log_2(v'(0)) -\nu(\varphi^{-1} v0) + \nu(v0)$$
and this formula is independent of the choice of $v$ by the previous lemma.
We can now define an action of $\SNQ\times \Aut(\Ga)$ on $Z\Ga\times N_{\ov K}(G)/Z\Ga$ which we will prove to be the action describing the semidirect product decomposition of $\Aut(G).$

\begin{proposition}\label{prop:semidirect}
We have an action by automorphism:
$$\sigma_0: \SNQ\times \Aut(\Ga) \to \Aut(Z\Ga\times \ov K/Z\Ga)$$
described by the formula:
$$\sigma_0(\varphi,\beta)(\zeta, f):= (\beta(\zeta),\ov\beta(f)^\varphi \cdot \beta(\zeta)^{\mu_\varphi})$$
for all 
$$(\zeta,f)\in  Z\Ga\times \ov K / Z\Ga \text{ and } (\varphi,\beta)\in \SNQ\times \Aut(\Ga).$$
This map $\sigma$ induces an injective group morphism 
$$\sigma: \SNQ\times \Aut(\Ga) \to \Aut(Z\Ga\times N_{\ov K }(G)/Z\Ga).$$
\end{proposition}
\begin{remark}
The notation is slightly misleading. 
Given $\beta\in\Aut(\Ga), \varphi\in\SNQ, f\in N_{\ov K}(G), \zeta\in Z\Ga, x\in \Q_2$ we have that $$\ov\beta(f)^\varphi(x):= \beta( f(\varphi^{-1}x))$$ where the superscript $\varphi$ means that we precompose $f$ with the function $\varphi^{-1}$ while
$$\beta(\zeta)^{\mu_\varphi}(x):= \beta( \zeta)^{\mu_\varphi(x)}$$
where the superscript $\mu_\varphi$ stands for $\zeta$ elevated to the power $\mu_\varphi.$
\end{remark}

\begin{proof}
Consider $(\varphi,\beta)\in \SNQ\times \Aut(\Ga)$.
It is clear that the formula
$$\sigma_0(\varphi,\beta):(\zeta,f)\mapsto (\beta(\zeta), \ov\beta(f)^\varphi\cdot \beta(\zeta)^{\mu_\varphi})$$ 
defines a map from $Z\Ga\times \ov K$ to itself since any automorphism of $\Ga$ maps its centre to itself. 
Moreover, this is clearly a group endomorphism of $Z\Ga\times \ov K.$

Consider the quotient map $q:Z\Ga\times \ov K\to Z\Ga\times \ov K/Z\Ga$.
If $(e,f)\in \ker(q)$ (that is $f\in Z\Ga$), then $\sigma_0(e,f)=(e, \ov\beta(f)^\varphi)=(e,\ov\beta(f))\in \ker(q)$ and thus $\sigma_0(\varphi,\beta)$ factorises into an endomorphism written $\sigma(\varphi,\beta)$ of $Z\Ga\times \ov K/Z\Ga.$

Let us check that $\sigma$ is multiplicative.
Observe that if $\varphi,\varphi_0\in \SNQ, \beta,\beta_0\in\Aut(\Ga), \zeta\in Z\Ga, f\in \ov K$, then:
\begin{align*}
\sigma(\varphi,\beta)\circ \sigma(\varphi_0,\beta_0) (\zeta,f) & = \sigma(\varphi,\beta) ( \beta_0(\zeta) , {\ov\beta_0(f)}^{\varphi_0} \cdot \beta_0(\zeta)^{\mu_{\varphi_0}})\\
& = (\beta\beta_0(\zeta) , \ov \beta( {\ov\beta_0(f)}^{\varphi_0} \beta_0(\zeta)^{\mu_{\varphi_0}})^\varphi \cdot \beta( \beta_0(\zeta) )^{\mu_\varphi}) \\
& =  (\beta\beta_0(\zeta) , \ov{(\beta\beta_0)}(f)^{\varphi\varphi_0} \cdot (\beta\beta_0)(\zeta)^{\mu_{\varphi_0}^\varphi} \cdot (\beta\beta_0)(\zeta)^{\mu_\varphi})\\
& =  (\beta\beta_0(\zeta) , \ov{(\beta\beta_0)}(f)^{\varphi\varphi_0} \cdot (\beta\beta_0)(\zeta)^{\mu_{\varphi_0}^\varphi+\mu_\varphi}).
\end{align*}
Now it is sufficient to check:
$$\zeta^{\mu_{\varphi_0}^\varphi+\mu_\varphi} = \zeta^{\mu_{\varphi\varphi_0}} \mod Z\Ga \text{ for all } \zeta\in Z\Ga.$$
Finally, one checks that
$$\mu_{\varphi\varphi_0} = \mu_{\varphi_0}^\varphi + \mu_\varphi \mod \Z.$$
Lemma \ref{lem:formula} Formula (3) implies that $$\mu_{\varphi\varphi_0} =  \mu_{\varphi_0}^\varphi + \mu_\varphi - \mu_{\varphi_0}(\varphi^{-1}(0))$$ giving us the desirable equality modulo the constant maps.

Observe now that $\sigma(e,e)$ is the identity implying that $\sigma$ is a group morphism from $\SNQ\times \Aut(\Ga)$ to $\Aut(Z\Ga\times \ov K/Z\Ga).$

We now show that $\sigma$ acts on the subgroup $Z\Ga\times N_{\ov K}(G)/Z\Ga$.
Consider $\zeta\in Z\Ga$ and $\varphi\in \SNQ.$
We want to show that $\zeta^{\mu_\varphi}$ is normalising $G$.
Consider $av\in G$ and observe that 
$$\ad(\zeta^{\mu_\varphi})(av) = \zeta^{\mu_\varphi - \mu_\varphi^v} av.$$
To conclude it is then sufficient to show that the support of $\mu_\varphi-\mu_\varphi^v$ is finite for all $v\in V.$
Observe that for $x\in \Q_2$ we have:
\begin{align*}
[\mu_\varphi-\mu_\varphi^v](vx)  = & \log_2((\varphi^{-1} v \varphi)'(\varphi^{-1}x)) - \log_2(v'(x))\\
& -\nu(\varphi^{-1} vx) + \nu(vx) + \nu(\varphi^{-1}x) - \nu(x) \text{ by Lemma \ref{lem:formula} } \\
= & [ \log_2((\varphi^{-1} v \varphi)'(\varphi^{-1}x)) - \nu(\varphi^{-1}vx) + \nu(\varphi^{-1} x) ]\\
& - [ \log_2(v'(x)) - \nu(vx) + \nu(x)]\\
= & p_{\varphi^{-1}v\varphi}(\varphi^{-1}x) - p_v(vx).
\end{align*}
Since $p_{\varphi^{-1}v\varphi}$ and $p_v$ are finitely supported so is the map $\mu_\varphi-\mu_\varphi^v$ implying that $\zeta^{\mu_\varphi}\in N_{\ov K}(G).$
It is now easy to deduce that 
$$\sigma_0(\varphi,\beta)( Z\Ga\times N_{\ov K}(G)) = Z\Ga\times N_{\ov K}(G) \text{ for all } (\varphi,\beta)\in\SNQ\times \Aut(\Ga).$$
This implies that $\sigma$ provides an action by automorphisms on $Z\Ga\times N_{\ov K}(G)/Z\Ga.$

It remains to prove that this latter action is faithful.
Assume that $\sigma(\varphi,\beta)$ is the identity automorphism.
Since $\Ga$ is nontrivial by assumption there exists $g\in \Ga, g\neq e$.
Write $g_x\in K$ for the map with support $\{x\}$ taking the value $g$.
Note that $\sigma(\varphi,\beta)(e,g_x) = (e, \beta(g)_{\varphi(x)})$ implying that 
$$g_x = \beta(g)_{\varphi(x)} \mod Z\Ga \text{ for all } x\in\Q_2, g\in \Ga.$$
This implies that $\varphi(x)=x$ for all $x\in \Q_2$ and $\beta(g)=g$ for all $g\in \Ga$ concluding the proof.
\end{proof}

\subsection{Decomposition of automorphisms of $G$}
We start by classifying all automorphisms of $G$ of the following form:
$$av\mapsto a \cdot c_v \cdot v , \ a\in K, v\in V$$
and where $c:V\to K, v\mapsto c_v.$
We call such $c$ cocycles for the following formula that they must satisfy for defining an action:
$$c_{vw} = c_v \cdot c_w^v ,\ v,w\in V.$$

Write $$\Coc(V\act \prod_{\Q_2}\Ga):=\{ c:V\to \prod_{\Q_2}\Ga:\ c_{vw}=c_v \cdot c_w^v \text{ for all } v,w\in V\}.$$
We will be particularly interested by the subsets $\Coc(V\act \prod_{\Q_2}Z\Ga)$ and $\Coc(V\act K)$.

\begin{proposition}\label{prop:cocycle Abelian}
Let $\La$ be an abelian group and consider the set of cocycles
$$\Coc=\Coc(V\act \prod_{\Q_2}\La):=\{ c:V\to \prod_{\Q_2}\La:\ c_{vw}=c_v \cdot c_w^v,\ \forall v,w\in V\}.$$
\begin{enumerate}
\item Equipped with the product 
$$(c\cdot d)_v(x):=c_v(x) d_v(x) , \ c,d\in \Coc, v\in V, x\in\Q_2$$
the set $\Coc$ is an abelian group.
\item Given any $\zeta\in \La$ the formula
$$s(\zeta)_v(x):= \zeta^{\log_2(v'(v^{-1}x))} , \ v\in V, x\in \Q_2$$
defines a cocycle that we call the slope cocycle associated to $\zeta$;
\item
For any $c\in \Coc$ there exists $\zeta\in\La$ and $f\in \prod_{\Q_2}\La$ satisfying
$$c_v= s(\zeta)_v \cdot f (f^v)^{-1} , \ v\in V.$$
The pair $(\zeta,f)$ is unique up to multiplying $f$ by a constant map.
\item The assignment $c\mapsto (\zeta,f)$ realises a group isomorphism from $\Coc$ onto $\La\times \left(\prod_{\Q_2}\La\right)/\La.$
\end{enumerate}
\end{proposition}

\begin{proof}
Consider $c,d\in \Coc$ and $v,w\in V.$
Then
$$(cd)_{vw} = c_{vw}\cdot d_{vw} = c_v\cdot c_w^v \cdot d_v \cdot d_w^v = (c_v d_v) (c_w^v d_w^v) = (cd)_v\cdot (cd)_w^v.$$
Therefore, $cd$ is a cocycle.
The cocycle $c$ satisfying $c_v(x)=e$ for all $v\in V, x\in\Q_2$ is neutral for the multiplication. 
Fix $d\in \Coc$ and define $b_v(x):= d_v(x)^{-1}, v\in V, x\in\Q_2$. 
Observe that
\begin{align*}
b_{vw}(vwx) & = d_{vw}(vwx)^{-1} = \left(d_v(vwx) d_w(wx)\right)^{-1} = d_w(wx)^{-1} d_v(vwx)^{-1} \\
& = d_v(vwx)^{-1}  d_w(wx)^{-1} = b_v(vwx) b_w(wx) = (b_v\cdot b_w^v)(vwx),
\end{align*}
for all $v,w\in V, x\in \Q_2$.
Therefore, $b$ is in $\Coc$ and it is easy to see that $b$ is an inverse of $d$ inside $\Coc$. 
Therefore, $\Coc$ is a group which is clearly abelian.

The fact that $(vw)'(x) = v'(wx) \cdot w'(x)$ for $v,w\in V, x\in\Q_2$ implies that $s(\zeta)$ is a cocycle for $\zeta\in \La.$

Fix $c\in \Coc$. We are going to decompose $c$ such that 
$$c_v=s(\zeta)_v\cdot f(f^v)^{-1}, \ v\in V$$ 
for some $\zeta\in \La, f\in\prod_{\Q_2}\La.$
Fix $x\in \Q_2$ and write $V_x:=\{v\in V:\ vx=x\}$ with derived group $V_x'.$
Consider the map $$P_x: V_x\mapsto \La, \ v\mapsto c_v(x)$$
and note that $P_x$ is a group morphism valued in an abelian group and thus factorises into a group morphism 
$$\ov P_x: V_x/V_x'\to \La.$$
Since $x\in\Q_2$ we can find $v\in V_x$ such that $v'(x)=2.$ 
(This would not be the case for general $x$ in Cantor space, see \cite{Bleak-Lanoue10} or \cite{Brothier20-2} for details.) 
This implies that the map $$\ell_x:V_x\to \Z,\ v\mapsto \log_2(v'(x))$$ is surjective. 
By Lemma \ref{lem:slopeone} we have that $V_x'=\ker(\ell_x)$ and thus $\ell_x$ factorises into an isomorphism 
$$\ov \ell_x: V_x/V_x'\to \Z.$$
Write $1_\Z\in \Z$ for the positive generator of $\Z$ and consider $\zeta_x:= \ov P_x\circ \ov\ell_x^{-1}(1_\Z)$ which is in $\La$.
Observe that
$$c_v(x) = \zeta_x^{\log_2(v'(x))} \text{ for all } v\in V_x.$$

Let us show that $\zeta_x$ does not depend on $x\in\Q_2.$
Consider $x,y\in \Q_2$.
There exists $w\in V$ such that $wx=y$ since $V$ acts transitively on $\Q_2.$
The adjoint map $$\ad_w:V\to V, v\mapsto wvw^{-1}$$ restricts into an isomorphism from $V_x$ onto $V_{y}.$
Fix such a $w$ and take $v\in V_x.$
Observe that
\begin{align*}
c_{wvw^{-1}}(y) & = c_{wvw^{-1}}(wx) = c_{wv}(wx) c_{w^{-1}}(x) = c_w(wx) c_v(x) c_{w^{-1}}(x)\\
& = c_v(x) [ c_w(wx) c_{w^{-1}}(x) ] = c_v(x) [c_wc_w^w](wx)\\
& = c_v(x) c_e(wx) =c_v(x).
\end{align*}
We obtain that $\zeta_y^{\log_2( \ad_w(v)'(y))} = \zeta_x^{\log_2(v'(x))}$ for all $v\in V_x.$
Now observe that $\ad_w(v)'(y) = v'(x)$ if $y=wx, v\in V_x$ by the chain rule applied to elements of Thompson's group $V$.
Choosing $v\in V_x$ with slope 2 at $x$ we obtain that $\zeta_x=\zeta_y$.
We have proven that there exists a unique $\zeta\in \La$ such that 
$$c_v(x) = \zeta^{\log_2(v'(x))} \text{ for all } x\in\Q_2, v\in V_x.$$

Put $$s(\zeta)_v(x):= \zeta^{\log_2(v'(v^{-1}x))}, v\in V, x\in \Q_2.$$
We write $c= s(\zeta) \cdot d$ where $d:=c \cdot s(\zeta)^{-1}\in \Coc.$
We are going to show that there exists $f:\Q_2\to\La$ satisfying
$$d_v = f(f^v)^{-1} \text{ for all } v.$$ 
Consider $x\in\Q_2$ and $v\in V_x$.
We have that $c_v(x) = s(\zeta)_v(x)$ and thus $d_v(x)=e.$
Given $x,y\in \Q_2$ we put $V_{y,x}:=\{v\in V:\ vx=y\}$.
If $v,w\in V_{y,x}$, then $w^{-1}v\in V_x$ implying that $d_{w^{-1}v}(x)=e.$
We obtain that 
$$d_v(y) = d_{w\cdot w^{-1}v} (y) = d_w(y) \cdot d_{w^{-1}v}^w(y) = d_w(y) \cdot d_{w^{-1}v}(x) = d_w(y).$$
Therefore, $u\in V_{y,x}\mapsto d_u(y)\in \La$ is constant equal to a certain $g_{y,x}\in\La.$
The cocycle formula and the fact that $V\act \Q_2$ is transitive imply that $g_{z,x}=g_{z,y} \cdot g_{y,x}$ for all $x,y,z\in \Q_2$.
Since $d_v(x)=e$ for all $v\in V_x$, we obtain that $g_{x,x}=e$ and thus $g_{y,x}^{-1} = g_{x,y}$ for all $x,y\in\Q_2.$
Fix a point of $\Q_2$ say $0$ and consider the map $f:\Q_2\to \La, f(x):=g_{x,0}.$
Observe that $f(y)f(x)^{-1} = g_{y,0} g_{x,0}^{-1} = g_{y,0} g_{0,x} = g_{y,x}$ for all $x,y\in \Q_2$ implying that
$$[f(f^v)^{-1}](x) = f(x) f(v^{-1}x)^{-1} = g_{x,v^{-1}x} = d_v(x) \text{ for all } v\in V, x\in \Q_2$$
since $v\in V_{x,v^{-1}x}.$
We have proven that $c_v = s(\zeta)_v\cdot f(f^v)^{-1}$ for all $v\in V.$

Assume that $c_v = s(\xi)_v\cdot h(h^v)^{-1}$ for some $\xi\in\La$ and $h:\Q_2\to\La.$
For any $v\in V,x\in\Q_2$ we have 
$$\zeta^{\log_2(v'(x))} f(vx)f(x)^{-1} = \xi^{\log_2(v'(x))} h(vx)h(x)^{-1}.$$ 
If we consider $x=0$ and $v\in V$ that dilates $[0,1/4]$ into $[0,1/2]$ we obtain that $\zeta=\xi.$
Given $x=0$ and $\tau$ the translation $z\mapsto z+y$ with $y\in\Q_2$ we obtain that $f(y)=h(y)h(0)^{-1}.$
Therefore, $f = h \cdot a$ where $a:\Q_2\to \La$ is the constant map with value $h(0)^{-1}.$

Let us show that the assignment 
$c\mapsto (\zeta, f)$
is an isomorphism from $\Coc$ onto $\La\times \left(\prod_{\Q_2}\La\right)/\La.$
Note that $\zeta\in \La\mapsto s(\zeta)$ and $f\in\prod_{\Q_2}\La\mapsto (v\mapsto f(f^v)^{-1})$ are group morphisms since $\La$ is abelian.
The first one is injective and the second has for kernel the constant functions since $V\act\Q_2$ is transitive.
Therefore, 
$$\La\times \left(\prod_{\Q_2}\La\right)/\La\to \Coc, \ (\zeta,f)\mapsto \left(v\mapsto s(\zeta)_v\cdot f(f^v)^{-1}\right)$$
is an injective group morphism and we have proven that it is surjective.
This finishes the proof.
\end{proof}

Define the semidirect products induced by the action $\sigma$ of the previous section that is:
$$\sigma(\varphi,\beta) (\zeta,f) := ( \beta(\zeta) , (\beta(\zeta))^{\mu_\varphi} \cdot (\ov\beta(f))^\varphi),$$ 
where $\varphi\in\SNQ, \beta\in\Aut(\Ga), \zeta\in Z\Ga, f\in N_{\ov K}(G)/Z\Ga$ and such that $$\mu_\varphi(v0) = \log_2((\varphi^{-1} v \varphi)'(\varphi^{-1}0) - \log_2(v'(0)) -\nu(\varphi^{-1}v0) + \nu(v0) ,v\in V.$$ 
Recall that the formula of $\mu_\varphi(v0)$ only depends on $\varphi$ and $v0$ (but not on $v$).
We put $$Q:=\left( Z\Ga\times N_{\ov K}(G) /Z\Ga \right) \rtimes \left(\SNQ\times \Aut(\Ga) \right).$$

\begin{theorem}\label{theo:isomB}
The map $$\Xi: Q\to \Aut(G),\ (\zeta,f,\varphi,\beta)\mapsto E_\zeta \ad(f) A_{\varphi,\beta}$$
is a surjective group morphism with kernel the normal subgroup
$$M:= \{ (e, \ov g, e, \ad(g^{-1})): \ g\in \Ga\}$$
where $\ov g:\Q_2\to \Ga$ is the constant map equal to $g$ everywhere.
\end{theorem}
\begin{proof}
Here is the plan of the proof.
First, we show that $\Xi$ is a morphism. The delicate points resides in checking that the commutation relations between $Z\Ga$ and $\SNQ$ are preserved by $\Xi$.
Second, we compute the kernel of $\Xi$. This is done by testing certain elements of $G$ and by using the transitivity of $V\act \Q_2$.
Third, we prove that $\Xi$ is surjective. This is the most difficult part of the proof. We start by considering an automorphism $\theta$ that we decompose as: 
$$\theta(av)=\kappa(a)\cdot c_v\cdot \ad_\varphi(v), \ a\in K, v\in V$$ for some $\kappa\in\Aut(K), c:V\to K, \varphi\in\NCV.$
We multiply $\theta$ by elements in the range of $\Xi$ until we obtain the trivial automorphism. 
Using this method we successively ``remove'' $\varphi$, $\kappa$ and $c_v$ using essentially that $\theta$ is spatial and the classification of cocycles valued in an abelian group.

Let us prove that $\Xi$ is a group morphism.
We have already checked that the maps $(\varphi,\beta)\mapsto A_{\varphi,\beta}$, $f\mapsto \ad(f)$ and  $\zeta\mapsto E_\zeta$ are injective group morphisms.
We are reduced to verify that $E_\zeta$ commutes with $\ad(f)$ and that 
$$A_{\varphi,\beta}\circ  E_\zeta \circ \ad(f)  \circ A_{\varphi,\beta}^{-1} = E_{\beta(\zeta)}\circ \ad(\ov\beta(f)^\varphi) \circ \ad(\beta(\zeta)^{\mu_\varphi})$$ for all $\zeta\in Z\Ga, f\in N_{\ov K}(G), \varphi\in\SNQ, \beta\in\Aut(\Ga).$
The first statement is proved by observing that given $\zeta\in Z\Ga$ we have that $\zeta^{p_v}$ commutes with any $f\in \ov K$ and thus for all $v\in V$ implying that $E_\zeta$ and $\ad(f)$ commute. 
Choose such a quadruple $(\zeta,f,\varphi,\beta)$ and an element $av\in G$ with $a\in K, v\in V.$
Observe that 
\begin{align*}
A_{\varphi,\beta}\circ  E_\zeta \circ \ad(f)  \circ A_{\varphi,\beta}^{-1} (av) & = A_{\varphi,\beta}\circ  E_\zeta \circ \ad(f) (\ov{\beta^{-1}}(a)\circ \varphi \cdot (\varphi^{-1} v \varphi))\\
& = A_{\varphi,\beta}\circ  E_\zeta (f\cdot \ov{\beta^{-1}}(a)\circ \varphi \cdot (\varphi^{-1} v \varphi)\cdot f^{-1})\\
& = A_{\varphi,\beta} (f\cdot  \ov{\beta^{-1}}(a)\circ \varphi \cdot \zeta^{p_{\varphi^{-1} v \varphi}}\cdot (\varphi^{-1} v \varphi) \cdot f^{-1})\\
& = \ov{\beta}(f)^\varphi \cdot  a \cdot \beta(\zeta)^{p^\varphi_{\varphi^{-1} v \varphi}}\cdot v \cdot ( \ov\beta( f)^\varphi)^{-1})\\
& = \ad(\ov{\beta}(f)^\varphi) (  a \cdot \beta(\zeta)^{p^\varphi_{\varphi^{-1} v \varphi}}\cdot v).
\end{align*}
On the other hand:
\begin{align*}
\Xi(\sigma(\varphi,\beta)(\zeta,f))(av) & = \Xi( ( \beta(\zeta) , \ov\beta(f)^\varphi, \beta(\zeta)^{\mu_\varphi}) (av)\\
& = \ad(\ov\beta(f)^\varphi) ( a \cdot \beta(\zeta)^{p_v + \mu_\varphi - \mu_\varphi^v} \cdot v).
\end{align*}

By Proposition \ref{prop:semidirect} we have the equality
$$p^\varphi_{\varphi^{-1} v \varphi} = p_v + \mu_\varphi - \mu_\varphi^v$$
implying that 
$$A_{\varphi,\beta}\circ  E_\zeta \circ \ad(f)  \circ A_{\varphi,\beta}^{-1} =\Xi(\sigma(\varphi,\beta)(\zeta,f)).$$
Therefore, $\Xi$ is a group morphism.

Let us show that the kernel of $\Xi$ is the normal subgroup $M$ described in the theorem.
Consider $(\zeta,f,\varphi,\beta)\in Q$ and assume that $\theta:=\Xi(\zeta,f,\varphi,\beta)$ is the trivial automorphism.
Consider $h\in\Ga, h\neq e$ and $h_x\in K$ the map supported in $\{x\}$ taking the value $h.$
Note that such a $h$ exists since $\Ga$ is nontrivial.
Observe that 
\begin{align*}
\theta(h_x) & = E_\zeta\circ\ad(f)\circ A_{\varphi,\beta}(h_x) = E_\zeta \circ \ad(f) \beta(h)_{\varphi(x)}\\
& = E_\zeta [f(\varphi(x)) \beta(h) f(\varphi(x))^{-1}]_{\varphi(x)}\\
& = [f(\varphi(x)) \beta(h) f(\varphi(x))^{-1}]_{\varphi(x)}.
\end{align*}
In particular, $\theta(h_x)$ has support $\{\varphi(x)\}$ but since $\theta(h_x)=h_x$ this latter support is equal to the support of $h_x$ which is $\{x\}$ implying that $\varphi=\id.$
We obtain that 
\begin{equation}\label{h} h = f(x) \beta(h) f(x)^{-1} \text{ for all } h\in \Ga, x\in\Q_2.\end{equation}
Consider $v\in V$ and observe that 
\begin{equation}\label{v}
v = \theta(v) =  \zeta^{p_v}\cdot f \cdot v \cdot f^{-1} = \zeta^{p_v}\cdot f (f^v)^{-1} \cdot v.\end{equation}
Therefore,
$$f(x) \zeta^{p_v(x)} = f(v^{-1}x) \text{ for all } v\in V, x\in \Q_2.$$
Fix $x\in\Q_2$ and choose $v\in V$ such that $v(x)=x$ and $v'(x)=2.$
We obtain that $$f(x) \zeta = f(x)$$ and thus $\zeta=e$.
Equation \eqref{v} becomes
$$f=f^v \text{ for all } v\in V.$$
Since $V\act \Q_2$ is transitive we deduce that $f:\Q_2\to \Ga$ is constant.
There exists $g\in \Ga$ such that $f(x)=g$ for all $x\in \Q_2$ and thus $f=\ov g.$
Using \eqref{h} we obtain that $g \beta(h) g^{-1} = h$ for all $h\in\Ga$ implying that $\beta=\ad(g^{-1}).$
Conversely, it is easy to see that $\Xi(e,\ov g, e,e) = \Xi(e,e,e,\ad(g))$ implying that $\ker(\Xi)=M.$

It remains to show that $\Xi$ is surjective.
Fix an automorphism $\theta\in \Aut(G).$

By Theorem \ref{theo:KinG}, we have that $\theta(K) = K.$
Moreover, Proposition \ref{prop:supportone} implies that there exists $\kappa\in \Aut(K), \varphi\in\SNQ$ and $c:V\to K$ such that 
$$\theta(av) = \kappa(a)\cdot c_v \cdot \ad_\varphi(v) \text{ for all } a\in K, v\in V.$$

Up to a composition with $A_{\varphi,e}=\Xi(e,e,\varphi,e)$ we can assume that $\varphi$ is the identity transformation and thus $$\theta(av) = \kappa(a)\cdot c_v \cdot v, \ a\in K, v\in V.$$

By Proposition \ref{prop:supportone} we have that $\supp(\kappa(a)) = \supp(a)$ for all $a\in K$ since $\varphi$ is trivial.
Moreover, $\kappa$ can be decomposed as a product of automorphisms:
$$\kappa = \prod_{x\in\Q_2} \kappa_x\in\prod_{Q_2}\Aut(\Ga)$$
such that $$\kappa(a)(x) = \kappa_x(a(x)) \text{ for all } a \in K, x\in\Q_2.$$

Now given $v\in V$ we have that 
$$\kappa(vav^{-1})(vx) = [c_v \cdot v \cdot \kappa(a)\cdot  v^{-1}\cdot c_v^{-1}](vx) = \ad(c_v(vx))[\kappa(a)(x)]= \ad(c_v(vx))[\kappa_x(a(x))].$$
This is equal to $\kappa_{vx}(a(x)).$
Therefore, 
\begin{equation}\label{eq:fund} 
\kappa_{vx} = \ad(c_v(vx))\circ \kappa_x \text{ for all } x\in\Q_2, v\in V.
\end{equation}
Fix $x=0$ in $\Q_2$ and write $\beta:=\kappa_{0}$.
Up to a composition with $A_{e,\beta}=\Xi(e,e,e,\beta)$ we can now assume that $\kappa_0=\id$.
This implies that $\kappa_{v0}=\ad(c_v(v0))$ for all $v\in V$. 
Therefore, since $V\act\Q_2$ is transitive we deduce that $\kappa_x$ is an inner automorphism of $\Ga$ for each $x\in\Q_2.$

We now take care of the cocyle part $c$.
For each $x\in \Q_2$ choose $v\in V$ such that $v0=x$ and put $h(x):=c_v(v0).$
Consider the map $h:\Q_2\to \Ga, x\mapsto h(x)$ and observe that $\kappa(a) = h a h^{-1}$ for all $a\in \oplus_{\Q_2}\Ga.$
Equation \ref{eq:fund} implies that:
$$h(vx) a(vx) h(vx)^{-1} = c_v(vx) h(x) a(x) h(x)^{-1} c_v(vx)^{-1} \text{ for all } a\in K, x\in \Q_2, v\in V.$$
Hence, $\ad(h) = \ad(c_v) \circ \ad(h^v)$ implying that 
$$c_v = h(h^v)^{-1} \mod \prod_{\Q_2}Z\Ga \text{ for all } v\in V.$$
Therefore, $c_v = d_v \cdot h(h^v)^{-1} $ where $d_v\in \prod_{\Q_2} Z\Ga$ for all $v\in V.$
Moreover, $d:v\mapsto d_v$ must be a cocycle that is valued in $\prod_{\Q_2}Z\Ga$.
Since $Z\Ga$ is abelian, Proposition \ref{prop:cocycle Abelian} implies that there exists a pair $(\zeta,f_0)$ with $\zeta\in Z\Ga, f_0\in \prod_{\Q_2} Z\Ga$ satisfying
$$d_v = s(\zeta)_v\cdot f_0(f_0^v)^{-1} \text{ for all } v\in V,$$
where $s(\zeta)_v(x)=\zeta^{\log_2(v'(v^{-1}x))},v\in V,x\in\Q_2$ is the slope cocycle defined in Proposition \ref{prop:cocycle Abelian}.
Therefore, $$c_v = \zeta^{p_v} \cdot f(f^v)^{-1}, \ v\in V$$
where $p_v=\log_2(v')^v - \nu + \nu^v$ is the map defined in Proposition \ref{prop:zetacocycle} and where $f:= h\cdot f_0\cdot \zeta^\nu.$

Up to a composition with $$E_\zeta:av\mapsto a \cdot \zeta^{p_v}\cdot v$$ we can now assume that $\theta$ is of the form
$$\theta(av) = \ad(f)(av) = faf^{-1} f(f^v)^{-1} v, a\in \oplus_{\Q_2}\Ga, v\in V$$
where $f\in \ov K$ and necessarily $f\in N_{\ov K}(G).$

We have proven that $\theta$ is a product of automorphisms of the form $\ad(f), E_\zeta$ and $A_{\varphi,\beta}$ with $f\in N_{\ov K}(G), \zeta\in Z\Ga, \varphi\in \NCV, \beta\in\Aut(\Ga)$ implying that the range of $\Xi$ is generating $\Aut(G)$ and thus $\Xi$ is surjective.
\end{proof}


\begin{thebibliography}{BCM{\etalchar{+}}19}

\bibitem[Aie20]{Aiello20}
V.~Aiello. 
\newblock On the {A}lexander {T}heorem for the oriented {T}hompson group $\Vec{F}$.
\newblock{\em Algebr. Geom. Topol.}, 20:429--438, 2020.

\bibitem[ABC19]{ABC19}
V.~Aiello, A.~Brothier, and R.~Conti.
\newblock Jones representations of {T}hompson's group {F} arising from
  {T}emperley-{L}ieb-{J}ones algebras.
\newblock {\em to appear in Int. Math. Res. Not.}, 2019.

\bibitem[BCM{\etalchar{+}}19]{BCMNO19}
C.~Bleak, P.~Cameron, Y.~Maissel, A.~Navas, and F.~Olukoya.
\newblock The further chameleon groups of {R}ichard {T}hompson and {G}raham
  {H}igman: automorphisms via dynamics for the {H}igman-{T}hompson groups
  ${G}_{n,r}$.
\newblock Preprint 2019.

\bibitem[Bel04]{Belk04}
J.~Belk.
\newblock {\em Thompson's group {F}}.
\newblock PhD thesis, Cornell University, 2004.

\bibitem[BZFG{\etalchar{+}}18]{BZFGHM18}
R.~Berns-Zieve, D.~Fry, J.~Gillings, H.~Hoganson, and H.~Mathews.
\newblock Groups with context-free co-word problem and embeddings into
  Thompson's group $V$.
\newblock {\em In N. Broaddus, M. Davis, J. Lafont, \& I. Ortiz (Eds.),
  Topological Methods in Group Theory (London Mathematical Society Lecture Note
  Series). Cambridge: Cambridge University Press.}, pages 19--37, 2018.
\newblock doi:10.1017/9781108526203.003.

\bibitem[Bis17]{Bischoff17}
M.~Bischoff.
\newblock The relation between subfactors arising from conformal nets and the
  realization of quantum doubles.
\newblock {\em Proc. Centre Math. Appl.}, 46:15--24, 2017.

\bibitem[BL10]{Bleak-Lanoue10}
C.~Bleak and D.~Lanoue.
\newblock A family of non-isomorphisms results.
\newblock {\em Geom. Dedicata}, 146:21--26, 2010.

\bibitem[BMN16]{BleakMatucciNeunhoffer16}
C.~Bleak, F.~Matucci, and M.~Neunh\"offer.
\newblock Embeddings into {T}hompson's group $v$ and $co\mathcal{CF}$ groups.
\newblock {\em J. London Math. Soc.}, 2016.

\bibitem[Bor94]{Bodnarchuk94}
Y.~Bodnarchuk.
\newblock On the isomorphism of wreath products of groups.
\newblock{\em Ukr. Math. J.}, 46(6):725--734, 1994. 

\bibitem[Bri04]{Brin04-nV}
M.~Brin.
\newblock Higher dimensional {T}hompson groups.
\newblock{\em Geom. Dedicata}, 108:163--192, 2004.

\bibitem[Bri07]{Brin07-BraidedThompson}
M.~Brin.
\newblock The algebra of stand splitting. {I}. {A} braided version of
  {T}hompson's group {V}.
\newblock {\em J. Group Theory}, 10(6):757--788, 2007.

\bibitem[BJ19a]{Brot-Jones18-1}
A.~Brothier and V.F.R. Jones.
\newblock On the {H}aagerup and {K}azhdan property of {R}. {T}hompson's groups.
\newblock {\em J. Group Theory}, 22(5):795--807, 2019.

\bibitem[BJ19b]{Brot-Jones18-2}
A.~Brothier and V.F.R. Jones.
\newblock Pythagorean representations of {T}homspon's groups.
\newblock {\em J. Funct. Anal.}, 277:2442--2469, 2019.

\bibitem[Bro19a]{Brothier19WP}
A.~Brothier.
\newblock Haagerup property for wreath products constructed with {T}hompson's
  groups.
\newblock {\em Preprint, arXiv:1906.03789}, 2019.

\bibitem[Bro19b]{Brothier-19-survey}
A.~Brothier.
\newblock On {J}ones' connections between subfactors, conformal field theory,
  {T}hompson's groups and knots.
\newblock {\em to appear in Celebratio Mathematica}.

\bibitem[Bro20]{Brothier20-2}
A.~Brothier.
\newblock Classification of {T}hompson related groups arising from {J}ones technology {II}.
\newblock {\em Preprint, arXiv:2011.13124}, 2020.

\bibitem[BS19]{Brot-Stottmeister-Phys}
A.~Brothier and A.~Stottmeister.
\newblock Canonical quantization of 1+1-dimensional yang-mills theory: an
  operator algebraic approach.
\newblock {\em Preprint, arXiv:1907.05549}, 2019.

\bibitem[BS20]{Brot-Stottmeister-M19}
A.~Brothier and A.~Stottmeister.
\newblock Operator-algebraic construction of gauge theories and {J}ones'
  actions of {T}hompson's groups.
\newblock {\em Comm. Math. Phys.}, 376:841--891, 2020.

\bibitem[Bro87]{Brown87}
K.~Brown
\newblock Finiteness properties of groups.
\newblock {\em J. Pure. App. Algebra}, 44:45--75, 1987.

\bibitem[CFP96]{Cannon-Floyd-Parry96}
J.W. Cannon, W.J. Floyd, and W.R. Parry.
\newblock Introductory notes on {R}ichard {T}hompson's groups.
\newblock {\em Enseign. Math.}, 42:215--256, 1996.

\bibitem[Cor06]{Cornulier06}
Y.~Cornulier.
\newblock Finitely presented wreath products and double coset decompositions.
\newblock {\em Geom. Dedicata}, 122(1):89--108, 2006.

\bibitem[Deh06]{Dehornoy06}
P.~Dehornoy.
\newblock The group of parenthesized braids.
\newblock {\em Advances Math.}, 205:354--409, 2006.

\bibitem[EK92]{Evans_Kawahigashi_92_sf_cft}
D.~Evans and Y.~Kawahigashi.
\newblock Subfactors and conformal field theory.
\newblock {\em Math. Phys. Stud.}, 16, 1992.

\bibitem[Far03]{Farley03-H}
D.~Farley.
\newblock Proper isometric actions of {T}hompson's groups on {H}ilbert space.
\newblock {\em Int. Math. Res. Not.}, 45:2409--2414, 2003.

\bibitem[Far05]{Farley05}
D.~Farley.
\newblock Actions of picture groups on {CAT(0)} cubical complexes. 
\newblock{\em Geom. Dedicata}, 110:221-242, 2005.

\bibitem[GS87]{Ghys-Sergiescu87}
E.~Ghys and V.~Sergiescu.
\newblock Sur un groupe remarquable de diffeomorphismes du cercle.
\newblock {\em Comment. Math. Helvetici}, 62:185--239, 1987.

\bibitem[GS00]{GNS00}
Nekrashevych~V. Grigorchuk, R. and V.~Sushchanskii.
\newblock Automata, dynamical systems, and groups.
\newblock {\em Proc. Steklov Inst. Math.}, 231:128--203, 2000.

\bibitem[GZ67]{GabrielZisman67}
P.~Gabriel and M.~Zisman.
\newblock {\em Calculus of fractions and homotopy theory}.
\newblock Springer-Verlag, 1967.

\bibitem[GS17]{Golan-Sapir17}
G.~Golan and M.~Sapir.
\newblock On {J}ones' subgroup of {R}. {T}hompson group {F}.
\newblock {\em J. of Algebra}, 470:122--159, 2017.

\bibitem[Gro88]{Gross88}
F.~Gross.
\newblock Automorphisms of permutational wreath products.
\newblock {\em J. Algebra}, 117:472--493, 1988.

\bibitem[Gro92]{Gross92}
F.~Gross.
\newblock On the uniqueness of wreath products.
\newblock{\em J. Algebra}, 147:147--175, 1992.

\bibitem[GS97]{Guba-Sapir97}
V.~Guba and M.~Sapir.
\newblock Diagram groups. 
\newblock{\em Mem. Amer. Math. Soc.}, 130,  1997.

\bibitem[Has78]{Hassanabi78}
A.~Hassanabi.
\newblock Automorphisms of permutational wreath products.
\newblock{\em J. Austral. Math. Soc. Ser. A}, 26:198-208, 1978.

\bibitem[Hig74]{Higman74}
G.~Higman.
\newblock Finitely presented infinite simples groups.
\newblock {\em Notes on Pure Mathematics 8}, Australian National University, Canberra, 1974.

\bibitem[Hou63]{Houghton63}
C.~Houghton.
\newblock On the automorphism groups of certain wreath products. 
\newblock{\em Publ. Math. Debrecen}, 9:307-313, 1963.

\bibitem[Hug09]{Hughes09}
B.~Hughes.
\newblock Local similarities and the {H}aagerup property. 
\newblock{\em Groups Geom. Dyn.}, 3:299-315, 2009.

\bibitem[Ish17]{Ishida17}
T.~Ishida.
\newblock Ordering of {W}itzel-{Z}aremsky-{T}hompson groups.
\newblock {\em Comm. Algebra}, 46:3806--3809, 2017.

\bibitem[JMS14]{Jones-Morrison-Synder14}
V.F.R. Jones, S.~Morrison, and N.~Snyder.
\newblock The classification of subfactors of index at most 5.
\newblock {\em Bull. Amer. Math. Soc.}, 51:277--327, 2014.

\bibitem[Jon85]{Jones_polynome_vna}
V.F.R. Jones.
\newblock A polynomial invariant for knots via von {N}eumann algebras.
\newblock {\em Bull. Amer. Math. Soc.}, 12:103--112, 1985.

\bibitem[Jon17]{Jones17-Thompson}
V.F.R. Jones.
\newblock Some unitary representations of {T}ompson's groups {F} and {T}.
\newblock {\em J. Comb. Algebra}, 1(1):1--44, 2017.

\bibitem[Jon18a]{Jones16-Thompson}
V.F.R. Jones.
\newblock A no-go theorem for the continuum limit of a periodic quantum spin
  chain.
\newblock {\em Comm. Math. Phys.}, 357(1):295--317, 2018.

\bibitem[Jon18b]{Jones18-Hamiltonian}
V.F.R. Jones.
\newblock Scale invariant transfer matrices and {H}amiltonians.
\newblock {\em J. Phys. A: Math. Theor.}, 51:104001, 2018.

\bibitem[Jon19a]{Jones19Irred}
V.F.R. Jones.
\newblock Irreducibility of the wysiwyg representations of {T}hompson's groups.
\newblock {\em Preprint, arXiv:1906.09619}, 2019.

\bibitem[Jon19b]{Jones19-thomp-knot}
V.F.R. Jones.
\newblock On the construction of knots and links from {T}hompson's groups.
\newblock {\em In: Adams C. et al. (eds) Knots, Low-Dimensional Topology and
  Applications}, 284, 2019.
  
\bibitem[Leh08]{Lehnert08-thesis}
J.~Lehnert.
\newblock Gruppen von quasi-automorphismen.
\newblock {\em Goethe Universitat, Frankfurt}, 2008. 

\bibitem[LR95]{Longo-Rehren95}
R.~Longo and K.~Rehren.
\newblock Nets of subfactors. 
\newblock {\em Rev. Math. Phys.}, 7:567-597, 1995.

\bibitem[Mal53]{Maltsev53}
A.~Mal'tsev.
\newblock Nilpotent semigroups.
\newblock {\em Uchen. Zap. Ivanovsk. Ped. Inst.}, 4:107--111, 1953.

\bibitem[Nav02]{Navas02}
A.~Navas.
\newblock Actions de groupes de {K}azhdan sur le cercle.
\newblock {\em Ann. Sci. Ec. Norm. Super.}, 35(4):749--758, 2002.

\bibitem[Nek04]{Nekrashevych04}
V.~Nekrashevych.
\newblock Cuntz-{P}imsner algebras of group actions.
\newblock {\em J. Operator Theory}, 52:223--249, 2004.

\bibitem[Nek05]{Nekrashevych05}
V.~Nekrashevych.
\newblock Self-similar groups.
\newblock {\em Math. Survey Monogr., Amer. Math. Soc.}, 117, 2005.

\bibitem[Neu64]{Neumann64}
P.~Neumann.
\newblock On the structure of standard wreath products of groups.
\newblock {\em Math. Zeitschr.}, 84:343--373, 1964.

\bibitem[OS19]{Osborne-Stiegemann19}
T.~Osborne and D.~Stiegemann.
\newblock Quantum fields for unitary representations of {T}hompson's group {F}
  and {T}.
\newblock {\em Preprint, arXiv:1903.00318}, 2019.

\bibitem[Rez01]{Reznikoff01}
A.~Reznikoff.
\newblock Analytic topology.
\newblock {\em Progr. Math.}, 1:519--532, 2001.

\bibitem[Rov99]{Rover99}
C.~Rover.
\newblock Constructing finitely presented simple groups that contain {G}rigorchuk groups. 
\newblock{ \em J. Alg.}, 220:284-313, 1999

\bibitem[Rub89]{Rubin89}
M.~Rubin.
\newblock On the reconstruction of topological spaces from their groups of homeomorphisms. 
\newblock {\em Trans. Amer. Math. Soc.}, 312(2):487--538, 1989.

\bibitem[Rub96]{Rubin96}
M.~Rubin.
\newblock Locally moving groups and reconstruction problems.
\newblock {\em Math. Appl.}, 354:121--157, 1996.

\bibitem[Sti19]{Stiegemann19-thesis}
D.~Stiegemann.
\newblock Thompson field theory.
\newblock {\em PhD thesis, arXiv:1907.08442}, 2019.

\bibitem[Tan16]{Tanushevski16}
S.~Tanushevski.
\newblock A new class of generalized thompson's groups and their normal
  subgroups.
\newblock {\em Commun. Algebra}, 44:4378--4410, 2016.

\bibitem[Tan17]{Tanushevski17}
S.~Tanushevski.
\newblock Presentations for a class of generalized thompson's groups.
\newblock {\em Commun. Algebra}, 45(5):2074--2090, 2017.

\bibitem[WZ18]{Witzel-Zaremsky18}
S.~Witzel and M.~Zaremsky.
\newblock Thompson groups for systems of groups, and their finiteness
  properties.
\newblock {\em Groups Geom. Dyn.}, 12(1):289--358, 2018.

\bibitem[Xu18]{Xu18-CFT}
F.~Xu.
\newblock Examples of subfactors from conformal field theory.
\newblock {\em Comm. Math. Phys.}, 357:61--75, 2018.

\bibitem[Zar18]{Zaremsky18-clone}
M.~Zaremsky.
\newblock A user's guide to cloning systems.
\newblock {\em Topology Proc.}, 52:13--33, 2018.

\end{thebibliography}

\newcommand{\etalchar}[1]{$^{#1}$}

%
%

\end{document}